%% file: regular-sequences.tex
\documentclass[a4paper, reqno]{amsart}
\pdfoutput=1
\newif\ifprivate
\privatefalse
\newif\ifdetails
\detailsfalse

\usepackage[utf8]{inputenc}
\usepackage[T1]{fontenc}
\usepackage{lmodern}
\usepackage[draft=false,kerning=true]{microtype}

\usepackage{enumerate}
\usepackage{mathtools}
\usepackage{multicol}
\usepackage{textcomp}
\usepackage{tikz-cd}
\usetikzlibrary{backgrounds,scopes}
\usetikzlibrary{automata}

\usepackage{hyperref}
\usepackage{todonotes}

\ifprivate
\usepackage[mark]{gitinfo2}
\renewcommand{\gitMark}{\jobname\,\textbullet{}\,\gitFirstTagDescribe\,\textbullet{}\,\gitAuthorName,\,\gitAuthorIsoDate}
\newcommand{\TODO}[1]%
{\par\fbox{\begin{minipage}{0.9\linewidth}\textbf{TODO:} #1\end{minipage}}\par}
\fi

\newcommand{\proofparagraph}[1]{\medskip\par\noindent{\itshape #1.} }

\DeclarePairedDelimiter{\abs}{\lvert}{\rvert}

\newcommand{\bfx}{\mathbf{x}}

\newcommand{\bibinom}[3]{\binom{#1}{#2, #3}}
\newcommand{\C}{\mathbb{C}}
\newcommand{\calC}{\mathcal{C}}
\newcommand{\calD}{\mathcal{D}}
\newcommand{\calF}{\mathcal{F}}
\newcommand{\calG}{\mathcal{G}}
\newcommand{\calH}{\mathcal{H}}
\newcommand{\calI}{\mathcal{I}}
\newcommand{\calJ}{\mathcal{J}}
\newcommand{\calK}{\mathcal{K}}
\newcommand{\calL}{\mathcal{L}}
\newcommand{\calM}{\mathcal{M}}
\newcommand{\calS}{\mathcal{S}}
\newcommand{\calT}{\mathcal{T}}
\newcommand{\calV}{\mathcal{V}}
\newcommand{\calX}{\mathcal{X}}
\newcommand{\calY}{\mathcal{Y}}
\newcommand{\calZ}{\mathcal{Z}}
\DeclarePairedDelimiter{\ceil}{\lceil}{\rceil}
\newcommand{\dd}{\mathrm{d}}
\DeclareMathOperator{\diag}{diag}
\newcommand{\DLMF}[2]{\cite[\href{http://dlmf.nist.gov/#1.E#2}{#1.#2}]{NIST:DLMF:v1.0.16}}

\DeclarePairedDelimiterXPP{\f}[2]{\foperator{#1}}(){}{#2}
\DeclarePairedDelimiterXPP{\fexp}[1]{\exp}(){}{#1}
\DeclarePairedDelimiter{\floor}{\lfloor}{\rfloor}
\newcommand{\foperator}[1]{\mathop{{#1}\empty{}}}
\DeclarePairedDelimiter{\fractional}{\{}{\}}
\DeclarePairedDelimiterXPP{\inftynorm}[1]{}{\lVert}{\rVert}{_\infty}{#1}
\newcommand{\itemref}[1]{\eqref{#1}}
\DeclarePairedDelimiter{\iverson}{[}{]}
\newcommand{\kerneldim}{D}
\renewcommand{\MR}[1]{}
\newcommand{\oeis}[1]{\cite[\href{https://oeis.org/#1}{#1}]{OEIS:2018}}
\DeclarePairedDelimiter{\norm}{\lVert}{\rVert}
\DeclarePairedDelimiterXPP{\Oh}[1]{\foperator{O}}(){}{#1}
\DeclarePairedDelimiterXPP{\oh}[1]{\foperator{o}}(){}{#1}

\newcommand{\R}{\mathbb{R}}
\newcommand{\repr}{\mathsf{reprq}}
\DeclarePairedDelimiterXPP{\Res}[2]{\operatorname{Res}}(){}{#1,\, #2}
\DeclarePairedDelimiter{\set}{\{}{\}}
\DeclarePairedDelimiterX{\setm}[2]{\{}{\}}{#1 \colon \mathopen{}#2}
\newcommand{\sigmaabs}{\sigma_{\mathrm{abs}}}
\newcommand{\tildea}{\widetilde{a}}
\newcommand{\tildeA}{\widetilde{A}}
\newcommand{\tildecalC}{\widetilde{\calC}}
\newcommand{\tildeM}{\widetilde{M}}
\newcommand{\tildev}{\widetilde{v}}
\newcommand{\tpmod}[1]{\ensuremath{\undisp{\pmod{#1}}}}
\newcommand{\trinom}[4]{\binom{#1}{#2, #3, #4}}
\makeatletter\newcommand\undisp[1]{\bgroup\@displayfalse #1\egroup}\makeatother
\newcommand{\val}{\mathsf{lval}}
\newcommand{\Z}{\mathbb{Z}}

\newtheorem{theorem}{Theorem}
\newtheorem{corollary}[theorem]{Corollary}

\newtheorem{lemma}{Lemma}[section]
\newtheorem{proposition}[lemma]{Proposition}

\theoremstyle{remark}
\newtheorem{example}[lemma]{Example}
\newtheorem{remark}[lemma]{Remark}

\numberwithin{equation}{section}
\numberwithin{figure}{section}
\numberwithin{table}{section}

\begin{document}
\title{Asymptotic Analysis of Regular Sequences}
%title aofa: Analysis of Summatory Functions of Regular Sequences: Transducer and Pascal's Rhombus
%title analco: Esthetic Numbers and Lifting Restrictions on the Analysis of Summatory Functions of Regular Sequences

\author[C.~Heuberger]{Clemens Heuberger}
\address{Clemens Heuberger,
  Institut f\"ur Mathematik, Alpen-Adria-Universit\"at Klagenfurt,
  Universit\"atsstra\ss e 65--67, 9020 Klagenfurt am W\"orthersee, Austria}
\email{\href{mailto:clemens.heuberger@aau.at}{clemens.heuberger@aau.at}}

\author[D.~Krenn]{Daniel Krenn}
\address{Daniel Krenn,
  Institut f\"ur Mathematik, Alpen-Adria-Universit\"at Klagenfurt,
  Universit\"atsstra\ss e 65--67, 9020 Klagenfurt am W\"orthersee, Austria}
\email{\href{mailto:math@danielkrenn.at}{math@danielkrenn.at} \textit{or}
  \href{mailto:daniel.krenn@aau.at}{daniel.krenn@aau.at}}

\thanks{C.~Heuberger and D.~Krenn are supported by the
   Austrian Science Fund (FWF): P\,28466-N35.}

\keywords{
  Regular sequence,
  Mellin--Perron summation,
  summatory function,
  Tauberian theorem,
  transducer,
  esthetic numbers,
  Pascal's rhombus%
}

\subjclass[2010]{%
05A16; %Combinatorics: Enumerative combinatorics: asymptotic enumeration
11A63, %Radix representation; digital problems
68Q45, %Formal languages and automata
68R05%  %discrete mathematics in relation to computer science: combinatorics
}

\begin{abstract}
  In this article, $q$-regular sequences in the sense of Allouche and
  Shallit are analysed asymptotically. It is shown that the summatory
  function of a regular sequence can asymptotically be decomposed as a finite sum
  of  periodic fluctuations
  multiplied by a scaling factor. Each of these terms corresponds to an
  eigenvalue of the sum of matrices of a linear representation of
  the sequence; only the eigenvalues of absolute value larger than the
  joint spectral radius of the matrices contribute terms which grow faster than the error term.

  The paper has a particular focus on the Fourier
  coefficients of the periodic fluctuations: They are
  expressed as residues
  of the corresponding Dirichlet generating function.
  This makes it possible to compute them in an efficient way.
  The asymptotic analysis deals with Mellin--Perron summations and
  uses two arguments to overcome convergence issues, namely Hölder
  regularity of the fluctuations together with a pseudo-Tauberian
  argument.

  Apart from the very general result, three examples are discussed in more detail:
  \begin{itemize}
  \item sequences defined as the sum of outputs written by a
    transducer when reading a $q$-ary expansion of the input;
  \item the amount of esthetic numbers in the first~$N$ natural numbers; and
  \item the number of odd entries in the rows of Pascal's rhombus.
  \end{itemize}
  For these examples, very precise asymptotic formul\ae{} are presented.
  In the latter two examples, prior to this analysis only rough
  estimates were known.
\end{abstract}

\thanks{The extended
  abstract~\cite{Heuberger-Krenn-Prodinger:2018:pascal-rhombus} (with
  appendices containing proofs available
  as~\href{https://arxiv.org/abs/1802.03266}{arXiv:1802.03266})
  imposes a restriction on the asymptotic growth. The extended
  abstract~\cite{Heuberger-Krenn:2018:esthetic} (with appendices
  containing proofs available
  as~\href{https://arxiv.org/abs/1808.00842}{arXiv:1808.00842}) lifts
  this restriction by completely getting rid of the corresponding
  technical condition.
  This article now contains the full (majorly restructured) proof
  covering all cases. It is shorter and simpler. We now use a
  generating functions approach which also gives additional
  insights. For example, the cancellations in the proof in
  \cite{Heuberger-Krenn-Prodinger:2018:pascal-rhombus} seem to be a
  kind of magic at that point, but with the new approach, it is now
  clear and no surprise anymore that they have to appear.
  Besides, the examples investigated
  in~\cite{Heuberger-Krenn-Prodinger:2018:pascal-rhombus,
    Heuberger-Krenn:2018:esthetic} are now presented with full
  details.
  A new part on computational aspects of the computation of
  Fourier coefficients is added.
  Reading strategies for various interests are now outlined in
  Part~\ref{part:overview} so that readers
  find their ways through this article.}

\thanks{The authors thank Helmut Prodinger for his early involvement
  and Sara Kropf for her comments on an early
  version of this paper.}

\thanks{This paper is published under a
  Creative Commons Attribution 4.0 International License
  \raisebox{-0.5ex}{\includegraphics[height=1em]{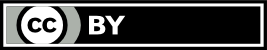}}
  (\url{http://creativecommons.org/licenses/by/4.0/}).
  Copyright is held by the authors.}
  
\maketitle

\newpage
\renewcommand{\contentsname}{}
\newcommand{\nocontentsline}[3]{}
\bgroup\let\addcontentsline=\nocontentsline{}{}{}
\section*{Contents}
\egroup
\vspace*{2ex}
\begin{multicols}{2}
{\vspace*{-12ex}\footnotesize\tableofcontents}
\end{multicols}
\newpage

\part{Introduction}\label{part:overview}

\section{Synopsis: The Objects of Interest and the Result}
In this paper, we study the asymptotic behaviour of the summatory function of a
$q$-regular sequence $x(n)$. At this
point, we give a short overview of the notion of $q$-regular sequences%
\footnote{
  In the standard literature
  \cite{Allouche-Shallit:1992:regular-sequences, Allouche-Shallit:2003:autom}
  these sequences
  are called $k$-regular sequences (instead of $q$-regular sequences).}
and our main result.

One characterisation of a $q$-regular sequence is as follows: The sequence
$x(n)$ is said to be $q$-regular if there are square matrices $A_0$, \ldots, $A_{q-1}$
and a vector-valued sequence $v(n)$ such that
\begin{equation*}
  v(qn+r)=A_r v(n)\qquad \text{for $0\le r<q$ and $n\ge 0$}
\end{equation*}
and such that $x(n)$ is the first component of $v(n)$.

Regular sequences are
intimately related to the $q$-ary expansion of their arguments.
They have been introduced by Allouche and Shallit
\cite{Allouche-Shallit:1992:regular-sequences}; see also
\cite[Chapter~16]{Allouche-Shallit:2003:autom}. Many special
cases have been investigated in the literature; this is also due to their
relation to divide-and-conquer algorithms. Moreover, every $q$-automatic
sequence---those sequences are defined by finite automata---is $q$-regular as well. Take also a look at the
book~\cite{Allouche-Shallit:2003:autom} for many examples.

Our main result is, roughly speaking, that the summatory function of a
$q$-regular sequence $x(n)$ has the asymptotic form
\begin{equation}\label{eq:synopsis-shape-main}
  \sum_{n<N}x(n) = \sum_{j=1}^J N^{\log_q \lambda_j} \frac{(\log
    N)^{k_j}}{k_j!} \Phi_{k_j}(\fractional{\log_q N}) + O(N^{\log_q R}) 
\end{equation}
as $N\to\infty$ for a suitable positive integer~$J$,
suitable constants~$\lambda_j\in \C$,
suitable non-negative integers~$k_j$,
a suitable~$R$ and $1$-periodic continuous functions~$\Phi_{k_j}$.
The $\lambda_j$ will turn out to be eigenvalues of
$C \coloneqq A_0+\cdots+A_{q-1}$, the $k_j$ be related to the multiplicities of these
eigenvalues and the constant $R$ will be a bound for
the joint spectral radius of the matrices $A_0$, \ldots, $A_{q-1}$.

While~\eqref{eq:synopsis-shape-main} gives the shape of the asymptotic
form, gathering as much information as possible on the periodic
fluctuations~$\Phi_{k_j}$ is required to have a full picture. To this aim,
we will give a description of the Fourier coefficients of the $\Phi_{k_j}$
which allows to compute them algorithmically and therefore to describe
these periodic fluctuations with high precision.
In particular, this allows to detect non-vanishing fluctuations. Code%
\footnote{The code accompanying this article can be found at
  \url{https://gitlab.com/dakrenn/regular-sequence-fluctuations}\,.
  It is meant to be used with the open source mathematics software SageMath~\cite{SageMath:2018:8.3}.}
is provided to compute the Fourier coefficients.

We close this introductory section by noting that the normalized sum
$\frac{1}{N} \sum_{n<N}x(n)$ enlightens us about the expectation of a
random element of the sequence~$x(n)$ with respect to uniform distribution
on the non-negative integers smaller than a certain~$N$.

\section{How to Read this Paper}

This is a long (and perhaps sometimes technical) paper and not all readers
might find the time to read it from the very beginning to the very end. We
therefore outline reading strategies for various interests.

For the reader who wants to \emph{apply our results to a particular problem}:
Read Section~\ref{section:introduction:regular-sequences}
on the definition of $q$-regular sequences and Section~\ref{section:introduction:main-result}
containing the main result in a condensed version which should cover
most applications. These two sections also have a simple,
illustrative and well-known running example.
If it turns out that the refined versions of the results are needed, follow
the upcoming paragraph below.

For the reader who still wants to \emph{apply our results to a particular
problem} but finds the \emph{condensed version insufficient},
turn to the overview of the results
(Section~\ref{section:results:overview}) and then continue
with Section~\ref{section:results} where the notations and results
are stated in full generality.
Formulating them will need quite a number of
definitions provided in Section~\ref{sec:definitions-notations}. In order to
cut straight to the results themselves, we will refrain from motivations and
comments on these definitions and
postpone those comments to Section~\ref{sec:motivation-definitions}.

For the reader who wants to \emph{determine the asympotics of a regular sequence}
instead of determining the asymptotics of the summatory function of the 
regular sequence, advice is given in
Section~\ref{section:asy-regular-sequences:non-summatory}.

For the reader who wants to read more about \emph{showcase applications} of our
method yielding \emph{new asymptotic results}, additionally to Section~\ref{section:user-friendly-main-and-example} read
Section~\ref{sec:overview-examples} where an overview of the examples in
this paper is given and then Part~\ref{part:examples}
where these examples are discussed in
detail. For many more examples to which the methods can be applied, read the
original papers~\cite{Allouche-Shallit:1992:regular-sequences,
  Allouche-Shallit:2003:regular-sequences-2} and the book by Allouche and
Shallit~\cite{Allouche-Shallit:2003:autom}
which contain many examples of $q$-regular sequences. 

For the reader who wants to \emph{compute the Fourier
coefficients} for a particular application, use the provided code. Read
Part~\ref{part:numerical} for more details, in particular, see
Section~\ref{section:non-vanishing} for some comments on how to decide
whether fluctuations are constant or even vanish.

Moreover, for the reader who is interested in
the background on the \emph{algorithmic aspects} and details of the
implementation of the actual computation, we also refer to
Part~\ref{part:numerical}; this part will also be useful for the reader
who wants to review the code written for SageMath.

For the reader who is interested in the \emph{history of the problem}, we refer to
Section~\ref{introduction:relation-to-previous-work}.

For the reader who wants to see a \emph{heuristic argument why everything works out},
there is Section~\ref{sec:heuristic} where it is shown that once one does not care about
convergence issues, the Mellin--Perron summation formula of order zero explains
the result.

For the reader who wants to understand the \emph{idea of the proof}, there is
Section~\ref{section:high-level-overview-proof}
with a high level overview of the proof how the above mentioned
convergence issues with the
Mellin--Perron summation formula can be overcome by a pseudo-Tauberian argument.

For the reader who wants to \emph{overcome convergence problems with the Mellin--Perron
summation formula} in other contexts involving periodic fluctuations, we note that
the pseudo-Tauberian argument (Proposition~\ref{proposition:pseudo-Tauber})
is completely independent of our application to
$q$-regular sequences; the only prerequisite is the knowledge on the existence
of the fluctuation and sufficient knowledge on analyticity and growth
of the Dirichlet generating function. As a consequence,
Theorem~\ref{theorem:use-Mellin--Perron} has been
formulated as an independent result and provisions have been made for several
applications of the pseudo-Tauberian argument.

Finally, for the reader who wants to \emph{fully understand the proof}: We have no other advice
than reading
the whole introduction,
the whole Section~\ref{section:results} on results and
the whole Part~\ref{part:proofs} on the proofs starting 
with a very short
Section~\ref{additional-notation} where a few notations used throughout the proofs
are fixed.

\section{User-friendly Main Result and a First Example Application}
\label{section:user-friendly-main-and-example}

\subsection{\texorpdfstring{$q$}{q}-Regular Sequences}\label{section:introduction:regular-sequences}

We start by giving a definition of $q$-regular sequences; see Allouche and
Shallit~\cite{Allouche-Shallit:1992:regular-sequences}. Let $q\ge 2$ be a fixed
integer and $x$ be a sequence on $\Z_{\ge 0}$.\footnote{We use a functional
  notation for sequences,
  i.e., a sequence $x$ on $\Z_{\ge 0}$ is seen as function $x\colon\Z_{\ge 0} \to \C$.}
Then $x$ is said to
be \emph{$(\C, q)$-regular} (briefly: \emph{$q$-regular} or simply \emph{regular}) if the $\C$-vector space
generated by its \emph{$q$-kernel}
\begin{equation*}
  \setm[\big]{x \circ (n \mapsto q^j n+r)}%
  {\text{integers $j\ge 0$, $0\le r<q^j$}}
\end{equation*}
has finite dimension.
In other words, $x$ is $q$-regular if
there are an integer $\kerneldim$ and sequences $x_1$, \dots, $x_\kerneldim$
such that for every $j\ge 0$ and $0\le r<q^j$
there exist complex numbers $c_1$, \ldots, $c_\kerneldim$ with
\begin{equation*}
  x(q^j n+r) = c_1 x_1(n) + \dotsb + c_\kerneldim x_\kerneldim(n)\qquad{\text{for all $n\ge 0$.}}
\end{equation*}

By Allouche and Shallit~\cite[Theorem~2.2]{Allouche-Shallit:1992:regular-sequences},
the sequence~$x$ is $q$-regular if and only if there exists a vector-valued
sequence~$v$ whose first component coincides with
$x$ and there exist square matrices $A_0$, \ldots, $A_{q-1}\in\C^{d\times d}$ such that
\begin{equation}\label{eq:linear-representation}
  v(qn+r) = A_r v(n)\qquad\text{for $0\le r<q$ and $n\ge 0$.}
\end{equation}
This is called a \emph{$q$-linear representation} of the sequence~$x$.

The best-known example for a $2$-regular function is the binary sum-of-digits
function.

\begin{example}\label{example:binary-sum-of-digits}
  For $n\ge 0$, let $x(n)=s(n)$ be the binary sum-of-digits of $n$. We clearly
  have
  \begin{equation}\label{eq:recursion-binary-sum-of-digits}
    \begin{aligned}
      x(2n)&=x(n),\\
      x(2n+1)&=x(n)+1
    \end{aligned}
  \end{equation}
  for $n\ge 0$.
  Indeed, we have
  \begin{equation*}
    x(2^j n+ r) = x(n) + x(r)\cdot 1
  \end{equation*}
  for integers $j\ge 0$, $0\le r <2^j$ and $n\ge 0$; i.e., the complex vector space
  generated by the $2$-kernel is generated by $x$ and the
  constant sequence $n \mapsto 1$.

  Alternatively, we set $v=(x, n \mapsto 1)^\top$ and have
  \begin{align*}
    v(2n)&=
    \begin{pmatrix}
      x(n)\\1
    \end{pmatrix}=
    \begin{pmatrix}
      1&0\\
      0&1
    \end{pmatrix}v(n),\\
    v(2n+1)&=
             \begin{pmatrix}
               x(n)+1\\
               1
             \end{pmatrix}=
    \begin{pmatrix}
      1 & 1\\
      0 & 1
    \end{pmatrix}v(n)
  \end{align*}
  for $n\ge 0$. Thus \eqref{eq:linear-representation} holds with
  \begin{equation*}
    A_0 =
    \begin{pmatrix}
      1&0\\
      0&1
    \end{pmatrix},\qquad
    A_1 =
    \begin{pmatrix}
      1&1\\
      0&1
    \end{pmatrix}.
  \end{equation*}
\end{example}

At this point, we note that a linear
representation~\eqref{eq:linear-representation} immediately leads to an
explicit expression for $x(n)$ by induction.

\begin{remark}\label{remark:regular-sequence-as-a-matrix-product}
  Let $r_{\ell-1}\ldots r_0$ be the $q$-ary digit
  expansion\footnote{
    Whenever we write that $r_{\ell-1}\ldots r_0$ is the  $q$-ary digit
    expansion of $n$, we
    mean that $r_j\in\set{0,\ldots, q-1}$ for $0\le j<\ell$, $r_{\ell-1}\neq 0$ and
    $n=\sum_{0 \le j < \ell} r_j q^j$. In particular, the $q$-ary expansion of
    zero is the empty word.}
  of $n$. Then
  \begin{equation*}
    x(n) = e_1 A_{r_0}\dotsm A_{r_{\ell-1}}v(0)
  \end{equation*}
  where $e_1=\begin{pmatrix}1& 0& \dotsc& 0\end{pmatrix}$.
\end{remark}

\subsection{Condensed Main Result}\label{section:introduction:main-result}

We are interested in the asymptotic behaviour of the summatory function
$X(N)=\sum_{0\le n<N}x(n)$.

At this point, we give a simplified version of our results. We choose any
vector norm $\norm{\,\cdot\,}$ on $\C^d$ and its induced matrix norm. We set $C\coloneqq
\sum_{0 \le r < q} A_r$. We choose $R>0$ such that $\norm{A_{r_1}\dotsm
  A_{r_\ell}}=\Oh{R^\ell}$ holds for all $\ell\ge 0$ and
$r_1$, \dots, $r_{\ell} \in \set{0,\dots,q-1}$.
In other words, $R$ is an upper bound for the joint spectral
radius of $A_0$, \ldots, $A_{q-1}$.
The spectrum of $C$, i.e., the set of eigenvalues of $C$, is denoted by
$\sigma(C)$. For $\lambda\in\C$, let $m(\lambda)$ denote the size of the
largest Jordan block of $C$ associated with $\lambda$; in particular,
$m(\lambda)=0$ if $\lambda\notin\sigma(C)$.
Finally, we consider the scalar-valued Dirichlet series~$\calX$
and the vector-valued Dirichlet series~$\calV$ defined
by\footnote{
Note that the summatory function $X(N)$ contains the summand $x(0)$ but the Dirichlet series cannot.
This is because the choice of including $x(0)$ into $X(N)$ will lead to more consistent results.}
\begin{equation*}
  \calX(s) = \sum_{n\ge 1} n^{-s}x(n) \qquad\text{and}\qquad
  \calV(s) = \sum_{n\ge 1} n^{-s}v(n)
\end{equation*}
where $v(n)$ is the vector-valued sequence defined in \eqref{eq:linear-representation}.
Of course, $\calX(s)$ is the first component of $\calV(s)$.
The principal value of the complex logarithm is denoted by $\log$. The
fractional part of a real number $z$ is denoted by $\fractional{z}\coloneqq z-\floor{z}$.

\begin{theorem}[User-friendly All-In-One Theorem]\label{theorem:simple}
  With the notations above, we have
  \begin{multline}\label{eq:formula-X-n}
    X(N) = \sum_{\substack{\lambda\in\sigma(C)\\\abs{\lambda}>R}}N^{\log_q\lambda}
    \sum_{0\le k<m(\lambda)}\frac{(\log N)^k}{k!}
    \Phi_{\lambda k}(\fractional{\log_q N}) \\
    + \Oh[\big]{N^{\log_q R}(\log N)^{\max\setm{m(\lambda)}{\abs{\lambda}=R}}}
  \end{multline}
  for suitable $1$-periodic continuous functions $\Phi_{\lambda k}$. If there
  are no eigenvalues $\lambda\in\sigma(C)$ with $\abs{\lambda}\le R$, the
  $O$-term can be omitted.

  For $\abs{\lambda}>R$ and $0\le k<m(\lambda)$, the function $\Phi_{\lambda k}$ is Hölder continuous with any exponent
  smaller than $\log_q(\abs{\lambda}/R)$.

  The Dirichlet series $\calV(s)$ converges absolutely and uniformly on compact
  subsets of the half plane $\Re
  s>\log_q R +1$ and can be continued to a meromorphic function on the half plane $\Re
  s>\log_q R$.
  It satisfies the functional equation
  \begin{equation}\label{eq:functional-equation-V}
    \bigl(I-q^{-s}C\bigr)\calV(s)= \sum_{1 \le n < q} n^{-s}v(n) +
    q^{-s}\sum_{0 \le r < q} A_r \sum_{k\ge
      1}\binom{-s}{k}\Bigl(\frac{r}{q}\Bigr)^k \calV(s+k)
  \end{equation}
  for $\Re s>\log_q R$. The right-hand side of~\eqref{eq:functional-equation-V} converges absolutely and uniformly on
  compact subsets of $\Re s>\log_q R$. In particular, $\calV(s)$ can only have
  poles where $q^s\in\sigma(C)$.

  For $\lambda\in\sigma(C)$ with
  $\abs{\lambda}>R$, the Fourier series
  \begin{equation*}
    \Phi_{\lambda k}(u) = \sum_{\ell\in \Z}\varphi_{\lambda k\ell}\exp(2\ell\pi i u)
  \end{equation*}
  converges pointwise for $u\in\R$ where the Fourier coefficients
  $\varphi_{\lambda k\ell}$ are defined by the singular expansion\footnote{We
    use the notion of singular expansion as defined by Flajolet, Gourdon and
    Dumas~\cite[Definition~2]{Flajolet-Gourdon-Dumas:1995:mellin}: it is the
    formal sum of the principal parts of a meromorphic function over all
    poles in the domain given.}
  \begin{equation}\label{eq:Fourier-coefficient:simple}
    \frac{x(0)+\calX(s)}{s} \asymp
    \sum_{\substack{\lambda\in\sigma(C)\\\abs{\lambda}>R}}\sum_{\ell\in\Z}\sum_{0\le k<m(\lambda)} \frac{\varphi_{\lambda k\ell}}{\bigl(s-\log_q \lambda-\frac{2\ell\pi i}{\log q}\bigr)^{k+1}}
  \end{equation}
  for $\Re s>\log_q R$.
\end{theorem}
This theorem is proved in Section~\ref{section:proof-theorem-simple}.
We note:
\begin{itemize}
\item We write $\Phi_{\lambda k}(\fractional{\log_q N})$ to optically
emphasise the $1$-periodicity; technically, we have $\Phi_{\lambda
  k}(\fractional{\log_q N})=\Phi_{\lambda k}(\log_q N)$.
\item The
arguments in the proof could be used to meromophically continue the Dirichlet
series to the complex plane, but we do not need this result for our purposes.
See~\cite{Allouche-Mendes-Peyriere:2000:autom-diric} for the corresponding argument for automatic sequences.
\item 
Sometimes, it will be convenient to write~\eqref{eq:Fourier-coefficient:simple}
in the equivalent explicit formulation
\begin{equation}\label{eq:Fourier-coefficient:simple-as-residue}
  \varphi_{\lambda k \ell}=\Res[\bigg]{\frac{x(0)+\calX(s)}{s}\Bigl(s-\log_q
    \lambda-\frac{2\ell\pi i}{\log q}\Bigr)^{k}}{s=\log_q \lambda+\frac{2\ell\pi i}{\log q}}.
\end{equation}
In particular, this can be used to algorithmically compute
the~$\varphi_{\lambda k \ell}$.
\item Computing the Fourier coefficients~$\varphi_{\lambda k \ell}$
  via the explicit formulation~\eqref{eq:Fourier-coefficient:simple-as-residue}
  by reliable numerical arithmetic (see Part~\ref{part:numerical} for details)
  enables us to detect the non-vanishing of a fluctuation; see also the
  example below and in Section~\ref{sec:transducer}
  (on sequences defined by transducers) for examples where the fluctuation of
  the leading term is in fact constant. There, additional arguments are required to actually prove this fact;
  see Section~\ref{section:non-vanishing} for more details.
\end{itemize}

We come back to the binary sum of digits.

\begin{example}[Continuation of Example~\ref{example:binary-sum-of-digits}]
  \label{example:binary-sum-of-digits:cont}
  We
  have $C=A_0+A_1=\bigl(
  \begin{smallmatrix}
    2&1\\0&2
  \end{smallmatrix}\bigr)
$. As $A_0$ is the identity matrix, any product $A_{r_1}\dotsm A_{r_\ell}$ has
the shape $A_1^k=\bigl(
\begin{smallmatrix}
  1&k\\0&1
\end{smallmatrix}\bigr)
$ where $k$ is the number of factors $A_1$ in the product. This implies that
$R$ with $\norm{A_{r_1}\dotsm A_{r_\ell}}=\Oh{R^\ell}$ may be chosen to be any number greater than $1$. As $C$ is a Jordan block
itself, we simply read off that the only eigenvalue of $C$
is $\lambda=2$ with $m(2)=2$.

Thus Theorem~\ref{theorem:simple} yields
\begin{equation*}
  X(N) = N(\log N) \f{\Phi_{21}}{\fractional{\log_2 N}}
  + N \f{\Phi_{20}}{\fractional{\log_2 N}}
\end{equation*}
for suitable $1$-periodic continuous functions $\Phi_{21}$ and $\Phi_{20}$.

In principle, we can now use the functional equation
\eqref{eq:functional-equation-V} to obtain the Dirichlet series~$\calX$.
Due to the fact that one component of $v$ is the constant
sequence where everything is known, it is more efficient to use an ad-hoc
calculation for $\calX$ by splitting the sum according to the parity of the index
and using the recurrence relation~\eqref{eq:recursion-binary-sum-of-digits} for $x(n)$. We obtain
\begin{align*}
  \calX(s)&=\sum_{n\ge 1}\frac{x(2n)}{(2n)^s} + \sum_{n\ge
            0}\frac{x(2n+1)}{(2n+1)^s}\\
  &=2^{-s}\sum_{n\ge 1}\frac{x(n)}{n^s} + \sum_{n\ge 0}\frac{x(n)}{(2n+1)^s} +
    \sum_{n\ge 0}\frac{1}{(2n+1)^s}\\
  &=2^{-s}\calX(s) + \frac{x(0)}{1^s} + \sum_{n\ge 1}\frac{x(n)}{(2n)^s} +
  \sum_{n\ge 1} x(n)\Bigl(\frac{1}{(2n+1)^s} - \frac{1}{(2n)^s}\Bigr) \\
  &\hspace{4.985em}+
  2^{-s}\sum_{n\ge 0}\frac1{\bigl(n+\frac12\bigr)^s}\\
  &= 2^{1-s}\calX(s) + 2^{-s}\f[\big]{\zeta}{s, \tfrac12} + \sum_{n\ge 1} x(n)\Bigl(\frac{1}{(2n+1)^s} - \frac{1}{(2n)^s}\Bigr),
\end{align*}
where the Hurwitz zeta function $\f{\zeta}{s, \alpha}\coloneqq\sum_{n+\alpha>0}(n+\alpha)^{-s}$ has been used. We get
\begin{equation}\label{eq:sum-of-digits-functional-equation}
  \bigl(1-2^{1-s}\bigr)\calX(s)=2^{-s} \f[\big]{\zeta}{s, \tfrac12} + \sum_{n\ge 1} x(n)\Bigl(\frac{1}{(2n+1)^s} - \frac{1}{(2n)^s}\Bigr).
\end{equation}
As the sum of digits is bounded by the length of the expansion, we have
$x(n)=\Oh{\log n}$. By combining this estimate with
\begin{equation*}
  (2n+1)^{-s}-(2n)^{-s}
  = \int_{2n}^{2n+1} \Bigl(\frac{\dd}{\dd t}t^{-s}\Bigr)\,\dd t
  = \int_{2n}^{2n+1}(-s)t^{-s-1}\,\dd t
  = \Oh[\big]{\abs{s}n^{-\Re s-1}},
\end{equation*}
we see that the sum in \eqref{eq:sum-of-digits-functional-equation}
converges absolutely for $\Re s>0$ and is therefore analytic for $\Re s>0$.

Therefore, the right-hand side of
\eqref{eq:sum-of-digits-functional-equation} is a meromorphic function for $\Re
s>0$ whose only pole is simple and at $s=1$ which originates from
$\f[\big]{\zeta}{s, \tfrac12}$.
Thus, $\calX(s)$ is a meromorphic function for $\Re s>0$ with a double
pole at $s=1$ and simple poles at $1+\frac{2\ell \pi i}{\log 2}$ for
$\ell\in\Z\setminus\set{0}$.

This gives us
\begin{equation}\label{eq:fluctuation-binary-sum-of-digit}
  \begin{aligned}
    \Phi_{21}(u) = \varphi_{210}
    &= \Res[\Big]{\frac{\calX(s)(s-1)}{s}}{s=1} \\
    &= \Res[\Big]{\frac{2^{-s}(s-1)}{1-2^{1-s}}
         \f[\big]{\zeta}{s, \tfrac12}}{s=1}
    = \frac1{2(\log 2)}
  \end{aligned}
\end{equation}
by \eqref{eq:Fourier-coefficient:simple-as-residue} and
\eqref{eq:sum-of-digits-functional-equation}.

We conclude that
\begin{equation*}
  X(N)=\frac12 N \log_2 N + N \f{\Phi_{20}}{\fractional{\log_2 N}}.
\end{equation*}
We will explain in Part~\ref{part:numerical} how to compute
rigorous numerical values for the Fourier coefficients, in our case
those of the fluctuation~$\Phi_{20}$ which can be deduced from
\eqref{eq:sum-of-digits-functional-equation}. 
In this particular case of the binary sum-of-digits, simpler and even explicit
expressions for the Fourier coefficients have been stated and derived by other
authors: They
can be obtained in our set-up by rewriting the residues of $\calX(s)$ in terms of shifted
residues of $\sum_{n\ge 1}\bigl(x(n)-x(n-1)\bigr)n^{-s}$ and by computing the latter
explicitly; see \cite[Proof of
Corollary~2.5]{Heuberger-Kropf-Prodinger:2015:output}. This yields the
well-known result by Delange~\cite{Delange:1975:chiffres}.

It will also turn out that \eqref{eq:fluctuation-binary-sum-of-digit} being a constant function is an
immediate consequence of the fact that $
\begin{pmatrix}
  0& 1
\end{pmatrix}
$ is a left eigenvector of both $A_0$ and $A_1$ associated with the eigenvalue
$1$;
see Theorem~\ref{theorem:contribution-of-eigenspace}.
\end{example}

\subsection{Asymptotics of Regular Sequences}
\label{section:asy-regular-sequences:non-summatory}

This article is written with a focus on the sequence of partial sums of a
regular sequence. In this section, however, we explain how to use all
material for the regular sequence itself.

Let $x(N)$ be a $q$-regular sequence. We may rewrite it as
a telescoping sum
\begin{equation}\label{eq:sum-of-differences}
  x(N) = x(0) + \sum_{n<N} \bigl( x(n+1) - x(n) \bigr).
\end{equation}
By~\cite[Theorems~2.5
and~2.6]{Allouche-Shallit:1992:regular-sequences}, the sequence of
differences $x(n+1) - x(n)$ is again $q$-regular. Conversely,
it is also well-known that the summatory function
of a $q$-regular sequence is itself $q$-regular. (This is an immediate
consequence of \cite[Theorem~3.1]{Allouche-Shallit:1992:regular-sequences}.)

Therefore, we
might also start to analyse a regular sequence by considering it to be the
summatory function of its sequence of differences as
in~\eqref{eq:sum-of-differences}. In this way, we can apply all of
the machinery developed in this article.

We end this short section with some remarks on why focusing on
the sequence of partial sums can be rewarding. When
modelling a quantity by a regular sequences, its asymptotic behaviour
is often not smooth, but the asymptotic behaviour of its summatory
function is. Moreover, we will see throughout this work that from a
technical perspective, considering partial sums is appropriate. Therefore,
we adopt this point of view of summatory functions of $q$-regular sequences
throughout this paper.

\section{Overview of the Full Results and Proofs}

\subsection{Overview of the Results}
\label{section:results:overview}

We have already seen the main results collected in a user-friendly
simplified version as Theorem~\ref{theorem:simple} which was written
down in a self-contained way in
Section~\ref{section:introduction:main-result}.

In Theorem~\ref{theorem:contribution-of-eigenspace} the assumptions
are refined. In particular, this theorem uses the joint spectral
radius~$R$ of the matrices in a linear representation of the sequence
(instead of a suitable bound for this quantity in
Theorem~\ref{theorem:simple}). Theorem~\ref{theorem:contribution-of-eigenspace}
states the contribution of each eigenvalue of the sum~$C$ of matrices
of the linear representation---split into the three cases of smaller,
equal and larger in absolute value than~$R$, respectively. This is formulated in
terms of generalised eigenvectors. As a consequence of this precise
breakdown of contributions, Theorem~\ref{theorem:main}, which collects
the different cases into one result, provides a condition on when the
error term vanishes.

Theorem~\ref{theorem:Dirichlet-series} brings up the full formulation
of the functional equation of the Dirichlet series associated to our
regular sequence. This is accompanied by a meromorphic continuation as
well as bounds on the growth of the Dirichlet series along vertical
lines (i.e., points with fixed real value). The analytic properties
provided by Theorem~\ref{theorem:Dirichlet-series} will be used to
verify the assumptions of Theorem~\ref{theorem:use-Mellin--Perron}.

Theorem~\ref{theorem:use-Mellin--Perron} is in fact stated and proved
very generally: It is not limited to Dirichlet
series coming from matrix products and regular sequences, but
it works for general Dirichlet series
provided that periodicity and continuity properties of the result
are known \emph{a priori}. This theorem
handles the Mellin--Perron summation and the theoretical foundations
for the computation of the Fourier coefficients of the appearing
fluctuations.

We want to point out that Theorem~\ref{theorem:use-Mellin--Perron}
can be viewed as a ``successful'' version of the
Mellin--Perron summation formula of order zero. In fact, the theorem
states sufficient conditions to provide the analytic justification
for the zeroth order formula.

Note that there is another result shown in this article, namely a
pseudo-Tauberian theorem for summing up periodic functions. This is
formulated as Proposition~\ref{proposition:pseudo-Tauber}, and all the
details around this topic are collected in
Section~\ref{sec:pseudo-tauber}. This pseudo-Tauberian argument is an
essential step in proving Theorem~\ref{theorem:use-Mellin--Perron}.

\subsection{Heuristic Approach: Mellin--Perron Summation}\label{sec:heuristic}
The purpose of this section is to explain why the formula
\eqref{eq:Fourier-coefficient:simple} for the Fourier coefficients is
expected. The approach here is heuristic and non-rigorous because we do not
have the required growth estimates. See also \cite{Drmota-Grabner:2010}.

By the Mellin--Perron summation formula of order $0$ (see, for example,
\cite[Theorem~2.1]{Flajolet-Grabner-Kirschenhofer-Prodinger:1994:mellin}),
we have
\begin{equation*}
  \sum_{1\le n<N}x(n) + \frac{x(N)}{2} = \frac1{2\pi i}\int_{\max\set{\log_q R + 2,1}
    -i\infty}^{\max\set{\log_q R + 2,1} +i\infty} \calX(s)\frac{N^s\,\dd s}{s}.
\end{equation*}
By Remark~\ref{remark:regular-sequence-as-a-matrix-product} and the definition
of $R$, we have
$x(N)=\Oh{R^{\log_q N}}=\Oh{N^{\log_q R}}$. Adding the summand $x(0)$ to match our definition of
$X(N)$ amounts to adding $\Oh{1}$.
Shifting the line of integration to the left---we have \emph{no analytic justification}
that this is allowed---and using the location of the poles of $\calX$ claimed in
Theorem~\ref{theorem:simple} yield
\begin{multline*}
  X(N) = \sum_{\substack{\lambda\in\sigma(C)\\\abs{\lambda}>R}}\sum_{\ell\in\Z}
  \Res[\Big]{\frac{\calX(s)N^s}{s}}%
  {s=\log_q \lambda + \frac{2\ell\pi i}{\log q}} \\
  + \frac1{2\pi i}\int_{\log_q R+\varepsilon
    -i\infty}^{\log_q R+\varepsilon +i\infty} \calX(s)\frac{N^s\,\dd s}{s} + \Oh{N^{\log_q R} + 1}
\end{multline*}
for some $\varepsilon>0$.
Expanding $N^s$ as
\begin{equation*}
  N^s = \sum_{k\ge 0} \frac{(\log N)^k}{k!} N^{\log_q \lambda + \frac{2\ell\pi
      i}{\log q}} \Bigl(s-\log_q \lambda
  -\frac{2\ell\pi i}{\log q}\Bigr)^k
\end{equation*}
and assuming that the remainder integral converges absolutely yield
\begin{multline*}
  X(N) = \sum_{\substack{\lambda\in\sigma(C)\\\abs{\lambda}>R}} N^{\log_q
    \lambda}\sum_{0\le k<m_{\lambda\ell}}
  \frac{(\log N)^k}{k!} \sum_{\ell\in\Z}\varphi_{\lambda k\ell}\exp\bigl(2\ell\pi i \log_q
  N\bigr)\\
  + \Oh{N^{\log_q R+\varepsilon}+1}
\end{multline*}
where $m_{\lambda \ell}$ denotes the order of the pole of $\calX(s)/s$ at
$\log_q\lambda + \frac{2\ell\pi i}{\log q}$ and $\varphi_{\lambda k \ell}$ is as
in \eqref{eq:Fourier-coefficient:simple}. (For $\lambda=1$ and $k=0$, the
contribution of $x(0)/s$ in \eqref{eq:Fourier-coefficient:simple} is absorbed
by the error term $\Oh{1}$ here.)

Summarising, this heuristic approach explains most of the formul\ae{} in
Theorem~\ref{theorem:simple}. Some details (exact error term and order of the
poles) are not explained by this approach.
A result ``repairing'' the zeroth order Mellin--Perron formula is known as
Landau's theorem; see \cite[\S~9]{Berthe-Lhote-Vallee:2016:probab}. It is not
applicable to our situation due to multiple poles along vertical lines which
then yield the periodic fluctuations. Instead, we present
Theorem~\ref{theorem:use-Mellin--Perron} which provides the required
justification (not by estimating the relevant quantities, but by reducing the
problem to higher order Mellin--Perron summation). The essential assumption is
that the summatory function can be decomposed into fluctuations multiplied by
some growth factors such as in \eqref{eq:formula-X-n}.

\subsection{High Level Overview of the Proof}
\label{section:high-level-overview-proof}

As we want to use Mellin--Perron summation in some form, we derive
properties of the Dirichlet series associated to the regular sequence. In
particular, we derive a functional equation which allows to compute the
Dirichlet series and its residues with arbitrary precision (Theorem~\ref{theorem:Dirichlet-series}).

We cannot directly use Mellin--Perron summation of order zero
for computing the Fourier coefficients of the fluctuations of interest.
As demonstrated in Section~\ref{sec:heuristic}, however, our theorems
coincide with the results which Mellin--Perron summation of order zero
would give if the required growth estimates could be
provided. Unfortunately, we are unable to prove these required growth
estimates. Therefore, we have to circumvent the problem by applying a
generalisation of the pseudo-Tauberian argument by
Flajolet, Grabner, Kirschenhofer, Prodinger and Tichy~\cite{Flajolet-Grabner-Kirschenhofer-Prodinger:1994:mellin}.

In order to use this argument, we have to know that the asymptotic formula has
the shape~\eqref{eq:formula-X-n}. Note that a successful application
(not \emph{directly} possible!)
of Mellin--Perron summation of order zero would give this directly.
Therefore, we first prove~\eqref{eq:formula-X-n}
and the existence of the
fluctuations (Theorems~\ref{theorem:contribution-of-eigenspace} and~\ref{theorem:main}).
To do so, we decompose the problem into contributions of
the eigenspaces of the matrix $C=A_0+\cdots+A_{q-1}$. The regular sequence is
then expressed as a matrix product. Next, we construct the
fluctuations by elementary means: We replace finite sums occurring in the
summatory functions by infinite sums involving digits using the factorisation
as a matrix product.

Then the pseudo-Tauberian argument states that the summatory function of the
fluctuation is again a fluctuation and there is a
relation between the Fourier coefficients of these fluctuations. The Fourier
coefficients of the summatory function of the fluctuation, however, can be
computed by Mellin--Perron summation of order one, so the Fourier coefficients
of the original fluctuation can be recovered; see Theorem~\ref{theorem:use-Mellin--Perron}.

\subsection{Relation to Previous Work}
\label{introduction:relation-to-previous-work}

The asymptotics of the summatory function of specific examples of regular
sequences has been studied in \cite{Grabner-Heuberger:2006:Number-Optimal},
\cite{Grabner-Heuberger-Prodinger:2005:counting-optimal-joint},
\cite{Dumas-Lipmaa-Wallen:2007:asymp}. There, various methods have been
used to show that the fluctuations exist; then the original
pseudo-Tauberian argument by Flajolet, Grabner, Kirschenhofer,
Prodinger and Tichy~\cite{Flajolet-Grabner-Kirschenhofer-Prodinger:1994:mellin}
is used to compute the Fourier coefficients of the fluctuations.

The first version of the pseudo-Tauberian argument in Theorem~\ref{theorem:use-Mellin--Perron}
was provided in \cite{Flajolet-Grabner-Kirschenhofer-Prodinger:1994:mellin}:
There, no logarithmic factors were allowed, only values $\gamma$ with $\Re \gamma>0$
were allowed and the result contained an error term of $o(1)$ whereas we give a
more precise error estimate in order to allow repeated application.

Dumas~\cite{Dumas:2013:joint, Dumas:2014:asymp} proved the first part
of Theorem~\ref{theorem:simple} using dilation equations. We re-prove it here
in a self-contained way
because we need more explicit results than obtained by Dumas
(e.g., we need explicit expressions for the
fluctuations) to explicitly get the precise
structure depending on the eigenspaces
(Theorem~\ref{theorem:contribution-of-eigenspace}).
Notice that the order of factors in Dumas' paper is inconsistent between
his versions of~\eqref{eq:linear-representation} and
Remark~\ref{remark:regular-sequence-as-a-matrix-product}.

A functional equation for the Dirichlet series of an automatic sequence
has been proved by Allouche, Mendès France and
Peyrière~\cite{Allouche-Mendes-Peyriere:2000:autom-diric}.

In Section~\ref{sec:transducer} we study transducers. The
sequences there are defined as the output sum of transducer automata in the sense of
\cite{Heuberger-Kropf-Prodinger:2015:output}. They are a special case of regular
sequences and are a generalisation of many previously studied concepts.
In that case, much more is known (variance, limiting distribution, higher
dimensional input); see \cite{Heuberger-Kropf-Prodinger:2015:output} for
references and results. A more detailed comparison can be found in Section~\ref{sec:transducer}.
Divide and conquer recurrences (see
\cite{Drmota-Szpankowski:2013:divide-and-conquer} and
\cite{Hwang-Janson-Tsai:2017:divide-conquer-half})
can also be seen as special cases of regular sequences.

The present article gives a unified approach which covers all cases of
regular sequences. As long as the conditions on the joint spectral radius are
met, the main asymptotic terms are not absorbed by the error terms. Otherwise,
the regular sequence is so irregular that the summatory function is not smooth
enough to allow a result of this shape.

\section{Overview of the Examples}
\label{sec:overview-examples}

We take a closer look at three particular examples.
In this section, we provide an overview of these examples;
all details can be found in Part~\ref{part:examples}.

At first gance it seems that
these examples are straight-forward applications of the results. However,
we have to reformulate the relevant questions in terms of a
$q$-regular sequence and will then provide shortcuts for the
computation of the Fourier series.  We put a special effort on the
details which gives additional insights like dependencies on certain
residue classes; see Section~\ref{sec:residue-classes}. Moreover, the
study of these examples also encourages us to investigate symmetries
in the eigenvalues; see
Section~\ref{sec:symmetric-eigenvalues-overview} for an overview and
Section~\ref{sec:symmetric} for general considerations.

We start with transducer automata. Transducers have been chosen in
order to compare the results here with the previously available
results~\cite{Heuberger-Kropf-Prodinger:2015:output}.
In some sense, the results complement each other: While the
results in~\cite{Heuberger-Kropf-Prodinger:2015:output} also contain information on the variance and the limiting
distribution, our approach here yields more terms of the asymptotic
expansion of the mean, at least in the general case. Also, it is a class of
examples.

We then continue with esthetic
numbers. These numbers are an example of an automatic sequence, therefore
can be treated by a transducer. However,
it turns out that the generic results
(the results here and in~\cite{Heuberger-Kropf-Prodinger:2015:output})
degenerate: They are too weak to give a meaningful main term. Therefore a
different effort is needed for esthetic numbers.
No precise asymptotic results were known previously.

The example on Pascal's Rhombus is a choice of a regular sequence
where all components of the vector sequence have some combinatorial
meaning. Again, no precise asymptotic results were known previously.

Section~\ref{sec:overview:further-examples} contains further
examples. Note that there are the two additional
Sections~\ref{sec:residue-classes}
and~\ref{sec:symmetric-eigenvalues-overview} pointing out phenomena
appearing in the analysis of our examples.

\subsection{Transducers}
\label{sec:overview-transducers}

The sum~$\calT(n)$ of the output labels of a complete deterministic finite transducer~$\calT$
when reading the $q$-ary expansion of an integer~$n$ has been investigated
in~\cite{Heuberger-Kropf-Prodinger:2015:output}. As this can be seen as a
$q$-regular sequence, we reconsider the problem in the light of our
general results in this article;
see Section~\ref{sec:transducer}. For the summatory function, the main
terms corresponding to the eigenvalue $q$ can be extracted by both results; if
there are further eigenvalues larger than the joint spectral radius, our
Corollary~\ref{corollary:transducer-main} allows to describe more asymptotic terms which are absorbed by
the error term in~\cite{Heuberger-Kropf-Prodinger:2015:output}. Note, however,
that our approach here does not give any readily available information on the
variance (this could somehow be repaired for specific examples because regular
sequences are known to form a ring) nor on the limiting distribution.

\subsection{Esthetic Numbers}
\label{sec:overview-esthetic}

In this article, we also contribute a
precise asymptotic analysis of $q$-esthetic numbers; see~De~Koninck
and Doyon~\cite{Koninck-Doyon:2009:esthetic-numbers}. These are
numbers whose $q$-ary digit expansion satisfies the condition that
neighboring digits differ by exactly one. The sequence of such numbers
turns out to be $q$-automatic, thus are $q$-regular and can also be
seen as an output sum of a transducer; see the first author's joint
work with Kropf and
Prodinger~\cite{Heuberger-Kropf-Prodinger:2015:output} or
Section~\ref{sec:transducer}. However, the
asymptotics obtained by using the main result of
\cite{Heuberger-Kropf-Prodinger:2015:output} is degenerated
in the sense that the provided main term and second order term both
equal zero; only an error term remains. On the other hand, using a more direct approach via our
main theorem brings up the actual main term and the fluctuation in
this main term. We also explicitly compute the Fourier coefficients.
The full theorem is formulated in
Section~\ref{sec:esthetic-numbers}.
Prior to this precise analysis,
the authors of~\cite{Koninck-Doyon:2009:esthetic-numbers} only performed an analysis
of esthetic numbers by digit-length (and not by the number itself).

The approach used in the analysis of $q$-esthetic numbers can easily
be adapted to numbers defined by other conditions on the word of
digits of their $q$-ary expansion.

\subsection{Dependence on Residue Classes}
\label{sec:residue-classes}

The analysis of $q$-esthetic numbers also brings another aspect into
the light of day, namely a quite interesting dependence of the
behaviour with respect to~$q$ on different moduli:
\begin{itemize}
\item The dimensions in the matrix approach of
  \cite{Koninck-Doyon:2009:esthetic-numbers} need to be increased for
  certain residue classes of~$q$ modulo~$4$ in order to get a
  formulation as a $q$-automatic and $q$-regular sequence,
  respectively.
\item The main result in~\cite{Koninck-Doyon:2009:esthetic-numbers}
  already depends on the parity of $q$ (i.e., on $q$
  modulo~$2$). This reflects our Corollary~\ref{corollary:esthetic:asy}
  by having $2$-periodic
  fluctuations (in contrast to $1$-periodic fluctuations in the main
  Theorem~\ref{theorem:simple}).
\item Surprisingly, the error term in the resulting formula of
  Corollary~\ref{corollary:esthetic:asy} depends on the residue class of $q$ modulo~$3$. This can be seen in the spectrum of the matrix~$C=\sum_{0 \le r < q} A_r$:
  There is an appearance of an eigenvalue~$1$ in certain cases.
\item As an interesting side-note: In the
  spectrum of~$C$, the algebraic multiplicity of the
  eigenvalue~$0$ changes again only modulo~$2$.
\end{itemize}

\subsection{Symmetrically Arranged Eigenvalues}
\label{sec:symmetric-eigenvalues-overview}

Fluctuations with longer periods (like in
the second of the four bullet points above) come from a particular
configuration in the spectrum of~$C$. Whenever eigenvalues are arranged as
vertices of a regular polygon, then their influence can be collected;
this results in periodic fluctuations with larger period than~$1$.
We elaborate on the influence of such eigenvalues in
Section~\ref{sec:symmetric}.
This is then used in the particular cases of esthetic numbers and in
conjunction with the output sum of transducers. More specifically, in the latter example
this yields the second order term in
Corollary~\ref{corollary:transducer-main}; see
also~\cite{Heuberger-Kropf-Prodinger:2015:output}.

\subsection{Pascal's Rhombus}
\label{sec:overview-pascal}

Beside esthetic numbers, we perform an asymptotic analysis of the
number of ones in the rows of Pascal's rhombus. The rhombus is in some
sense a variant of Pascal's triangle---its recurrence is similar to that
of Pascal's triangle. It turns out that the
number of ones in the rows of Pascal's rhombus can be modelled by a
$2$-regular sequence.

The authors
of~\cite{Goldwasser-Klostermeyer-Mays-Trapp:1999:Pascal-rhombus}
investigate this number of ones, but only for blocks whose number of rows
is a power of~$2$. In the precise analysis in
Section~\ref{sec:pascal} we not only obtain the asymptotic formula, we
also explicitly compute the Fourier coefficients.

\subsection{Further Examples}
\label{sec:overview:further-examples}

There are many further examples of specific $q$-regular sequences which await
precise asymptotic analysis, for example the Stern--Brocot sequence~\oeis{A002487}, the
denominators of Farey tree fractions~\oeis{A007306}, the
 number of unbordered factors of length $n$ of the Thue--Morse
 sequence (see \cite{Goc-Mousavi-Shallit:2013}).

The Stern--Brocot sequence is a typical example: It is defined by $x(0)=0$,
$x(1)=1$ and
\begin{equation}\label{eq:stern-brocot:rec}
\begin{aligned}
  x(2n)&=x(n),\\
  x(2n+1)&=x(n)+x(n+1),
\end{aligned}
\end{equation}
i.e., the right-hand sides are linear combinations of shifted versions of the
original sequence.

Note that recurrence relations like~\eqref{eq:stern-brocot:rec} are
not proper linear representations of regular sequences in the sense of~\eqref{eq:linear-representation}. The good news,
however, is that in general, such a sequence is $q$-regular. The
following remark formulates this more explicitly.

\begin{remark}
  Let $x(n)$ be a sequence such that
  there are fixed integers $\ell\le 0\le u$ and constants $c_{rk}$ for $0\le r<q$
  and $\ell\le k\le u$ such that
  \begin{equation*}
    x(qn+r) = \sum_{\ell \le k\le u} c_{rk}x(n+k)
  \end{equation*}
  holds for $0\le r<q$ and $n\ge 0$. Then the sequence $x(n)$ is $q$-regular with
  $q$-linear representation for $v(n)=\bigl(x(n+\ell'), \ldots, x(n), \ldots,
  x(n+u')\bigr)^\top$ where
  \begin{equation*}
    \ell'=\floor[\Big]{\frac{q\ell}{q-1}},\qquad
    u'=\ceil[\Big]{\frac{qu}{q-1}}.
  \end{equation*}
  Note that if $\ell'<0$, then a simple permutation of the components
  of~$v(n)$ brings~$x(n)$ to its first component (so that the above is indeed
  a proper linear representation as defined in
  Section~\ref{section:introduction:regular-sequences}).
\end{remark}

By using this remark on~\eqref{eq:stern-brocot:rec}, we set
$v(n)=\bigl(x(n), x(n+1), x(n+2)\bigr)^\top$ and obtain the $2$-linear
representation
\begin{equation*}
  v(2n)=
  \begin{pmatrix}
    1&0&0\\
    1&1&0\\
    0&1&0
  \end{pmatrix}v(n),\qquad
  v(2n+1)=
  \begin{pmatrix}
    1&1&0\\
    0&1&0\\
    0&1&1
  \end{pmatrix}v(n)
\end{equation*}
for $n\ge 0$ for the Stern--Brocot sequence.

\section{Full Results}\label{section:results}

In this section, we fully formulate our results. As pointed out in
Remark~\ref{remark:regular-sequence-as-a-matrix-product}, regular
sequences can essentially be seen as matrix products. Therefore, we
will study these matrix products instead of regular sequences.
Theorem~\ref{theorem:simple} can then be proved as a simple corollary of the
results for matrix products; see Section~\ref{section:proof-theorem-simple}.

\subsection{Problem Statement}
Let $q\ge 2$, $d\ge 1$ be fixed integers and $A_0$, \ldots,
$A_{q-1}\in\C^{d\times d}$.
We investigate the sequence~$f$ of $d\times d$ matrices such that
\begin{equation}\label{eq:regular-matrix-sequence}
  f(qn+r)=A_r f(n) \quad\text{ for $0\le r<q$, $0\le n$ with  $qn+r\neq 0$}
\end{equation}
and $f(0)=I$.

Let $n$ be an integer with $q$-ary expansion
$r_{\ell-1}\ldots r_0$. Then it is easily seen that \eqref{eq:regular-matrix-sequence} implies that
\begin{equation}\label{eq:f-as-product}
  f(n)=A_{r_0}\ldots A_{r_{\ell-1}}.
\end{equation}

We are interested in the asymptotic behaviour of $F(N)\coloneqq\sum_{0\le n<N} f(n)$.

\subsection{Definitions and Notations}\label{sec:definitions-notations}
In this section, we give all definitions and notations which are required in
order to state the results. For the sake of conciseness, we do not give any
motivations for our definitions here; those are deferred to Section~\ref{sec:motivation-definitions}.

The following notations are essential:
\begin{itemize}
\item
Let $\norm{\,\cdot\,}$ denote a fixed norm on $\C^d$ and its induced matrix
norm on $\C^{d\times d}$.

\item We set $B_r \coloneqq \sum_{0\le r'<r} A_{r'}$ for $0\le
r<q$ and $C\coloneqq\sum_{0\le r<q} A_r$.

\item
The joint spectral radius of $A_0$, \ldots, $A_{q-1}$ is denoted by
\begin{equation*}
  \rho\coloneqq\inf_{\ell}\sup
  \setm[\big]{ \norm{A_{r_1}\ldots A_{r_\ell}}^{1/\ell}}{r_1, \ldots, r_\ell\in\set{0, \ldots, q-1}}.
\end{equation*}
If the set of matrices $A_0$, \dots, $A_{q-1}$ has the \emph{finiteness property},
i.e., there is an $\ell>0$ such that
\begin{equation*}
  \rho = \sup
  \setm[\big]{\norm{A_{r_1}\ldots A_{r_\ell}}^{1/\ell}}{r_1, \ldots, r_\ell\in\set{0, \ldots, q-1}},
\end{equation*}
then we set $R=\rho$. Otherwise, we choose $R>\rho$ in such a way that there is
no eigenvalue $\lambda$ of $C$ with $\rho<\abs{\lambda}\le R$.

\item
The spectrum of $C$, i.e., the set of eigenvalues of $C$, is denoted by
$\sigma(C)$.

\item For a positive integer $n_0$, let $\calF_{n_0}$ be the
  matrix-valued Dirichlet series defined by
  \begin{equation*}
    \calF_{n_0}(s) \coloneqq \sum_{n\ge n_0} n^{-s}f(n)
  \end{equation*}
  for a complex variable $s$.

\item Set $\chi_k\coloneqq \frac{2\pi i k}{\log q}$ for $k\in\Z$.
\end{itemize}

In the formulation of Theorem~\ref{theorem:contribution-of-eigenspace} and
Theorem~\ref{theorem:main}, the following constants are needed
additionally:

\begin{itemize}
\item
Choose a regular matrix $T$ such that $T C T^{-1}\eqqcolon J$ is in Jordan form.

\item  Let $D$ be
the diagonal matrix whose $j$th diagonal element is $1$ if the $j$th diagonal
element of $J$ is not equal to $1$; otherwise the $j$th diagonal element of $D$
is $0$.

\item
Set $C'\coloneqq T^{-1}DJT$.

\item
Set $K\coloneqq T^{-1}DT(I-C')^{-1}(I-A_0)$.

\item
For a $\lambda\in\C$, let $m(\lambda)$ be the size of the largest
Jordan block associated with $\lambda$. In particular, $m(\lambda)=0$ if $\lambda\not\in\sigma(C)$.

\item For $m\ge 0$, set
\begin{equation*}
  \vartheta_m \coloneqq \frac1{m!}T^{-1}(I-D)T(C-I)^{m-1}(I-A_0);
\end{equation*}
here, $\vartheta_0$ remains undefined if $1\in\sigma(C)$.\footnote{
If $1\in\sigma(C)$, then the matrix $C-I$ is singular. In that case, $\vartheta_0$ will never be used.}

\item Define $\vartheta \coloneqq \vartheta_{m(1)}$.
\end{itemize}

All implicit $O$-constants depend on $q$, $d$, the matrices $A_0$, \ldots, $A_{q-1}$ (and therefore on $\rho$),
as well as on $R$.

\subsection{Decomposition into Periodic Fluctuations}
Instead of considering $F(N)$, it is certainly enough to consider $wF(N)$ for
all generalised left eigenvectors $w$ of $C$, e.g., the rows of $T$. The
result for $F(N)$ then follows by taking appropriate linear combinations.

\begin{theorem}\label{theorem:contribution-of-eigenspace}
  Let $w$ be a generalised left eigenvector of rank $m$ of $C$ corresponding to the eigenvalue $\lambda$.
  \begin{enumerate}
  \item\label{item:small-eigenvalue} If $\abs{\lambda}<R$, then
    \begin{equation*}
      wF(N)=wK + (\log_q N)^m w\vartheta_m  + \Oh{N^{\log_q R}}.
    \end{equation*}
  \item\label{item:R-eigenvalue} If $\abs{\lambda}=R$, then
    \begin{equation*}
      wF(N)=wK + (\log_q N)^m w\vartheta_m + \Oh{N^{\log_q R} (\log N)^{m}}.
    \end{equation*}
  \item\label{item:large-eigenvalue} If $\abs{\lambda}>R$, then there are $1$-periodic continuous functions
    $\Phi_k\colon \R\to\C^d$, $0\le k<m$, such that
    \begin{equation*}
      wF(N)=wK + (\log_q N)^mw\vartheta_m + N^{\log_q\lambda} \sum_{0\le k<m}(\log_q N)^k\Phi_k(\fractional{\log_q N})
    \end{equation*}
    for $N\ge q^{m-1}$. The function $\Phi_k$ is Hölder continuous with any
    exponent smaller than  $\log_q\abs{\lambda}/R$.

    If, additionally, the left eigenvector $w(C-\lambda I)^{m-1}$ of $C$ happens to be a left eigenvector to each matrix
    $A_0$, \ldots, $A_{q-1}$ associated with the eigenvalue~$1$, then
    \begin{equation*}
      \Phi_{m-1}(u)=\frac1{q^{m-1}(m-1)!}w(C-q I)^{m-1}
    \end{equation*}
    is constant.
  \end{enumerate}
  Here, $wK=0$  for $\lambda=1$ and $w\vartheta_m=0$ for $\lambda\neq 1$.
\end{theorem}
This theorem is proved in Section~\ref{section:proof-contribution-of-eigenspace}.
Note that in general, the three summands in the theorem have different growths:
a constant, a logarithmic term and a term whose growth depends essentially
on the joint spectral radius and the eigenvalues larger than the
joint spectral radius, respectively. The vector $w$ is not directly visible in front of
the third summand; instead, the vectors of its Jordan chain are part of the function~$\Phi_k$.

Expressing the identity matrix as linear combinations of generalised left
eigenvalues and summing up the contributions of
Theorem~\ref{theorem:contribution-of-eigenspace} essentially yields the following corollary.

\begin{theorem}\label{theorem:main}
  With the notations above, we have
  \begin{multline*}
    F(N) = \sum_{\substack{\lambda\in\sigma(C)\\\abs{\lambda}>\rho}} N^{\log_q
      \lambda}\sum_{0\le k<m(\lambda)}(\log_q N)^k\Phi_{\lambda
      k}(\fractional{\log_q N})  + (\log_q N)^{m(1)} \vartheta + K\\
    + \Oh[\big]{N^{\log_q R}(\log N)^{\max\setm{m(\lambda)}{\abs{\lambda}=R}}}
  \end{multline*}
  for suitable $1$-periodic continuous functions $\Phi_{\lambda k}$.
  If $1$ is not an eigenvalue of $C$, then $\vartheta=0$.  If
  there are no eigenvalues $\lambda\in\sigma(C)$ with $\abs{\lambda}\le \rho$,
  then the $O$-term can be omitted.

  For $\abs{\lambda}>R$, the function $\Phi_{\lambda k}$ is Hölder continuous with any exponent
  smaller than $\log_q(\abs{\lambda}/R)$.
\end{theorem}
This theorem is proved in Section~\ref{section:proof:corollary-main}.

\begin{remark}
  We want to point out that the condition $\abs{\lambda}>R$ is
  inherent in the problem: Single summands $f(n)$ might be as large as
  $n^{\log_q R}$ and must therefore be absorbed by the error term in
  any smooth asymptotic formula for the summatory function.
\end{remark}

\subsection{Dirichlet Series}
This section gives the required result on the Dirichlet series~$\calF_{n_0}$. For
theoretical purposes, it is enough to study $\calF\coloneqq\calF_1$; for numerical purposes,
however, convergence improves for larger values of $n_0$. This is because for
large $n_0$ and large $\Re s$, the value of $\calF_{n_0}(s)$ is roughly $n_0^{-s} f(n_0)$; see also Part~\ref{part:numerical}.

\begin{theorem}\label{theorem:Dirichlet-series}Let $n_0$ be a positive
  integer. Then the Dirichlet series $\calF_{n_0}(s)$
  converges absolutely and uniformly on compact subsets of the half plane $\Re s > \log_q \rho + 1$, thus is analytic there.

  We have
  \begin{equation}\label{eq:analytic-continuation}
    \bigl(I-q^{-s}C\bigr)\calF_{n_0}(s) = \calG_{n_0}(s)
  \end{equation}
  for $\Re s>\log_q \rho +1$ with
  \begin{equation}\label{eq:Dirichlet-recursion}
    \calG_{n_0}(s) = \sum_{n_0 \le n < qn_0} n^{-s}f(n)
    + q^{-s} \sum_{0 \le r < q} A_r
    \sum_{k\ge 1} \binom{-s}{k}\Bigl(\frac{r}{q}\Bigr)^k \calF_{n_0}(s+k).
  \end{equation}
  The series in \eqref{eq:Dirichlet-recursion} converge
  absolutely and uniformly on compact sets for $\Re s>\log_q \rho$. Thus \eqref{eq:analytic-continuation} gives a meromorphic
  continuation of $\calF_{n_0}(s)$ to the half plane $\Re s>\log_q \rho$ with
  possible poles at $s=\log_q \lambda + \chi_\ell$ for each
  $\lambda\in \sigma(C)$ with $\abs{\lambda}>\rho$ and $\ell\in\Z$
  whose pole order is at most $m(\lambda)$.

  Let $\delta>0$. For real $z$, we set
  \begin{equation*}
    \mu_\delta(z)= \max\set{ 1 - (z-\log_q \rho -\delta), 0},
  \end{equation*}
  i.e., the linear function on the interval
  $[\log_q\rho+\delta, \log_q\rho+\delta+1]$
  with~$\mu_\delta(\log_q\rho+\delta)=1$ and~$\mu_\delta(\log_q\rho+\delta+1)=0$.
  Then
  \begin{equation}\label{eq:order-F}
    \calF_{n_0}(s) = \Oh[\big]{\abs{\Im s}^{\mu_\delta(\Re s)}}
  \end{equation}
  holds uniformly for $\log_q \rho+\delta\le \Re s$ and $\abs{q^s-\lambda} \ge \delta$
  for all eigenvalues $\lambda\in\sigma(C)$. Here, the implicit $O$-constant
  also depends on $\delta$.
\end{theorem}

Note that by the introductory remark on $\calF_{n_0}(s)$, the infinite
sum over $k$ in~\eqref{eq:Dirichlet-recursion} can be well approximated
by a finite sum. Detailed error bounds are discussed in
Part~\ref{part:numerical}. Therefore the theorem allows to transfer the information
on $\calF_{n_0}(s)$ for large~$\Re s$ where convergence is unproblematical
to values of $s$ where the convergence of the Dirichlet series $\calF_{n_0}$ itself is bad.

\begin{remark}\label{remark:Dirichlet-series:bound}
  By the identity theorem for analytic functions, the meromorphic
  continuation of $\calF_{n_0}$ is unique on the domain given in the
  theorem. Therefore, the bound~\eqref{eq:order-F} does not depend on
  the particular expression for the meromorphic continuation given
  in~\eqref{eq:analytic-continuation}
  and~\eqref{eq:Dirichlet-recursion}.
\end{remark}

Theorem~\ref{theorem:Dirichlet-series} is proved in
Section~\ref{section:proof:Dirichlet-series}. In the proof we
translate the linear representation
of $f$ into a system of equations involving $\calF_{n_0}(s)$
and shifted versions like $\sum_{n\ge n_0}f(n)(n+\beta)^{-s}$.
We will have to bound the
difference between the shifted and unshifted versions of the Dirichlet series.
These bounds are provided by the following lemma.
It will turn out
to be useful to have it as a result listed in this section and not
buried in the proofs sections.

\begin{lemma}\label{lemma:shifted-Dirichlet}
  Let $\calD(s) = \sum_{n \ge n_0} d(n)/n^s$ be a Dirichlet series with
  coefficients $d(n)=\Oh{n^{\log_q R'}}$ for all $R'>\rho$.
  Let $\beta\in\C$ with $\abs{\beta}<n_0$ and $\delta>0$. Set
  \begin{equation*}
    \f{\Sigma}{s, \beta, \calD} \coloneqq
    \sum_{n\ge n_0} \frac{d(n)}{(n+\beta)^s} - \calD(s).
  \end{equation*}
  Then
  \begin{equation*}
    \f{\Sigma}{s, \beta, \calD} = \sum_{k\ge 1}
    \binom{-s}{k} \beta^k \calD(s+k),
  \end{equation*}
  where the series converges
  absolutely and uniformly on compact sets for $\Re s>\log_q \rho$,
  thus $\f{\Sigma}{s, \beta, \calD}$ is analytic there.
  Moreover, with $\mu_\delta$ as in Theorem~\ref{theorem:Dirichlet-series},
  \begin{equation*}
    \f{\Sigma}{s, \beta, \calD}=\Oh[\big]{\abs{\Im s}^{\mu_\delta(\Re s)}}
  \end{equation*}
  as $\abs{\Im s}\to\infty$
  holds uniformly for $\log_q \rho + \delta\le \Re s\le \log_q \rho +\delta+1$.
\end{lemma}

\subsection{Fourier Coefficients}
As discussed in Section~\ref{sec:heuristic}, we would like to apply the zeroth
order Mellin--Perron summation formula but need analytic justification. In the
following theorem we prove that whenever it is known that the result is a
periodic fluctuation, the use of zeroth order Mellin--Perron summation can be
justified. In contrast to the remaining parts of the paper, this theorem does \emph{not}
assume that $f(n)$ is a matrix product.

\begin{theorem}\label{theorem:use-Mellin--Perron}
  Let $f$ be a sequence on $\Z_{>0}$, let $\gamma_0\in\R\setminus \Z_{\le 0}$ and $\gamma\in\C$ with $\Re \gamma>
  \gamma_0$, $\delta>0$,
  $q>1$ be real numbers with $\delta \le \pi/(\log q)$
  and $\delta < \Re \gamma-\gamma_0$,
  and let $m$ be a positive integer. Moreover, let $\Phi_j$ be
  Hölder continuous (with exponent $\alpha$ with
  $\Re\gamma-\gamma_0<\alpha\le 1$) $1$-periodic functions for $0\le j<m$ such that
  \begin{equation}\label{eq:F-N-periodic}
    F(N)\coloneqq \sum_{1\le n< N} f(n) = \sum_{\substack{j+k=m-1\\0\le j<m}}N^\gamma \frac{(\log N)^k}{k!}
    \Phi_j(\fractional{\log_q N}) + \Oh{N^{\gamma_0}}
  \end{equation}
  for integers $N\to\infty$.

  For the Dirichlet series $\calF(s)\coloneqq \sum_{n\ge 1}n^{-s}f(n)$
  assume that
  \begin{itemize}
  \item there is some real number $\sigmaabs\ge \Re \gamma$ such that $\calF(s)$ converges absolutely for $\Re s>\sigmaabs$;
  \item  the function $\calF(s)/s$
    can be continued to a meromorphic function for $\Re s > \gamma_0-\delta$
    such that poles can only occur at $\gamma+\chi_\ell$ for $\ell\in\Z$ and such that these poles have order at
    most $m$ and a possible pole at $0$; the local expansions are written as
    \begin{equation}\label{eq:Fourier:F-s-principal-part}
      \frac{\calF(s)}{s}=\frac1{(s-\gamma-\chi_\ell)^m}\sum_{j\ge 0}\varphi_{j\ell}(s-\gamma-\chi_\ell)^j
    \end{equation}
    with suitable constants $\varphi_{j\ell}$ for $j$, $\ell\in\Z$;
  \item there is some real number~$\eta>0$ such that
    for $\gamma_0 \le \Re s  \le \sigmaabs$ and
    $\abs{s-\gamma-\chi_\ell}\ge \delta$ for all $\ell\in\Z$, we have
    \begin{equation}\label{eq:Dirichlet-order}
      \calF(s) = \Oh[\big]{\abs{\Im s}^{\eta}}
    \end{equation}
    for $\abs{\Im s}\to\infty$.
  \end{itemize}
  All implicit $O$-constants may depend on $f$, $q$, $m$, $\gamma$, $\gamma_0$,
  $\alpha$, $\delta$, $\sigmaabs$ and $\eta$.

  Then
  \begin{equation}\label{eq:Fourier:fluctuation-as-Fourier-series}
    \Phi_j(u) = \sum_{\ell\in \Z}\varphi_{j\ell}\exp(2\ell\pi i u)
  \end{equation}
  for $u\in\R$, $\ell\in\Z$ and $0\le j<m$.

  If $\gamma_0<0$ and $\gamma\notin \frac{2\pi i}{\log q}\Z$, then $\calF(0)=0$.
\end{theorem}
This theorem is proved in Section~\ref{section:proof:use-Mellin--Perron}.
The theorem is more general than necessary for $q$-regular sequences because
Theorem~\ref{theorem:Dirichlet-series} shows that we could use some $0<\eta<1$.
However, it might be applicable in other cases, so we prefer to state it in
this more general form.

\subsection{Fluctuations of Symmetrically Arranged Eigenvalues}
\label{sec:symmetric}

In our main results, the occurring fluctuations are always
$1$-periodic functions. However, if eigenvalues of the sum of matrices
of the linear representation are
arranged in a symmetric way, then we can combine summands and get
fluctuations with longer periods. This is in particular true if all
vertices of a regular polygon (with center~$0$) are eigenvalues.

\begin{proposition}\label{proposition:symmetric-eigenvalues}
  Let $\lambda\in\C$, and let $k\ge0$ and $p>0$ be integers. Denote by $U_p$ the
  set of $p$th roots of unity. Suppose for each $\zeta\in U_p$
  we have a continuous $1$-periodic function
  \begin{equation*}
    \Phi_{(\zeta\lambda)}(u)
    = \sum_{\ell\in\Z}\varphi_{(\zeta\lambda)\ell}\exp(2\ell\pi i u)
  \end{equation*}
  whose Fourier coefficients are
  \begin{equation*}
    \varphi_{(\zeta \lambda)\ell}
    =\Res[\bigg]{\calD(s)
      \Bigl(s - \log_q (\zeta\lambda) - \frac{2\ell\pi i}{\log q}\Bigr)^k}%
    {s=\log_q (\zeta\lambda) + \frac{2\ell\pi i}{\log q}}
  \end{equation*}
  for a suitable function $\calD$.

  Then
  \begin{equation}\label{eq:proposition:symmetric-eigenvalues:sum-fluct}
    \sum_{\zeta\in U_p} N^{\log_q (\zeta\lambda)} (\log_q N)^k
    \Phi_{(\zeta\lambda)}(\fractional{\log_q N})
    = N^{\log_q \lambda} (\log_q N)^k \Phi(p\fractional{\log_{q^p} N})
  \end{equation}
  with a continuous $p$-periodic function
  \begin{equation*}
    \Phi(u)
    = \sum_{\ell\in\Z}\varphi_{\ell}\fexp[\Big]{\frac{2\ell\pi i}{p} u}
  \end{equation*}
  whose Fourier coefficients are
  \begin{equation*}
    \varphi_{\ell}
    =\Res[\bigg]{\calD(s)
      \Bigl(s - \log_q \lambda - \frac{2\ell\pi i}{p\log q}\Bigr)^k}%
    {s=\log_q \lambda + \frac{2\ell\pi i}{p\log q}}.
  \end{equation*}
\end{proposition}

Note that we again write $\Phi(p\fractional{\log_{q^p} N})$ to
optically emphasise the $p$-periodicity. Moreover, the factor
$(\log_q N)^k$ in~\eqref{eq:proposition:symmetric-eigenvalues:sum-fluct}
could be cancelled, however it is there to
optically highlight the similarities to the main results (e.g.\@
Theorem~\ref{theorem:simple}).
The proof of Proposition~\ref{proposition:symmetric-eigenvalues}
can be found in
Section~\ref{sec:proof-symmetric-eigenvalues}.

The above proposition will be used for proving
Corollary~\ref{corollary:transducer-main} which deals with transducer
automata; there, the second order term exhibits a fluctuation with
possible period larger than~$1$. We will also use the proposition for
the analysis of esthetic numbers in
Section~\ref{sec:esthetic-numbers}.

\begin{remark}
We can view Proposition~\ref{proposition:symmetric-eigenvalues} from
a different perspective:
A $q$-regular sequence is $q^p$-regular as well
(by~\cite[Theorem~2.9]{Allouche-Shallit:1992:regular-sequences}).
Then, all eigenvalues $\zeta\lambda$ of the original sequence
become eigenvalues $\lambda^p$ whose algebraic multiplicity is the sum
of the individual multiplicities but the sizes of the corresponding
Jordan blocks do not change.
Moreover, the joint spectral radius is also taken to the $p$th power.
We apply, for example,
Theorem~\ref{theorem:simple} in our $q^p$-world and get again
$1$-period fluctuations.
Note that for actually computing the Fourier coefficients,
the approach presented in the proposition seems to be more suitable.
\end{remark}

\section{Remarks on the Definitions}\label{sec:motivation-definitions}
In this section, we give some motivation for and comments on the definitions listed in
Section~\ref{sec:definitions-notations}.

\subsection{\texorpdfstring{$q$}{q}-Regular Sequences vs.\ Matrix Products}\label{section:q-regular-matrix-product}
We note one significant difference between the study of $q$-regular sequences
as in \eqref{eq:linear-representation} and the study of matrix
products~\eqref{eq:f-as-product}.
The recurrence \eqref{eq:linear-representation} is supposed to hold for
$qn+r=0$, too; i.e. $v(0)=A_0v(0)$. This implies that $v(0)$ is either the zero
vector (which is not interesting at all) or that $v(0)$ is a right eigenvector of
$A_0$ associated with the eigenvalue~$1$.

We do not want to impose this condition in the study of the matrix
product~\eqref{eq:f-as-product}. Therefore, we exclude the case $qn+r=0$ in
\eqref{eq:regular-matrix-sequence}.
This comes at the price of the terms $K$,
$\vartheta_m$, $\vartheta$ in Theorem~\ref{theorem:contribution-of-eigenspace} which
vanish if multiplied by a right eigenvector to the eigenvalue $1$ of $A_0$ from the
right. This is the reason why Theorem~\ref{theorem:simple} has simpler
expressions than those encountered in Theorem~\ref{theorem:contribution-of-eigenspace}.

\subsection{Joint Spectral Radius}
Let
\begin{equation*}
\rho_\ell\coloneqq \sup
\setm[\big]{\norm{A_{r_1}\ldots A_{r_\ell}}^{1/\ell}}{r_1, \ldots, r_\ell\in\set{0, \ldots, q-1}}.
\end{equation*}
Then the
submultiplicativity of the norm and Fekete's subadditivity lemma~\cite{Fekete:1923:ueber-verteil} imply that
$\lim_{\ell\to\infty}\rho_\ell=\inf_{\ell>0}\rho_{\ell}=\rho$;
cf.~\cite{Rota-Strang:1960}. In view of equivalence of norms, this shows that
the joint spectral radius does not depend on the chosen norm. For our purposes,
the important point is that the choice of $R$ ensures that there is an
$\ell_0>0$ such that $\rho_{\ell_0}\le R$, i.e., $\norm{A_{r_1}\ldots
A_{r_{\ell_0}}}\le R^{\ell_0}$ for all $r_j\in\set{0,\ldots, q-1}$. For any $\ell>0$, we use long division to write
$\ell=s\ell_0+r$, and by submultiplicativity of the norm, we get $\norm{A_{r_1}\ldots
    A_{r_\ell}}\le R^{s\ell_0} \rho_{r}^r$ and thus
\begin{equation}\label{eq:bound-prod}
  \norm{A_{r_1}\ldots
    A_{r_\ell}}=\Oh{R^{\ell}}
\end{equation}
for all $r_j\in\set{0,\ldots,q-1}$ and $\ell\to\infty$. We will only use
\eqref{eq:bound-prod} and no further properties of the joint spectral radius.
Note that~\eqref{eq:f-as-product} and \eqref{eq:bound-prod} imply that
\[f(n)=\Oh{R^{\log_q n}}=\Oh{n^{\log_q R}}\]
for $n\to\infty$.

As mentioned, we say that the set of matrices $A_0$, \dots, $A_{q-1}$,
has the \emph{finiteness property} if there is an $\ell>0$ with
$\rho_\ell=\rho$; see~\cite{Jungers:2009:joint-spectral-radius,
  Lagarias-Wang:1995:finiteness-conjecture-jsr}.

\subsection{Constants for Theorem~\ref{theorem:contribution-of-eigenspace}}\label{section:constants-for-theorem}
In contrast to usual conventions, we write matrix representations of
endomorphisms as multiplications $x\mapsto xM$ where $x$ is a (row) vector in
$\C^d$ and $M$ is a matrix. Note that we usually denote this endomorphism
by the corresponding calligraphic letter, for example, the endomorphism represented
by the matrix~$M$ is denoted by $\calM$.

Consider the endomorphism $\calC$ which maps a row vector $x\in\C^d$ to $xC$
and its generalised eigenspaces $W_\lambda$ for $\lambda\in\C$.  (These are the
generalised left eigenspaces of $C$. If $\lambda\notin\sigma(C)$, then
$W_\lambda=\set{0}$.) Then it is well-known that $\calC\rvert_{W_\lambda}$ is an
endomorphism of $W_\lambda$ and that
$\C^d=\bigoplus_{\lambda\in\sigma(C)}W_\lambda$. Let $\calT$ be the basis formed by
the rows of $T$. Then the matrix representation of $\calC$
with respect to~$\calT$ is $J$.

Let now $\calD$ be the endomorphism of $\C^d$ which acts as identity on
$W_\lambda$ for $\lambda\neq 1$ and as zero on $W_1$. Its matrix representation
with respect to the basis $\calT$ is $D$; its matrix representation with
respect to the standard basis is $T^{-1}DT$.

Finally, let $\calC'$ be the endomorphism $\calC'=\calC \circ \calD$. As
$\calC$ and $\calD$ decompose along
$\C^d=\bigoplus_{\lambda\in\sigma(C)}W_\lambda$ and $\calD$ commutes with every
other endomorphism on $W_\lambda$ for all $\lambda$, we clearly also have
$\calC'=\calD\circ\calC$. Thus the matrix representation of $\calC'$ with
respect to $\calT$ is $DJ=JD$; its matrix representation with respect to the
standard basis is $T^{-1}DJT=C'$.

Now consider a generalised left eigenvector $w$ of $C$.  If it is associated to the eigenvalue $1$, then $w
T^{-1}DT=\calD(w)=0$, $wK=0$ and $wC'=\calC'(w)=0$. Otherwise, that is, if $w$ is associated to an eigenvalue not equal to $1$, we have $wT^{-1}DT=\calD(w)=w$,
$wC'=\calC'(w)=\calC(w)=wC$,
$w{C'}^j={\calC'}^j(w)=\calC^j(w)=wC^j$ for $j\ge 0$ and $w\vartheta_m=0$. Also note that
$1$ is not an eigenvalue of $C'$, thus $I-C'$ is indeed regular.
If $1$ is not an eigenvalue of $C$, then everything is simpler:
$D$ is the identity matrix, $C'=C$, $K=(I-C)^{-1}(I-A_0)$ and $\vartheta=0$.

\part{Examples}\label{part:examples}

In this part we investigate three examples in-depth. For an overview,
we refer to Section~\ref{sec:overview-examples} where some of the
appearing phenomena are discussed as well. Further
examples are also mentioned there.

\section{Sequences Defined by Transducer Automata}
\label{sec:transducer}
We discuss the asymptotic analysis related to transducers; see also
Section~\ref{sec:overview-transducers} for an overview.

\input{transducer}

\section{Esthetic Numbers}
\label{sec:esthetic-numbers}
We discuss the asymptotic analysis of esthetic numbers; see also
Section~\ref{sec:overview-esthetic} for an overview.

\input{esthetic-numbers}

\section{Pascal's Rhombus}
\label{sec:pascal}
We discuss the asymptotic analysis of odd entries in Pascal's rhombus; see also
Section~\ref{sec:overview-pascal} for an overview.

\input{pascal-rhombus}

\part{Proofs}\label{part:proofs}

Before reading this part on the collected proofs, it is recommended to
recall the definitions and notations of
Section~\ref{sec:definitions-notations}. Some additional notations which are only used in the proofs are
introduced in the following section.

\section{Additional Notations}\label{additional-notation}
We use Iverson's convention $\iverson{\mathit{expr}}=1$ if
$\mathit{expr}$ is true and $0$ otherwise, which was popularised
by Graham, Knuth and Patashnik~\cite{Graham-Knuth-Patashnik:1994}.
We use the notation
$z^{\underline{\ell}}\coloneqq z(z-1)\dotsm (z-\ell+1)$ for falling factorials.
We use $\binom{n}{k_1, \dotsc, k_r}$ for multinomial coefficients. We sometimes
write a binomial coefficient $\binom{n}{a}$ as $\bibinom{n}{a}{b}$ with $a+b=n$ when we want
to emphasise the symmetry and analogy to a multinomial coefficient.

\section{Decomposition into Periodic Fluctuations: Proof of Theorem~\ref{theorem:contribution-of-eigenspace}}
\label{section:proof-contribution-of-eigenspace}

We first give an overview over the proof.

\begin{proof}[Overview of the Proof of Theorem~\ref{theorem:contribution-of-eigenspace}]
The first step will be
to express the summatory function $F$ in terms of the matrices $C$, $B_r$ and
$A_r$. Essentially, this corresponds to the fact that the summatory function of
a $q$-regular function is again $q$-regular. This expression of $F$ will
consist of two terms: the first is a sum over $0\le j<\log_q N$
involving a $j$th power of $C$ and matrices $B_r$ and $A_r$ depending on the
$\ell-j$ most significant digits of $N$. The second term is again a sum, but
does not depend on the digits of $N$; it only encodes the fact that $f(0)=A_0
f(0)$ may not hold. The fact that we are interested in $wF(N)$ for the
generalised left eigenvector~$w$ corresponding to the eigenvalue~$\lambda$
allows to express $wC^j$ in terms of $w\lambda^j$ (plus some other terms if $w$
is not an eigenvector).

The second term can be disposed of by elementary observations using a geometric
series. We reverse the order of summation in the first summand and extend it to
an infinite sum. The infinite sum is written in terms of periodic fluctuations;
the difference between the infinite sum and the finite sum is absorbed by the
error term.
In order not to have to deal with ambiguities due to non-unique
$q$-ary expansions of real numbers, we define the fluctuations on an infinite
product space instead of the unit interval.
\end{proof}

\subsection{Upper Bound for Eigenvalues of~\texorpdfstring{$C$}{C}}
We start with an upper bound for the eigenvalues of $C$ in terms of the joint
spectral radius.
\begin{lemma}\label{lemma:eigenvalue-spectral-radius-bound}
  Let $\lambda\in\sigma(C)$. Then $\abs{\lambda}\le q\rho$.
\end{lemma}
\begin{proof}
  For $\ell\to\infty$, we have
  \begin{equation*}
    \abs{\lambda}^\ell
    \le \max \setm{\abs{\lambda}}{\lambda \in \sigma(C)}^\ell
    = \Oh[\big]{\norm{C^\ell}}
  \end{equation*}
  and
  \begin{equation*}
    \norm{C^\ell}
    \le \sum_{0\le r_1, \ldots, r_\ell<q}
    \norm{A_{r_1}\dotsm A_{r_\ell}}
    = \Oh{q^\ell R^\ell}
  \end{equation*}
  by \eqref{eq:bound-prod}. Taking $\ell$th roots and the limit $\ell\to\infty$
  yields $\abs{\lambda}\le qR$. This last inequality does not depend on
  our particular (cf.\@ Section~\ref{sec:definitions-notations}) choice
  of $R>\rho$, so the inequality is valid for all $R>\rho$, and
  we get the result.
\end{proof}

\subsection{Explicit Expression for the Summatory Function}
In this section, we give an explicit formula for $F(N)=\sum_{0\le n<N} f(n)$ in
terms of the matrices $A_r$, $B_r$ and $C$.

\begin{lemma}\label{lemma:explicit-summatory}
  Let $N$ be an integer with $q$-ary expansion
  $r_{\ell-1}\ldots r_0$. Then
  \begin{equation*}
    F(N)=\sum_{0\le j<\ell} C^j B_{r_j} A_{r_{j+1}}\dotsm
    A_{r_{\ell-1}} + \sum_{0\le j<\ell} C^j (I-A_0).
  \end{equation*}
\end{lemma}
\begin{proof}
  We claim that
  \begin{equation}\label{eq:sum-recursion}
    F(qN+r)=C F(N) + B_r f(N) + (I-A_0)\iverson{qN+r > 0}
  \end{equation}
  holds for non-negative integers $N$ and $r$  with $0\le r<q$.

  We now prove \eqref{eq:sum-recursion}: Using
  \eqref{eq:regular-matrix-sequence} and $f(0)=I$ yields
  \begin{align*}
    F(qN+r)
    &= f(0)\, \iverson{qN+r > 0}
    + \sum_{\substack{0<qn+r'<qN+r\\0\le n\\ 0\le r'<q}} f(qn+r')\\
    &= f(0)\, \iverson{qN+r > 0}
    + \sum_{\substack{0<qn+r'<qN+r\\0\le n\\ 0\le r'<q}} A_{r'}f(n)\\
    &= \bigl(f(0)-A_{0}f(0)\bigr) \iverson{qN+r > 0}
    + \sum_{\substack{0\le qn+r'<qN+r\\0\le n\\ 0\le r'<q}} A_{r'}f(n)\\
    &= (I-A_0) \iverson{qN+r > 0}
    + \sum_{0\le n<N}\sum_{0\le r'<q} A_{r'}f(n)
    + \sum_{0\le r'<r} A_{r'}f(N)\\
    &= (I-A_0) \iverson{qN+r > 0}
    + CF(N)+B_{r}f(N).
  \end{align*}
  This concludes the proof of \eqref{eq:sum-recursion}.

  Iteration of \eqref{eq:sum-recursion} and using~\eqref{eq:f-as-product} yield the assertion of the lemma;
  cf.~\cite[Lemma~3.6]{Heuberger-Kropf-Prodinger:2015:output}.
\end{proof}

\subsection{Proof of Theorem~\ref{theorem:contribution-of-eigenspace}}

\begin{proof}[Proof of Theorem~\ref{theorem:contribution-of-eigenspace}]
  For readability, this proof is split into several
  steps.

  \proofparagraph{Setting}
  Before starting the actual proof, we introduce the setting
  using an infinite product space
  which will be used
  to define the fluctuations $\Phi_k$. We also introduce the maps
  linking the infinite product space to the unit interval.

  We will first introduce
  functions $\Psi_k$ defined on the infinite product space
  \begin{equation*}
    \Omega\coloneqq
    \setm[\big]{\bfx=(x_0, x_1, \ldots)}{
      x_j\in\set{0, \ldots,q-1} \text{ for $j\ge 0$}, x_0\neq 0}.
  \end{equation*}
  We equip it with the metric such that two elements~$\bfx\neq\bfx'$ with
  a common prefix of length~$j$ and $x_j\neq x'_j$ have distance~$q^{-j}$.
  We consider the map $\val\colon \Omega\to [0, 1]$ with
  \begin{equation*}
    \val(\bfx) \coloneqq \log_q\sum_{j\ge 0}x_jq^{-j};
  \end{equation*}
  see Figure~\ref{fig:commutative-diagram}. By using the assumption that the
  zeroth component of elements of $\Omega$ is assumed to be non-zero, we easily check that $\val$ is
  Lipschitz-continuous; i.e.,
  \begin{equation}\label{eq:lval:Lipschitz}
    \abs[\big]{\val(\bfx)-\val(\bfx')} = \Oh{q^{-j}}
  \end{equation}
  for $\bfx\neq\bfx'$ with a common prefix of length~$j$.

  \begin{figure}[htbp]
    \centering
    \begin{tikzcd}
      \Omega \arrow{rr}{\Psi}\arrow[shift left=0.5ex]{rd}{\val}&&\C^d\\
      &{[0, 1]}\arrow{ru}{\Phi}\arrow[dotted,shift left=0.5ex]{lu}{\repr}
    \end{tikzcd}
    \caption{Maps in the proof of Theorem~\ref{theorem:contribution-of-eigenspace}.}
    \label{fig:commutative-diagram}
  \end{figure}
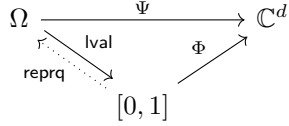
  For $y\in[0, 1)$, let $\repr(y)$ be the unique $\bfx\in\Omega$ with
  $\val(\bfx)=y$ such that $\bfx$ does not end on infinitely many digits~$q-1$, i.e.,
  $\repr(y)$ represents a $q$-ary expansion of $q^y$. This means that
  $\val\circ\repr$ is the identity on $[0, 1)$.

  From the definition of the metric on $\Omega$,
  recall that a function
  $\Psi\colon \Omega\to\C^d$ is continuous if and only if for each
  $\varepsilon>0$, there is a $j$ such that
  $\norm{\Psi(\bfx')-\Psi(\bfx)}<\varepsilon$ holds for all $\bfx$ and
  $\bfx'$ that have a common prefix of length $j$.
  Further recall from the universal property of quotients that if such a continuous function $\Psi$ satisfies
  $\Psi(\bfx)=\Psi(\bfx')$ whenever $\val(\bfx)=\val(\bfx')$, then there is a
  unique continuous function $\Phi\colon [0, 1]\to\C^d$ such that $\Phi\circ\val=\Psi$.
  This will be used in the ``Descent''-step of the proof.

  \proofparagraph{Notation}
  We will deal with the two sums in Lemma~\ref{lemma:explicit-summatory}
  separately. We will first introduce notations corresponding to this split
  and to the eigenvector structure.

  Let $N$ have the $q$-ary expansion
  $r_{\ell-1}\ldots r_0$ and set
  \begin{equation*}
    F_1(N) \coloneqq \sum_{0\le j<\ell} C^j B_{r_j} A_{r_{j+1}}\ldots
    A_{r_{\ell-1}}, \qquad F_2(N) \coloneqq \sum_{0\le j<\ell} C^j(I-A_0)
  \end{equation*}
  so that $F(N)=F_1(N)+F_2(N)$ by Lemma~\ref{lemma:explicit-summatory}.

  We consider the Jordan chain $w=v_{0}'$, \ldots, $v_{m-1}'$ generated by $w$,
  i.e., $v_k'=w(C-\lambda I)^k$ for $0\le k<m$ and $v_{m-1}'$ is a left eigenvector
  of $C$. Thus we have $wC^j=\sum_{0\le
    k<m}\binom{j}{k}\lambda^{j-k}v_k'$  for all $j\ge 0$.
  If $\lambda\neq 0$, choose vectors $v_0$, \ldots, $v_{m-1} \in \C^d$ such that
  \begin{equation}\label{eq:C-sum-eigenvectors}
    wC^j=\lambda^j\sum_{0\le k<m}j^kv_k
  \end{equation}
  holds for all $j\ge 0$. These vectors are
  suitable linear combinations of the vectors $v_0'$, \ldots,
  $v_{m-1}'$. We note that we have
  \begin{equation}\label{eq:v_m-1-expression}
    v_{m-1}=\frac1{\lambda^{m-1}(m-1)!}v_{m-1}'.
  \end{equation}

  \proofparagraph{Second Summand}
  We will now rewrite $wF_2(N)$ by evaluating the geometric sum
  and rewriting it in terms of a fluctuation.

  We claim that
  \begin{multline}\label{eq:constant-term}
    wF_2(N) = wK + N^{\log_q \lambda}\sum_{0\le k<m} (\log_q
    N)^k\Phi^{(2)}_{k}(\fractional{\log_q N}) \\
    + (\log_q N)^mw\vartheta_m
    + \iverson{\lambda = 0} \Oh{N^{\log_q R}}
  \end{multline}
  for suitable continuously differentiable functions $\Phi^{(2)}_{ k}$ on $\R$,
  $0\le k<m$. If $R=0$, then $\Oh{N^{\log_q R}}$ shall mean that the error
  vanishes for almost all $N$.

  Consider first the case that $\lambda \neq 1$.
  Because of $wC^j=w{C'}^j$ and $wT^{-1}DT=w$ (see Section~\ref{section:constants-for-theorem}) we have
  \begin{align*}
    wF_2(N)&=\sum_{0\le j<\ell}w{C'}^j \bigl(I-A_0\bigr)\\
           &=w\bigl(I-{C'}^\ell\bigr)\bigl(I-C'\bigr)^{-1}\bigl(I-A_0\bigr)
           = wK - wC^\ell \bigl(I-C'\bigr)^{-1}\bigl(I-A_0\bigr).
  \end{align*}
  If $\lambda=0$, then $wC^\ell=0$ for almost all $\ell$. We may set
  $\Phi^{(2)}_k=0$ for $0\le k<m$ and \eqref{eq:constant-term} is shown.
  Otherwise, as
  we have $\ell-1=\floor{\log_q N}=\log_q N - \fractional{\log_q N}$ and
  by~\eqref{eq:C-sum-eigenvectors}, we can
  rewrite $wC^\ell$ as
  \begin{equation*}
    wC^\ell=\lambda^{\ell}\sum_{0\le k'<m}\ell^{k'}
    v_{k'}=\lambda^{1+\log_q N-\fractional{\log_q N}}\sum_{0\le k'<m}(\log_q
    N+1-\fractional{\log_q N})^{k'} v_{k'}.
  \end{equation*}
  Let
  \begin{equation*}
    G_2(L, \nu)\coloneqq-\lambda^{1-\nu}\sum_{0\le k'<m}(L+1-\nu)^{k'} v_{k'}(I-C')^{-1}(I-A_0)
  \end{equation*}
  for reals $L$ and $\nu$,
  i.e.,
  \begin{equation*}
    wF_2(N)=wK + \lambda^{\log_q N} G_2(\log_q N, \fractional{\log_q N}).
  \end{equation*}
  By the binomial theorem, we have
  \begin{equation*}
    G_2(L, \nu)=-\lambda^{1-\nu}\sum_{0\le k<m}L^k\sum_{\substack{0\le r \\ k+r<m}}\bibinom{k+r}{k}{r}(1-\nu)^r v_{k+r}(I-C')^{-1}(I-A_0).
  \end{equation*}
  This leads to a representation
  $G_2(L, \nu)=\sum_{0\le k<m}L^k\Phi^{(2)}_{ k}(\nu)$ for continuously differentiable functions
  \begin{equation*}
    \Phi_k^{(2)}(\nu)=-\lambda^{1-\nu}\sum_{\substack{0\le r <m-k}}\bibinom{k+r}{k}{r}(1-\nu)^r v_{k+r}(I-C')^{-1}(I-A_0)
  \end{equation*}
  for $0\le k<m$. As the functions~$\Phi^{(2)}_{k}$ are continuously differentiable,
  they are Lipschitz
  continuous on compact subsets of $\R$. We note that in the case $k=m-1$, the
  only occurring summand is for $r=0$, which implies that
  \begin{equation}\label{eq:fluctuation-2-m-1}
    \Phi_{m-1}^{(2)}(\nu) = -\lambda^{1-\nu}v_{m-1}(I-C')^{-1}(I-A_0).
  \end{equation}
  Rewriting $\lambda^{\log_q N}$ as
  $N^{\log_q \lambda}$ and recalling that $w\vartheta_m=0$ yields \eqref{eq:constant-term} for $\lambda\neq 1$.

  We now turn to the case $\lambda=1$. We use $wC^j=\sum_{0\le
    k<m}\binom{j}{k}v_k'$ for $j\ge 0$ as above.
  Thus
  \begin{align*}
    wF_2(N) &= \sum_{0\le j<\ell}\sum_{0\le k<m}\binom{j}{k}v'_{k}(I-A_0)\\
    &= \sum_{0\le k<m}v'_k(I-A_0)\sum_{0\le j<\ell}\binom{j}{k}\\
    &= \sum_{0\le k<m}v'_k(I-A_0) \binom{\ell}{k+1},
  \end{align*}
  where the identity \cite[(5.10)]{Graham-Knuth-Patashnik:1994} (``summation on the upper index'')
  has been used in the last step.

  Thus $wF_2(N)$ is a polynomial in $\ell$ of degree $m$. By writing
  $\ell=1+\log_qN-\fractional{\log_q N}$, we can again rewrite this as a
  polynomial in $\log_q N$ whose coefficients depend on $\fractional{\log_q N}$.
  The coefficient of $(\log_q N)^m$ comes from
  $v_{m-1}'(I-A_0)\binom{\ell}{m}$, therefore, this coefficient is
  \begin{equation*}
    \frac1{m!}v_{m-1}'(I-A_0)=\frac1{m!}w(C-I)^{m-1}(I-A_0)=w\vartheta_m.
  \end{equation*}
  The additional factor $T^{-1}(I-D)T$ in $\vartheta_m$ has been introduced in order
  to annihilate generalised eigenvectors to other eigenvalues.  By construction
  of $K$, we have $wK=0$. Thus we have shown \eqref{eq:constant-term} for
  $\lambda=1$, too.

  \proofparagraph{Lifting the Second Summand}
  For later use---at this point, this may seem to be quite artificial---we
  set $\Psi^{(2)}_{k}=\Phi^{(2)}_{k}\circ \val$. As $\Phi^{(2)}_{k}$ is continuously differentiable, it
  is Lipschitz continuous on $[0, 1]$. As $\val$ is also Lipschitz continuous,
  so is $\Psi_k^{(2)}$.

  \proofparagraph{First Summand}
  We now turn to $wF_1(N)$. To explain our plan, assume that $w$ is in fact an
  eigenvector. Then $wF_1(N)=\sum_{0\le j<\ell}\lambda^j
  wB_{r_j}A_{r_{j+1}}\ldots A_{r_{\ell-1}}$. For $\abs{\lambda}\le R$, it will
  be rather easy to see that the result holds. Otherwise,
  we will factor out $\lambda^\ell$ and
  write the sum as $wF_1(N)=\lambda^{\ell} \sum_{0\le
    j<\ell}\lambda^{-(\ell-j)}wB_{r_j}A_{r_{j+1}}\ldots A_{r_{\ell-1}}$.
  We will then reverse the order of summation and extend the sum to an infinite
  sum, which will be represented by periodic fluctuations. The difference
  between the finite and the infinite sums will be absorbed by the error term.
  The periodic fluctuations will be defined on the infinite product space $\Omega$.

  We now return to the general case of a generalised eigenvector $w$ and the
  actual proof.
  If $\lambda=0$, we certainly have $\abs{\lambda}\le R$ and we are in one of
  the first two cases of this theorem. Furthermore, we have
  $wC^j=0$ for $j\ge m$, thus
  \begin{equation*}
    wF_1(N)=\Oh[\bigg]{\sum_{0\le j<m} R^{\ell-j}}
    = \Oh{R^{\ell}}
    = \Oh{N^{\log_q R}}
  \end{equation*}
  by using~\eqref{eq:bound-prod}.
  Together with~\eqref{eq:constant-term}, the result follows.

  From now on, we may assume that $\lambda\neq 0$.
  By using~\eqref{eq:C-sum-eigenvectors}, we have
  \begin{equation}\label{eq:w-F-1-n}
    wF_1(N)=\sum_{0\le j<\ell} \lambda^j\biggl(\sum_{0\le k<m}j^kv_k \biggr)B_{r_j} A_{r_{j+1}}\ldots
    A_{r_{\ell-1}}.
  \end{equation}
  We first consider the case that $\abs{\lambda}<R$ (corresponding to
  Theorem~\ref{theorem:contribution-of-eigenspace}, \itemref{item:small-eigenvalue}). We get
  \begin{align*}
    wF_1(N) &= \Oh[\bigg]{\sum_{0\le j<\ell}\abs{\lambda}^j j^{m-1}
    R^{\ell-j}} \\
    &= \Oh[\bigg]{R^{\ell} \sum_{0\le j<\ell}
    j^{m-1}\Bigl(\frac{\abs{\lambda}}{R}\Bigr)^j}
    =\Oh{R^{\ell}}
    =\Oh{N^{\log_q R}},
  \end{align*}
  where \eqref{eq:bound-prod} was used.
  Together with \eqref{eq:constant-term}, the result follows.

  Next, we consider the case where $\abs{\lambda}=R$
  (Theorem~\ref{theorem:contribution-of-eigenspace}, \itemref{item:R-eigenvalue}).
  In that case, we get
  \begin{equation*}
    wF_1(N)=\Oh[\bigg]{\sum_{0\le j<\ell}\abs{\lambda}^j j^{m-1}
      R^{\ell-j}}
    = \Oh[\bigg]{R^\ell\sum_{0\le j<\ell}j^{m-1}}
    =\Oh{R^\ell \ell^m}.
  \end{equation*}
  Again, the result follows.

  From now on, we may assume that $\abs{\lambda}>R$. We set $Q\coloneqq
  \abs{\lambda}/R$ and note that $1<Q\le q$ by assumption and Lemma~\ref{lemma:eigenvalue-spectral-radius-bound}.
  We claim that there are continuous functions $\Psi^{(1)}_{k}$ on $\Omega$ for $0\le k<m$
  such that
  \begin{equation}\label{eq:first-term}
    wF_1(N) = N^{\log_q \lambda}\sum_{0\le k<m} (\log_q N)^k
    \f[\big]{\Psi^{(1)}_{k}}{\repr(\fractional{\log_q N})}
  \end{equation}
  and such that
  \begin{equation}\label{eq:quasi-Hoelder}
    \norm[\big]{\Psi^{(1)}_{ k}(\bfx)-\Psi^{(1)}_{ k}(\bfx')}=\Oh{j^{m-1} Q^{-j}}
  \end{equation}
  when the first $j$ entries of $\bfx$ and $\bfx'\in\Omega$ coincide.

    Write $N=q^{\ell-1+\fractional{\log_q N}}$ and let $\bfx=\repr(\fractional{\log_q N})$,
    i.e., $\bfx$ is the $q$-ary expansion of $q^{\fractional{\log_q N}}=N/q^{\ell-1}\in[1, q)$
    ending on infinitely many zeros. This means that $x_j=r_{\ell-1-j}$
    for $0\le j<\ell$ and $x_j=0$ for $j\ge \ell$.
    Reversing the order of summation in \eqref{eq:w-F-1-n} yields
    \begin{align*}
      wF_1(N)=\lambda^{\ell-1}\sum_{0\le j<\ell}\lambda^{-j}\biggl(\sum_{0\le k<m}(\ell-1-j)^kv_k \biggr)B_{x_j} A_{x_{j-1}}\ldots
    A_{x_0}.
    \end{align*}
    For $j\ge \ell$, we have $x_j=0$ and therefore $B_{x_j}=0$. Thus we may
    extend the sum to run over all $j\ge 0$, i.e.,
    \begin{equation*}
      wF_1(N)=\lambda^{\ell-1}\sum_{j\ge 0}\lambda^{-j}\biggl(\sum_{0\le k<m}(\ell-1-j)^kv_k \biggr)B_{x_j} A_{x_{j-1}}\ldots
    A_{x_0}.
    \end{equation*}
    We insert $\ell-1=\log_q N - \fractional{\log_q N}$ and obtain
    \begin{equation*}
      wF_1(N)=\lambda^{\log_q
        N}
      \f[\big]{G_1}{\log_q N, \repr(\fractional{\log_q N})}
    \end{equation*}
    where
    \begin{align*}
      G_1(L, \bfx)&=\lambda^{-\val(\bfx)}\sum_{j\ge 0}\lambda^{-j}\biggl(\sum_{0\le
        k<m}(L-\val(\bfx) - j)^kv_k \biggr)B_{x_j} A_{x_{j-1}}\ldots
    A_{x_0}\\
&=\lambda^{-\val(\bfx)}\sum_{j\ge 0}\lambda^{-j}\biggl(\sum_{\substack{0\le a,\ 0\le r,\
  0\le s\\a+r+s<m}}L^a
  (-j)^r \trinom{a+r+s}{a}{r}{s}\\&\hspace*{11.225em}\times\bigl(-\val(\bfx)\bigr)^{s}v_{a+r+s} \biggr)B_{x_j} A_{x_{j-1}}\ldots
    A_{x_0}
    \end{align*}
    for $L\in\R$ and $\bfx\in\Omega$. Note that in contrast to $G_2$, the second argument of $G_1$ is an element of
    $\Omega$ instead of $\R$.
    Collecting $G_1(L, \bfx)$ by powers of $L$, we get
    \begin{equation*}
      G_1(L, \bfx) = \sum_{0\le k<m} L^k \Psi^{(1)}_{ k}(\bfx)
    \end{equation*}
    where
    \begin{equation*}
      \Psi^{(1)}_k(\bfx) = \sum_{j\ge 0}\lambda^{-j}\sum_{0\le r<m-k}j^r
      \f[\big]{\psi_{kr}}{\val(\bfx)} B_{x_j}A_{x_{j-1}}\ldots A_{x_0}
    \end{equation*}
    for functions
    \begin{equation*}
      \psi_{kr}(\nu)=\lambda^{-\nu}
  (-1)^r\sum_{0\le s<m-k-r} \trinom{k+r+s}{k}{r}{s}(-\nu)^{s}v_{k+r+s}
    \end{equation*}
    which are continuously differentiable and therefore Lipschitz
    continuous on the unit interval.
    This shows \eqref{eq:first-term}.
    For $k=m-1$, only summands with $r=s=0$ occur, thus
    \begin{equation}\label{eq:fluctuation-1-m-1}
      \Psi_{m-1}^{(1)}(\bfx)=\sum_{j\ge 0}\lambda^{-j-\val(\bfx)}v_{m-1}B_{x_j}A_{x_{j-1}}\ldots A_{x_0}.
    \end{equation}

    Note that $\Psi^{(1)}_{ k}(\bfx)$ is majorised by
    \begin{equation*}
      \Oh[\bigg]{\sum_{j\ge 0} \abs{\lambda}^{-j} j^{m-1} R^{j}}
    \end{equation*}
    according to \eqref{eq:bound-prod}.
    We now prove \eqref{eq:quasi-Hoelder}. So let $\bfx$ and $\bfx'$ have a
    common prefix of length $i$. Consider the summand of $\Psi^{(1)}_k(\bfx)$ with index $j$.
    First consider the case that $j<i$. For all $r$, we have
    \begin{equation*}
      \norm[\big]{\f[\big]{\psi_{kr}}{\val(\bfx)}-\f[\big]{\psi_{kr}}{\val(\bfx')}}
      = \Oh{q^{-i}}
    \end{equation*}
    due to Lipschitz continuity of $\psi_{kr}\circ \val$.
    As the matrix product~$A_{x_{j-1}} \ldots A_{x_0}$
    is the same for $\bfx$ and $\bfx'$, the
    difference with respect to this summand is bounded by
    \begin{equation*}
      \Oh[\big]{\abs{\lambda}^{-j}j^{m-1}q^{-i}R^{j}}
      = \Oh{q^{-i}j^{m-1}Q^{-j}}.
    \end{equation*}
    Thus the total
    contribution of all summands with $j<i$ is $\Oh{q^{-i}}$.
    Any summand with $j \ge i$ is bounded by
    $\Oh[\big]{\abs{\lambda}^{-j}j^{m-1}R^{j}} = \Oh{j^{m-1}Q^{-j}}$,
    which leads to a total contribution of $\Oh{i^{m-1}Q^{-i}}$.
    Adding the two bounds leads to a bound of $\Oh{i^{m-1}Q^{-i}}$, as
    requested.

    \proofparagraph{Descent}
    As we have defined the periodic fluctuations $\Psi^{(1)}_k$ on the infinite
    product space $\Omega$, we now need to prove that the periodic fluctuation
    descends to a periodic fluctuation on the unit interval. To do so,
    we will verify that the
    values of the fluctuation coincide whenever
    sequences in the infinite product space
    correspond to the same real number in the interval.

    By setting $\Psi_k(\bfx)=\Psi^{(1)}_{k}(\bfx)+\Psi^{(2)}_{k}(\bfx)$, we obtain
    \begin{equation}\label{eq:w-F-n}
      wF(N)=wK + N^{\log_q\lambda} \sum_{0\le k<m}(\log_q
      N)^k\f[\big]{\Psi_k}{\repr(\fractional{\log_q N})} +(\log_q N)^mw\vartheta_m
    \end{equation}
    and
    \begin{equation}\label{eq:Psi-continuity}
      \norm{\Psi_{k}(\bfx)-\Psi_{k}(\bfx')}=\Oh{j^{m-1}Q^{-j}}
    \end{equation}
    whenever $\bfx$ and $\bfx'\in\Omega$ have a common prefix of length $j$.

    It remains to show that $\Psi_k(\bfx)=\Psi_k(\bfx')$ holds whenever
    $\val(\bfx)=\val(\bfx')$ or $\val(\bfx)=0$ and $\val(\bfx')=1$.

    Choose $\bfx$ and $\bfx'$ such that one of the above
    two conditions on $\val$ holds and such that $x_j=0$ for $j\ge j_0$ and
    $x'_j=q-1$ for $j\ge j_0$. Be aware that now the prefixes of
    $\bfx$ and $\bfx'$ of length $j_0$ do not coincide except for the trivial
    case $j_0=0$.

    Fix some $j\ge j_0$ and set $\bfx''$ to be
    the prefix of $\bfx'$ of length $j$, followed by infinitely many zeros.
    Note that we have $q^{\val(\bfx'')}=q^{\val(\bfx')}-q^{-(j-1)}$. Set
    $n=q^{j-1+\val(\bfx'')}$. By construction, we have
    $n+1=q^{j-1+\val(\bfx)+\iverson{\val(\bfx)=0}}$. This implies
    $\repr(\fractional{\log_q n})=\bfx''$ and
    $\repr(\fractional{\log_q(n+1)})=\bfx$. Taking the difference of
    \eqref{eq:w-F-n} for $n+1$ and $n$ yields
    \begin{multline*}
      wf(n)=(n+1)^{\log_q \lambda} \sum_{0\le k<m}\bigl(\log_q (n+1)\bigr)^k
      \Psi_k(\bfx)  - n^{\log_q \lambda} \sum_{0\le k<m} (\log_q n)^k
      \Psi_k(\bfx'') \\+\big((\log_q(n+1))^m-(\log_q n)^m\big)w\vartheta_m.
    \end{multline*}
    We estimate $n+1$ as $n(1+\Oh{1/n})$ and get
    \begin{equation}\label{eq:Tenenbaum-2}
      wf(n)=n^{\log_q \lambda }\sum_{0\le k<m} (\log_q n)^k
      \bigl(\Psi_k(\bfx)-\Psi_k(\bfx'')\bigr)
      + \Oh[\big]{n^{\log_q \abs{\lambda} -1}(\log n)^{m-1}}.
    \end{equation}
    We have $wf(n)=\Oh{R^j}=\Oh{R^{\log_q n}}=\Oh{n^{\log_q R}}$ by~\eqref{eq:f-as-product} and~\eqref{eq:bound-prod}. By \eqref{eq:Psi-continuity},
    \begin{equation*}
      \norm[\big]{\Psi_k(\bfx'')-\Psi_k(\bfx')}
      = \Oh[\big]{(\log n)^{m-1}n^{-\log_q Q}}
    \end{equation*}
    which is used below to replace $\bfx''$ by $\bfx'$.
    Inserting these estimates in \eqref{eq:Tenenbaum-2} and dividing by
    $n^{\log_q \lambda}$ yields
    \begin{equation}\label{eq:Tenenbaum-3}
      \sum_{0\le k<m}(\log_q n)^k\bigl(\Psi_k(\bfx')-\Psi_k(\bfx)\bigr)
      = \Oh[\big]{n^{-\log_q Q} (\log n)^{2m-2}}.
    \end{equation}
    Note that $\Psi_k(\bfx')-\Psi_k(\bfx)$ does not depend on $j$. Now we let
    $j$ (and therefore $n$) tend to infinity. We see that
    \eqref{eq:Tenenbaum-3} can only remain true if
    $\Psi_k(\bfx')=\Psi_k(\bfx)$ for $0\le k<m$, which we had set out to show.

    Therefore, $\Psi_k$ descends to a continuous function $\Phi_k$ on $[0, 1]$ with
    $\Phi_k(0)=\Phi_k(1)$; thus $\Phi_k$ can be extended to a $1$-periodic continuous
    function.

    \proofparagraph{Hölder Continuity}
    We will now prove Hölder continuity. As the fluctuations have been defined
    on the infinite product space $\Omega$, we will basically have to prove Hölder
    continuity there. The difficulty will be that points in the unit interval which
    are close to each other there may have drastically different $q$-ary
    expansions, thus correspond to drastically different points in the infinite
    product space $\Omega$. To circumvent this problem, the interval between
    the two points will be split into two parts.

    We
    first claim that for $0\le y<y'''<1$,
    we have
    \begin{equation}\label{eq:Hoelder-1}
      \norm[\big]{\Phi_k(y''')-\Phi_k(y)}
      = \Oh[\big]{(\log(q^{y'''}-q^{y}))^{m-1}(q^{y'''}-q^y)^{\log_q Q}}
    \end{equation}
    as $y'''\to y$.
    To prove this, let $\bfx\coloneqq \repr(y)$ and $\bfx'''\coloneqq
    \repr(y''')$.
    Let $\ell$ be the length of the longest common prefix of $\bfx$ and
    $\bfx'''$ and choose $j\ge 0$ such that $q^{-j}\le q^{y'''}-q^y< q^{-j+1}$.
    We define $\bfx'$ and $\bfx''\in\Omega$ such that
    \begin{alignat*}{4}
      \bfx&=(x_0,x_1,\ldots, x_{\ell-1}, x_{\ell},{}&& x_{\ell+1},{}&& x_{\ell+2},{}&& \ldots),\\
      \bfx'&=(x_0, x_1, \ldots, x_{\ell-1}, x_{\ell},{}&& q-1,{}&& q-1,{}&& \ldots),\\
      \bfx''&=(x_0, x_1, \ldots, x_{\ell-1}, x_{\ell}+1,{}&& 0,{}&& 0,{}&& \ldots),\\
      \bfx'''&=(x_0, x_1, \ldots, x_{\ell-1}, x'''_{\ell},{}&& x'''_{\ell+1},{}&&
      x'''_{\ell+2},{}&& \ldots)
    \end{alignat*}
    and set $y'\coloneqq\val(\bfx')$ and $y''\coloneqq\val(\bfx'')$.
    As $\val(\bfx)=y<y'''=\val(\bfx''')$, we have $x'''_\ell>x_\ell$. We
    conclude that $y\leq y'=y''\leq y'''$. Therefore,
    \begin{equation*}
      q^{y'}-q^{y} \le q^{y'''}-q^{y}< q^{-j+1},
    \end{equation*}
    so in view of the fact that each entry of
    $\bfx'$ is greater or equal than the corresponding entry of $\bfx$,
    the expansions $\bfx$ and $\bfx'$ must have a common prefix of length $j$.
    Similarly, the expansions $\bfx''$ and $\bfx'''$ must have a common prefix
    of length~$j$. Thus \eqref{eq:Psi-continuity} implies that
    \begin{align*}
      \norm[\big]{\Phi_k(y''')-\Phi_k(y)}
      &\le \norm[\big]{\Phi_k(y''')-\Phi_k(y'')}+
      \norm[\big]{\Phi_k(y')-\Phi_k(y)}\\
      &= \norm[\big]{\Psi_k(\bfx''')-\Psi_k(\bfx'')}+
      \norm[\big]{\Psi_k(\bfx')-\Psi_k(\bfx)}
      = \Oh{j^{m-1}Q^{-j}}.
    \end{align*}
    Noting that $-j = \log_q(q^{y'''}-q^y) + \Oh{1}$ leads to~\eqref{eq:Hoelder-1}.

    In order to prove Hölder continuity with exponent $\alpha<\log_q Q$, we first
    note that Lipschitz-continuity of $y\mapsto q^y$ on the interval $[0, 1]$
    shows that \eqref{eq:Hoelder-1} implies
    \begin{equation*}
      \norm[\big]{\Phi_k(y''')-\Phi_k(y)}
      = \Oh[\big]{(y'''-y)^{\alpha}}.
    \end{equation*}
    This can then easily be extended to arbitrary reals $y<y'''$
    by periodicity of $\Phi_k$ because it is sufficient to consider small
    $y'''-y$ and the interval may be subdivided at an integer between $y$ and $y'''$.

    \proofparagraph{Constant Dominant Fluctuation}
    To finally prove the final assertion on constant fluctuations, we will have
    to inspect the explicit expression for the fluctuations using the
    additional assumption.

    Under the additional
    assumption that the vector~$w(C-\lambda I)^{m-1}=v_{m-1}'$ is a left eigenvector to all
    matrices $A_0$, \ldots, $A_{q-1}$ associated with the eigenvalue $1$, the same holds for $v_{m-1}$
    by~\eqref{eq:v_m-1-expression}. Then $v_{m-1}$ is also a left eigenvector of
    $C$ associated with the eigenvalue $q$. In particular, $\lambda=q\neq 1$.

    We can compute $\Phi_{m-1}^{(2)}(\nu)$ using
    \eqref{eq:fluctuation-2-m-1}. As $v_{m-1}\in W_{q}$, we have $v_{m-1}C=v_{m-1}C'$
    by definition of $C'$ (see Section~\ref{section:constants-for-theorem}) which implies that
    $v_{m-1}(I-C')^{-1}=\frac1{1-q}v_{m-1}$. As $v_{m-1}(I-A_0)=0$ by
    assumption, we conclude that $\Phi_{m-1}^{(2)}(\nu)=0$ in this case.

    We use \eqref{eq:fluctuation-1-m-1} to compute $\Psi_{m-1}^{(1)}(\bfx)$.
    By assumption, $v_{m-1}B_{x_j}=x_j v_{m-1}$ which implies that
    \begin{equation*}
      \Psi_{m-1}^{(1)}(\bfx)
      = q^{-\val(\bfx)} \biggl(\sum_{j\ge 0}q^{-j}x_j\biggr) v_{m-1}
      =q^{-\val(\bfx)}q^{\val(\bfx)}v_{m-1}=v_{m-1}
    \end{equation*}
    by definition of $\val$.

    Together with \eqref{eq:v_m-1-expression}, we obtain the assertion.
\end{proof}

\subsection{Proof of Theorem~\ref{theorem:main}}\label{section:proof:corollary-main}

\begin{proof}[Proof of Theorem~\ref{theorem:main}]
  We denote the rows of $T$ as $w_1$, \ldots, $w_d$ and the columns of $T^{-1}$
  by $t_1$, \ldots, $t_d$. Thus $\sum_{1 \le j \le d} t_jw_j=I$ and $w_j$ is a
  generalised left eigenvector of $C$ of some rank $m_j$ corresponding to some
  eigenvalue
  $\lambda_j\in\sigma(C)$. Theorem~\ref{theorem:contribution-of-eigenspace} and
  the fact that there are no eigenvalues of $C$ of absolute value between
  $\rho$ and $R$
  then immediately imply that
  \begin{align*}
    F(N) &= \sum_{1 \le j \le d} t_j w_j F(N) \\
    &= K + \sum_{1 \le j \le d} (\log_q N)^{m_j} t_jw_j \vartheta_{m_j}\\
    &\phantom{= K}\; + \sum_{\substack{1\le j\le d\\\abs{\lambda_j}>\rho }} N^{\log_q \lambda_j}
    \sum_{0\le k<m_j}(\log_q N)^k t_j\Psi_{jk}(\fractional{\log_q N}) \\
    &\phantom{= K}\; + \iverson{\exists \lambda\in\sigma(C)\colon \abs{\lambda}\le\rho}
    \Oh[\big]{N^{\log_q R}(\log N)^{\max\set{0}\cup \setm{m_j}{\abs{\lambda_j}=R}}}
  \end{align*}
  for some  $1$-periodic Hölder continuous functions $\Psi_{jk}$ with exponent
  less than $\log_q\abs{\lambda_j}/R$. The first summand $K$ as well as the
  error term already coincide with the result stated in the theorem.
  From Section~\ref{section:constants-for-theorem} we recall that $w_j\vartheta_{m_j}=0$
  for $\lambda_j\neq 1$.

  We set
  \begin{equation*}
    \Phi_{\lambda k}(u)\coloneqq  \sum_{\substack{1\le j\le
        d\\\lambda_j=\lambda\\k<m_j}}
    \bigl(t_j\Psi_{jk}(u) +\iverson{\lambda=1}\iverson{m_j=k}t_jw_j\vartheta_{m_j}\bigr)
  \end{equation*}
  for $\lambda\in\sigma(C)$ with $\abs{\lambda}>\rho$ and $0\le k<m(\lambda)$.

  Then we still have to account for
  \begin{equation}\label{eq:phi-m-1-sum}
    (\log_q N)^{m(1)}\sum_{\substack{1\le j\le d\\\lambda_j=1\\m_j=m(1)}}t_jw_j\vartheta_{m(1)}.
  \end{equation}
  The factor $(C-I)^{m(1)-1}$ in the definition of $\vartheta_{m(1)}$ implies that
  $w_j\vartheta_{m(1)}$ vanishes unless $\lambda_j=1$ and $m_j=m(1)$. Therefore, the
  sum in \eqref{eq:phi-m-1-sum} equals $\vartheta$.
\end{proof}

\section{Meromorphic Continuation of the Dirichlet Series: Proof of
  Theorem~\ref{theorem:Dirichlet-series}}
\label{section:proof:Dirichlet-series}
For future use, we state an estimate for the binomial
coefficient. Unsurprisingly, it is a consequence of
a suitable version of Stirling's formula.
Alternatively, it can be seen as the most basic case of Flajolet and
Odlyzko's singularity
analysis~\cite[Proposition~1]{Flajolet-Odlyzko:1990:singul}, where
uniformity in $s$ is easily checked.

\begin{lemma}\label{lemma:binomial-coefficient-asymptotics}
  Let $k\in\Z$, $k\ge 0$. Then
  \begin{equation}\label{eq:binomial-coefficient-estimate}
    \abs[\bigg]{\binom{-s}{k}}\sim \frac{1}{\abs{\Gamma(s)}}k^{\Re s-1}
  \end{equation}
  uniformly for $s$ in a compact subset of $\C$ and $k\to\infty$.
\end{lemma}
\begin{proof}
  By \cite[(5.14)]{Graham-Knuth-Patashnik:1994} (``negating the upper index''), we rewrite the binomial coefficient as
  \begin{equation*}
    \binom{-s}{k}=(-1)^{k}\binom{s+k-1}{k}=\frac{(-1)^k}{\Gamma(s)}\frac{\Gamma(s+k)}{\Gamma(k+1)}.
  \end{equation*}
  Thus~\eqref{eq:binomial-coefficient-estimate} follows by \DLMF{5.11}{12}
  (which is an easy consequence of Stirling's formula for the Gamma function).
\end{proof}

\begin{proof}[Proof of Lemma~\ref{lemma:shifted-Dirichlet}]
  We have
  \begin{equation}\label{eq:Dirichlet-shifted:Sigma-as-diff}
    \f{\Sigma}{s, \beta, \calD}
    = \sum_{n\ge n_0} \bigl((n+\beta)^{-s}-n^{-s}\bigr) d(n)
  \end{equation}
  for $\Re s>\log_q R'+ 1$.
  We note that
  \begin{equation*}
    (n+\beta)^{-s} - n^{-s}
    = n^{-s}\Bigl(\Bigl(1+\frac{\beta}{n}\Bigr)^{-s} - 1 \Bigr)
    = \Oh[\big]{\abs{s}n^{-\Re s-1}}.
  \end{equation*}
  Therefore,
  \begin{equation*}
    \f{\Sigma}{s, \beta, \calD} = \Oh[\bigg]{\abs{s}\sum_{n\ge n_0}n^{\log_q R'-\Re s-1}},
  \end{equation*}
  and the series
  converges for $\Re s>\log_q R'$. As this holds for all $R'>\rho$, we obtain
  $\f{\Sigma}{s, \beta, \calD}=\Oh{\abs{\Im s}}$ as $\abs{\Im s}\to\infty$
  uniformly for $\log_q \rho + \delta \le \Re s \le \log_q \rho+\delta+1$.
  In the language of \cite[\S~III.3]{Hardy-Riesz:1915},
  $\f{\Sigma}{s, \beta, \calD}$ has order at most $1$ for $\log_q \rho + \delta \le \Re s \le \log_q
  \rho+\delta+1$. As $\log_q \rho+\delta+1$ is larger
  than the abscissa of absolute convergence of $\f{\Sigma}{s, \beta, \calD}$, it is clear that
  $\f{\Sigma}{s, \beta, \calD}=\Oh{1}$ for $\Re s=\log_q \rho+\delta+1$, i.e., $\f{\Sigma}{s, \beta, \calD}$ has order at most $0$ for
  $\Re s=\log_q \rho+\delta+1$. By Lindelöf's theorem
  (see \cite[Theorem~14]{Hardy-Riesz:1915}), we conclude that
  $\f{\Sigma}{s, \beta, \calD}=\Oh[\big]{\abs{\Im s}^{\mu_\delta(\Re s)}}$ for $\log_q \rho + \delta\le \Re s\le
  \log_q \rho +\delta+1$.

  For $\Re s > \log_q R' + 1$, we may
  rewrite~\eqref{eq:Dirichlet-shifted:Sigma-as-diff}
  using the binomial series as
  \begin{align}\label{eq:shifted-Dirichlet:diff:inner-sum}
    \f{\Sigma}{s, \beta, \calD} &=
    \sum_{n\ge n_0}{n^{-s}}\sum_{k\ge 1}\binom{-s}{k}\frac{\beta^k}{n^k} d(n)\notag\\
    &= \sum_{k\ge 1}
    \binom{-s}{k} \beta^k \sum_{n\ge n_0} n^{-(s+k)} d(n).
  \end{align}
  Switching the order of summation was legitimate because
  \begin{align*}
    \norm[\bigg]{\sum_{n\ge n_0} n^{-(s+k)} d(n)}
    &\le \sum_{n\ge n_0} n^{-(\Re s+k)} \norm{d(n)}\\
    &= \sum_{n\ge n_0} \Oh[\big]{n^{\log_q R'-\Re s -k}}
    = \Oh[\big]{n_0^{\log_q R'-\Re s-k+1}}
  \end{align*}
  for $\Re s+k>\log_q R'+1$ and Lemma~\ref{lemma:binomial-coefficient-asymptotics} imply absolute and
  uniform convergence for $s$ in a compact set.
  Noting that the previous arguments hold again for all $R'>\rho$ and that
  the inner sum in \eqref{eq:shifted-Dirichlet:diff:inner-sum}
  is $\calD(s+k)$ completes the proof.
\end{proof}

\begin{proof}[Proof of Theorem~\ref{theorem:Dirichlet-series}]
  As $f(n)=\Oh{R^{\log_q n}}=\Oh{n^{\log_q R}}$ by \eqref{eq:f-as-product} and
  \eqref{eq:bound-prod}, the Dirichlet series
  $\calF_{n_0}(s) = \sum_{n \ge n_0} n^{-s} f(n)$
  (see Section~\ref{sec:definitions-notations})
  converges absolutely and uniformly on compact sets for $\Re s>\log_q R+1$.
  As this holds for all $R>\rho$, i.e., does not depend on
  our particular (cf.\@ Section~\ref{sec:definitions-notations}) choice
  of $R>\rho$, this convergence result holds
  for $\Re s>\log_q \rho+1$.

  We use~\eqref{eq:regular-matrix-sequence} and
  Lemma~\ref{lemma:shifted-Dirichlet} (including its notation)
  to rewrite $\calF_{n_0}$ as
  \begin{align*}
    \calF_{n_0}(s) &= \sum_{n_0 \le n < qn_0}n^{-s}f(n) + \sum_{0 \le r < q} \sum_{n\ge n_0} (qn+r)^{-s} f(qn+r)\\
    &= \sum_{n_0 \le n < qn_0} n^{-s}f(n) + q^{-s} \sum_{0 \le r < q} A_r
    \sum_{n\ge n_0} \Bigl(n+\frac{r}{q}\Bigr)^{-s} f(n)\\
    &= \sum_{n_0 \le n < qn_0} n^{-s}f(n) + q^{-s}C\calF_{n_0}(s) + \calH_{n_0}(s)
  \end{align*}
  with
  \begin{equation*}
    \calH_{n_0}(s)\coloneqq q^{-s} \sum_{0 \le r < q} A_r \f[\big]{\Sigma}{s, \tfrac{r}{q}, \calF_{n_0}}
  \end{equation*}
  for $\Re s>\log_q R+ 1$.
  Thus
  \begin{equation}\label{eq:functional-equation-H}
    \bigl(I-q^{-s}C\bigr)\calF_{n_0}(s) = \sum_{n_0 \le n < qn_0}n^{-s}f(n)+\calH_{n_0}(s)
  \end{equation}
  for $\Re s>\log_q R+ 1$.
  By Lemma~\ref{lemma:shifted-Dirichlet} we have
  $\calH_{n_0}(s)=\Oh[\big]{\abs{\Im s}^{\mu_\delta(\Re s)}}$
  for $\log_q \rho + \delta\le \Re s\le
  \log_q \rho +\delta+1$.
  Rewriting the expression for $\calH_{n_0}(s)$ using the binomial series
  (see Lemma~\ref{lemma:shifted-Dirichlet} again) yields
  \begin{equation*}
    \calH_{n_0}(s)
    = q^{-s}\sum_{0 \le r < q} A_r \sum_{k\ge
      1}\binom{-s}{k}\Bigl(\frac{r}{q}\Bigr)^k \calF_{n_0}(s+k).
  \end{equation*}
  Combining this with~\eqref{eq:functional-equation-H}
  yields the expression~\eqref{eq:Dirichlet-recursion} for $\calG_{n_0}$.

  Solving \eqref{eq:analytic-continuation} for $\calF_{n_0}$ yields the
  meromorphic continuation of $\calF_{n_0}(s)$ to $\Re s>\log_q R$ (and thus to $\Re
  s>\log_q \rho$) with possible poles where
  $q^s$ is an eigenvalue of $C$. As long as $q^s$ keeps a fixed positive
  distance $\delta$ from the eigenvalues, the bound for $\calG_{n_0}$
  (coming from the bound for $\calH_{n_0}$) carries over
  to a bound for $\calF_{n_0}$, i.e., \eqref{eq:order-F}.

  To estimate the order of the poles, let $w$ be generalised left eigenvector
  of rank $m$ of $C$ corresponding to an eigenvalue $\lambda$ with $\abs{\lambda}>R$. We claim that
  $w\calF_{n_0}(s)$ has a pole of order at most $m$ at $s=\log_q \lambda+\chi_k$ and no other
  poles for $\Re s>\log_q R$. We prove this by induction on $m$.

  Set $v\coloneqq w(C-\lambda I)$. By definition, $v=0$ or $v$ is a generalised
  eigenvector of rank $m-1$ of $C$. By induction hypothesis, $v\calF_{n_0}(s)$ has a
  pole of order at most $m-1$ at $s=\log_q \lambda+\chi_k$ for $k\in\Z$ and no other
  poles for $\Re s>\log_q R$.

  Multiplying \eqref{eq:analytic-continuation} by $w$,
  inserting the definition of~$v$ and reordering the summands yields
  \begin{equation*}
    \bigl(1 - q^{-s}\lambda\bigr)w\calF_{n_0}(s) = q^{-s}v \calF_{n_0}(s) + w\calG_{n_0}(s).
  \end{equation*}
  The right-hand side has a pole of order at most $m-1$ at $\log_q \lambda+\chi_k$ for
  $k\in\Z$ and $1-q^{-s}\lambda$ has a simple zero at the same places. This
  proves the claim.
\end{proof}

\section{Fourier Coefficients: Proof of
  Theorem~\ref{theorem:use-Mellin--Perron}}
\label{section:proof:use-Mellin--Perron}
In contrast to the rest of this paper, this section does not directly relate to
a regular sequence but gives a general method to derive Fourier coefficients of
fluctuations.

\subsection{Pseudo-Tauberian Theorem}
\label{sec:pseudo-tauber}

In this section, we generalise the pseudo-Tau\-be\-rian argument by Flajolet, Grabner,
Kirschenhofer, Prodinger and
Tichy~\cite[Proposition~6.4]{Flajolet-Grabner-Kirschenhofer-Prodinger:1994:mellin}.
The basic idea is that for a $1$-periodic Hölder-continuous function
$\Phi$ and $\gamma\in\C$, there is a $1$-periodic continuously
differentiable function $\Psi$ such that
\begin{equation*}
  \sum_{1\le n<N} n^{\gamma} \Phi(\log_q n)
  = N^{\gamma+1} \Psi(\log_q N) + \oh{N^{\Re \gamma+1}},
\end{equation*}
and there is a straight-forward relation between the Fourier
coefficients of $\Phi$ and the Fourier coefficients of $\Psi$. This relation
exactly corresponds to the additional factor $s+1$ when transitioning
from the zeroth order Mellin--Perron formula to the first order
Mellin--Perron formula.

In contrast
to~\cite[Proposition~6.4]{Flajolet-Grabner-Kirschenhofer-Prodinger:1994:mellin},
we allow for an additional logarithmic factor, have weaker growth
conditions on the Dirichlet series and quantify the error. We also
extend the result to all complex $\gamma$. The generalisation from
$q=2$ there to our real~$q>1$ is trivial.

\begin{proposition}\label{proposition:pseudo-Tauber}
  Let $\gamma\in\C$ and $q>1$ be a real number,  $m$ be a
  positive integer, $\Phi_0$, \ldots, $\Phi_{m-1}$ be $1$-periodic Hölder continuous
  functions with exponent $\alpha>0$, and $0<\beta<\alpha$. Then there exist continuously differentiable functions
  $\Psi_{-1}$, $\Psi_{0}$, \ldots, $\Psi_{m-1}$, periodic with period $1$, and a constant $c$ such that
  \begin{multline}
    \sum_{1\le n< N}n^\gamma \sum_{\substack{j+k=m-1\\0\le j<m}}\frac{(\log n)^{k}}{k!}\Phi_j(\log_q n)\\
    =c + N^{\gamma+1}\sum_{\substack{j+k=m-1\\-1\le j<m}} \frac{(\log N)^{k}}{k!}\Psi_j(\log_q N)
    + \Oh[\big]{N^{\Re \gamma+1-\beta}}
    \label{eq:pseudo-Tauber-relation}
  \end{multline}
  for integers $N\to\infty$.

  Denote the Fourier coefficients of $\Phi_j$ and $\Psi_j$ by $\varphi_{j\ell}\coloneqq
  \int_0^1\Phi_j(u)\exp(-2\ell\pi i u)\, \dd u$ and $\psi_{j\ell}\coloneqq
  \int_0^1\Psi_j(u)\exp(-2\ell\pi i u)\, \dd u$, respectively.
  Then the corresponding generating functions fulfil
  \begin{equation}\label{eq:pseudo-Tauber-Fourier}
    \sum_{0\le j<m}\varphi_{j\ell}Z^j = \Bigl(\gamma+1+\frac{2\ell \pi i}{\log q} + Z\Bigr)\sum_{-1\le j<m}\psi_{j\ell}Z^j
    +\Oh{Z^m}
  \end{equation}
  for $\ell\in \Z$ and $Z\to 0$.

  If $q^{\gamma+1}\neq 1$, then $\Psi_{-1}$ vanishes.
\end{proposition}
\begin{remark}
  Note that the constant $c$ is absorbed by the error term if
  $\Re\gamma+1>\alpha$, in particular if $\Re\gamma>0$.
  Therefore, this constant does not occur in the
  article~\cite{Flajolet-Grabner-Kirschenhofer-Prodinger:1994:mellin}.
\end{remark}
\begin{remark}
  \label{remark:recurrence-fluctuation}
  The factor $\gamma+1+\frac{2\ell \pi i}{\log q} + Z$ in
  \eqref{eq:pseudo-Tauber-Fourier} will turn out to
  correspond exactly to the additional factor $s+1$ in the first order
  Mellin--Perron summation formula with the substitution
  $s=\gamma+\frac{2\ell\pi i}{\log q}+ Z$ such that the local expansion around
  the pole in $s=\gamma+\frac{2\ell\pi i}{\log q}$ of the Dirichlet generating
  function is conveniently written as a Laurent series in $Z$. See the proof of
  Theorem~\ref{theorem:use-Mellin--Perron} for details.
\end{remark}

Before actually proving Proposition~\ref{proposition:pseudo-Tauber},
we give an outline.

\begin{proof}[Overview of the Proof of Proposition~\ref{proposition:pseudo-Tauber}]
  We start with the left-hand side
  of~\eqref{eq:pseudo-Tauber-relation} and split the range of summation
  according to $\floor{\log_q n}$, thereby, in terms of our periodic functions,
  split after each period. We then use periodicity of the~$\Phi_j$
  and collect terms. This results in
  Riemann sums which converge to the corresponding
  integrals. Therefore, we can approximate these sums by the integrals.

  More rewriting constructs and reveals the functions~$\Psi_j$ (of the
  right-hand side of~\eqref{eq:pseudo-Tauber-relation}): These functions
  are basically defined via the above mentioned integral.
  We then show
  that these functions are indeed periodic and that their Fourier coefficients
  relate to the Fourier coefficients of the~$\Phi_j$.
  The latter is done by a direct computation of the integrals
  defining these coefficients.

  For this proof, we use an approach via exponential generating functions. 
  This reduces the overhead for
  dealing with the logarithmic factors $(\log n)^k$
  in~\eqref{eq:pseudo-Tauber-relation} such that we can essentially
  focus on the case~$m=1$.
  The resulting formula~\eqref{eq:pseudo-Tauber-relation}
  follows by extracting a suitable coefficient of this power series.

  There is another benefit of the generating function approach:
  This formulation allows to easily translate the relation between
  the Fourier coefficients here to the additional factors occurring
  when transitioning to higher order Mellin--Perron summation
  formul\ae{},
  in particular the factor $s+1$ in the first order Mellin--Perron summation.
\end{proof}
\begin{proof}[Proof of Proposition~\ref{proposition:pseudo-Tauber}]
  We split the proof into six parts.
  
  \proofparagraph{Notations}
  We start by defining quantities that are used through the whole proof.

  Without loss of generality, we assume that $q^{\Re \gamma+1}\neq q^{\alpha}$:
  otherwise, we slightly decrease $\alpha$ keeping the inequality
  $\beta<\alpha$ intact.
  We use the abbreviations $\Lambda\coloneqq \floor{\log_q N}$,
  $\nu\coloneqq \fractional{\log_q N}$, i.e.,
  $N=q^{\Lambda+\nu}$. We use the generating functions
  \begin{align*}
    \f{\Phi}{u, Z}&\coloneqq \sum_{0\le j<m}\Phi_j(u)Z^j,\\
    L(N, Z)&\coloneqq \sum_{1\le n<N}n^{\gamma+Z} \f{\Phi}{\log_q n, Z}=\sum_{1\le
             n<N}n^\gamma \fexp[\big]{(\log n) Z}\f{\Phi}{\log_q n, Z},\\
    Q(Z)&\coloneqq q^{\gamma+1+Z}
  \end{align*}
  for $0\le u\le 1$ and $0<\abs{Z}<2r$ where $r>0$ is chosen such that
  $r<(\alpha-\beta)/2$ and such that
  $Q(Z)\neq 1$ and $\abs{Q(Z)}\neq q^{\alpha}$ for these $Z$.
  (The condition $Z\neq 0$ is only needed for the case $q^{1+\gamma}=1$.)
  We will stick to the above choice of~$r$ and restrictions for~$Z$ throughout
  the proof.

  It is easily seen that  the left-hand side
  of~\eqref{eq:pseudo-Tauber-relation} equals $[Z^{m-1}]L(N, Z)$, where
  $[Z^{m-1}]$ denotes extraction of the coefficient of $Z^{m-1}$.

  \proofparagraph{Approximation of the Sum by an Integral}
  We will now rewrite $L(N, Z)$ so that its shape is that of a
  Riemann sum, therefore enabling us to approximate it by an integral.

  Splitting the range of summation with respect to powers of $q$ yields
  \begin{align*}
    L(N, Z) = \phantom{+\;}&
    \sum_{0\le p<\Lambda}\sum_{q^p\le n<q^{p+1}}n^{\gamma+Z}
    \f{\Phi}{\log_q n, Z} \\
    +\; &
    \sum_{q^\Lambda\le n< q^{\Lambda+\nu}}n^{\gamma+Z}\f{\Phi}{\log_q n, Z}.
  \end{align*}
  We write $n=q^px$ (or $n=q^\Lambda x$ for the second sum), use the
  periodicity of $\Phi$ in $u$ and get
  \begin{align*}
    L(N, Z) = \phantom{+\;}&
    \sum_{0\le p<\Lambda}Q(Z)^p\sum_{\substack{x\in q^{-p}\Z\\ 1\le x < q}}
     x^{\gamma+Z}\f{\Phi}{\log_q x, Z}\frac{1}{q^p} \\
    +\; &
    Q(Z)^\Lambda \sum_{\substack{x\in q^{-\Lambda}\Z\\ 1\le x < q^{\nu}}}
    x^{\gamma+Z}\f{\Phi}{\log_q x, Z}\frac{1}{q^\Lambda}.
  \end{align*}
  The inner sums are Riemann sums converging
  to the corresponding integrals for $p\to\infty$.
  We set
  \begin{equation*}
    I(u, Z)\coloneqq\int_{1}^{q^u}x^{\gamma+Z} \f{\Phi}{\log_q x, Z}\,\dd x.
  \end{equation*}
  It will be convenient to change variables $x=q^w$ in $I(u, Z)$ to get
  \begin{equation}\label{eq:Pseudo-Tauber:I-definition}
    I(u, Z)=(\log q)\int_{0}^{u}Q(Z)^w \f{\Phi}{w, Z}\,\dd w.
  \end{equation}
  We define the error~$\varepsilon_p(u, Z)$ by
  \begin{equation*}
    \sum_{\substack{x\in q^{-p}\Z\\
        1\le x < q^u}}x^{\gamma+Z} \f{\Phi}{\log_q x, Z}\frac1{q^p}=I(u, Z) +
    \varepsilon_{p}(u, Z).
  \end{equation*}
  As the sum and the integral are both analytic in $Z$, their difference
  $\varepsilon_p(u, Z)$ is analytic in $Z$, too.
  We bound~$\varepsilon_{p}(u, Z)$ by the difference of upper and lower
  Darboux sums (step size~$q^{-p}$)
  corresponding to the integral~$I(u, Z)$:
  On each interval of length $q^{-p}$, the maximum and minimum of a
  Hölder continuous function can differ by at most $\Oh{q^{-\alpha p}}$. As
  the integration interval as well as the range for $u$ and $Z$ are finite, this translates to the bound
  $\varepsilon_p(u, Z)=\Oh{q^{-\alpha p}}$ as $p\to\infty$
  uniformly in $0\le u\le 1$ and $\abs{Z}<2r$. This results in
    \begin{multline*}
    L(N, Z)=
    I(1, Z)\sum_{0\le p<\Lambda}Q(Z)^p
    + \sum_{0\le p<\Lambda}Q(Z)^p \varepsilon_{p}(1, Z)
    \\
    + I(\nu, Z)\,Q(Z)^{\Lambda} + Q(Z)^\Lambda  \varepsilon_{\Lambda}(\nu, Z).
  \end{multline*}
  If $\abs{Q(Z)}/q^\alpha=q^{\Re\gamma+1 + \Re Z -\alpha}<1$, i.e.,
  $\Re \gamma+\Re Z<\alpha-1$,
  the second sum involving the integration error
  converges absolutely and uniformly in $Z$
  for $\Lambda\to\infty$ to some analytic function
  $c'(Z)$; therefore, we can
  replace the second sum by
  $c'(Z)+\Oh[\big]{q^{(\Re \gamma+1+2r-\alpha)\Lambda}}=c'(Z)+\Oh[\big]{N^{\Re\gamma+1+2r-\alpha}}$
  in this case.
  If $\Re \gamma + \Re Z>\alpha-1$, then the second sum is
  $\Oh[\big]{q^{(\Re \gamma+2r+1-\alpha)\Lambda}}=\Oh[\big]{N^{\Re\gamma+1+2r-\alpha}}$.
  By our choice of $r$, the case $\Re \gamma+\Re Z=\alpha-1$ cannot occur.
  So in any case, we may write the second sum as
  $c'(Z)+\Oh[\big]{N^{\Re \gamma+1-\beta}}$ by our choice of $r$.
  The last summand involving $\varepsilon_{\Lambda}(\nu, Z)$ is absorbed by
  the error term of the second summand.
  Note that the error term is uniform in $Z$ and, by its construction,
  analytic in~$Z$.

  Thus we end up with
  \begin{equation}\label{eq:Pseudo-Tauber:L-decomposition}
    L(N, Z)= c'(Z) + S(N, Z) + \Oh[\big]{N^{\Re \gamma+1-\beta}}
  \end{equation}
  where
  \begin{equation}\label{eq:pseudo-Tauber-S-definition}
    S(N, Z)\coloneqq I(1, Z)\sum_{0\le
      p<\Lambda}Q(Z)^p+I(\nu, Z)Q(Z)^\Lambda.
  \end{equation}

  It remains to rewrite $S(N, Z)$ in the form required by
  \eqref{eq:pseudo-Tauber-relation}. We emphasise that we will compute $S(N, Z)$
  exactly, i.e., no more asymptotics for $N\to\infty$ will play any rôle.

  \proofparagraph{Construction of $\Psi$}
  We will now rewrite the expression $S(N, Z)$ such that the generating
  function~$\Psi$ (i.e., the fluctuations of the right-hand side
  of~\eqref{eq:pseudo-Tauber-relation}) appears. After this, we
  will gather properties of~$\Psi$ including
  properties of its Fourier coefficients.

  We rewrite~\eqref{eq:pseudo-Tauber-S-definition} as
  \begin{align*}
    S(N, Z)&=
             I(1, Z)\frac{1-\f{Q}{Z}^\Lambda}{1-\f{Q}{Z}}
             + I(\nu, Z) \f{Q}{Z}^\Lambda.
  \end{align*}
  We replace
  $\Lambda$ by $\log_q N - \nu$ and use
  \begin{align*}
    \f{Q}{Z}^\Lambda
    &= \f{Q}{Z}^{\log_q N}\f{Q}{Z}^{-\nu}
    = N^{\gamma+1+Z}  \f{Q}{Z}^{-\nu}
  \end{align*}
  to get
  \begin{equation}\label{eq:Pseudo-Tauber:S-decomposition}
    S(N, Z)= \frac{I(1, Z)}{1-\f{Q}{Z}}
    + N^{\gamma+1+Z} \Psi(\nu, Z)
  \end{equation}
  with
  \begin{equation}\label{eq:Pseudo-Tauber:Psi-definition}
    \Psi(u, Z)\coloneqq  \f{Q}{Z}^{-u}
    \Bigl(I(u, Z)-\frac{I(1, Z)}{1-\f{Q}{Z}}\Bigr).
  \end{equation}

  \proofparagraph{Periodic Extension of $\Psi$}
  A priori, it is not clear that the function~$\Psi(u, Z)$ defined above can
  be extended to a periodic function (and therefore Fourier coefficients
  can be computed later on). The aim now is to show that it is
  possible to do so.
  
  It is obvious that $\f{\Psi}{u, Z}$ is continuously differentiable in $u\in[0,
  1]$.
  We have
  \begin{equation*}
    \f{\Psi}{1, Z}=\frac{I(1, Z)}{\f{Q}{Z}}
    \Bigl(1-\frac{1}{1-\f{Q}{Z}}\Bigr)
    =-\frac{I(1, Z)}{1-\f{Q}{Z}}=\f{\Psi}{0, Z}
  \end{equation*}
  because $I(0, Z)=0$ by \eqref{eq:Pseudo-Tauber:I-definition}.
  The derivative of $\f{\Psi}{u, Z}$ with respect to $u$ is
  \begin{align*}
    \frac{\partial \f{\Psi}{u,Z}}{\partial u}
    &= -\bigl(\log\f{Q}{Z}\bigr) \f{\Psi}{u, Z}
    + (\log q) \f{Q}{Z}^{-u} \f{Q}{Z}^u \f{\Phi}{u, Z}\\
    &= -\bigl(\log\f{Q}{Z}\bigr) \f{\Psi}{u, Z} + (\log q) \f{\Phi}{u, Z},
  \end{align*}
  which implies that
  \begin{equation*}
    \frac{\partial \f{\Psi}{u,Z}}{\partial u}\Bigr|_{u=1}=\frac{\partial \f{\Psi}{u,Z}}{\partial u}\Bigr|_{u=0}.
  \end{equation*}
  We can therefore extend $\f{\Psi}{u, Z}$ to a $1$-periodic continuously
  differentiable function in $u$ on $\R$.

  \proofparagraph{Fourier Coefficients of $\Psi$}
  Knowing that~$\Psi$ is a periodic function, we can now head for
  its Fourier coefficients and relate them to those of~$\Phi$.

  By using equations~\eqref{eq:Pseudo-Tauber:Psi-definition} and
  \eqref{eq:Pseudo-Tauber:I-definition}, $Q(Z)=q^{\gamma+1+Z}$, and
  $\exp(-2\ell\pi iu)=q^{-\chi_\ell u}$ with $\chi_\ell=\frac{2\pi i\ell}{\log q}$, we now express the Fourier coefficients of $\f{\Psi}{u, Z}$ in terms of those of
  $\f{\Phi}{u, Z}$ by
  \begin{multline*}
    \int_{0}^1 \f{\Psi}{u, Z} \exp(-2\ell\pi i u) \,\dd u\\
    \begin{aligned}
    &=
    (\log q)\int_{0\le w\le u\le 1}
      \f{Q}{Z}^{w-u} \f{\Phi}{w, Z} q^{-\chi_\ell u} \,\dd w\,\dd u \\
    &\phantom{=}\;
      -\frac{I(1, Z)}{1-\f{Q}{Z}} \int_0^1
        q^{-(\gamma+1+Z+\chi_\ell)u} \,\dd u\\
    &=
    (\log q)\int_{0\le w\le 1} \f{Q}{Z}^w \f{\Phi}{w, Z}
      \int_{w\le u\le 1} q^{-(\gamma+1+Z+\chi_\ell)u} \,\dd u\,\dd w \\
    &\phantom{=}\;
      -\frac{I(1, Z)}{(1-\f{Q}{Z})(\log q)(\gamma+1+Z+\chi_\ell)}
        \Bigl(1-\frac{1}{\f{Q}{Z}}\Bigr)\\
    &=
    \frac{1}{\gamma+1+Z+\chi_\ell}
      \int_0^1 \f{Q}{Z}^w \f{\Phi}{w, Z}
      \Bigl(q^{-(\gamma+1+Z+\chi_\ell)w}-\frac1{\f{Q}{Z}}\Bigr)
        \,\dd w \\
    &\phantom{=}\;
      + \frac{I(1, Z)}{\f{Q}{Z}(\log q)(\gamma+1+Z+\chi_\ell)}\\
    &=
    \frac{1}{\gamma+1+\chi_\ell+Z}
      \int_0^1 \f{\Phi}{w, Z} \fexp{-2\ell\pi i w} \,\dd w\\
    &\phantom{=}\;
      -\frac{1}{\f{Q}{Z} (\gamma+1+\chi_\ell+Z)}
        \int_0^1 \f{Q}{Z}^w \f{\Phi}{w, Z} \,\dd w\\
    &\phantom{=}\;
      + \frac{I(1, Z)}{\f{Q}{Z}(\log q)(\gamma+1+Z+\chi_\ell)}.
    \end{aligned}
  \end{multline*}
  The second and third summands cancel, and we get
  \begin{equation}\label{eq:Pseudo-Tauber:Fourier-Coefficients-GF}
    \Bigl(\gamma+1+\chi_\ell + Z\Bigr)
    \int_{0}^1 \f{\Psi}{u, Z}\exp(-2\ell\pi i u)\,\dd u =
    \int_0^1\f{\Phi}{w, Z}
    \exp(-2\ell\pi i w)\,\dd w.
  \end{equation}

  \proofparagraph{Extracting Coefficients}
  So far, we have proven everything in terms of generating functions.
  We now extract the coefficients of these power series which will
  give us the result claimed in Proposition~\ref{proposition:pseudo-Tauber}.

  By~\eqref{eq:Pseudo-Tauber:Psi-definition}, $\f{\Psi}{u, Z}$ is analytic in $Z$
  for $0<\abs{Z}<2r$. If $q^{\gamma+1}\neq 1$, then it is analytic in $Z=0$, too. If
  $q^{\gamma+1}=1$, then~\eqref{eq:Pseudo-Tauber:Psi-definition} implies that $\f{\Psi}{u, Z}$
  might have a simple pole in $Z=0$.
  Note that all other possible poles have been excluded by our choice of $r$.
  For $j\ge -1$, we write
  \begin{equation*}
    \Psi_j(u)\coloneqq [Z^j]\f{\Psi}{u, Z}
  \end{equation*}
  and use Cauchy's formula to obtain
  \begin{equation*}
    \Psi_j(u) = \frac1{2\pi i}\oint_{\abs{Z}=r}\frac{\f{\Psi}{u, Z}}{Z^{j+1}}\,\dd Z.
  \end{equation*}
  This and the properties of $\f{\Psi}{u, Z}$ established above
  imply that $\Psi_j$ is a $1$-periodic continuously differentiable function.

  Inserting \eqref{eq:Pseudo-Tauber:S-decomposition}
  in~\eqref{eq:Pseudo-Tauber:L-decomposition} and extracting the coefficient of
  $Z^{m-1}$ using Cauchy's theorem and the analyticity of the error in $Z$ yields~\eqref{eq:pseudo-Tauber-relation}
  with $c=[Z^{m-1}]\bigl(c'(Z) + \frac{I(1, Z)}{1-\f{Q}{Z}}\bigr)$.
  Rewriting
  \eqref{eq:Pseudo-Tauber:Fourier-Coefficients-GF} in terms of $\Psi_j$ and $\Phi_j$ leads to~\eqref{eq:pseudo-Tauber-Fourier}.
  Note that we have to add $\Oh{Z^m}$ in~\eqref{eq:pseudo-Tauber-Fourier} to
  compensate the fact that we do not include $\psi_{j\ell}$ for $j\ge m$.
\end{proof}

We prove a uniqueness result.

\begin{lemma}\label{lemma:uniqueness-fluctuations}
  Let $m$ be a positive integer, $q>1$ be a real number, $\gamma\in\C$ such
  that $\gamma\notin \frac{2\pi i}{\log q}\Z$, $c\in\C$, and $\Psi_0$, \ldots, $\Psi_{m-1}$ and $\Xi_0$,
  \ldots, $\Xi_{m-1}$ be $1$-periodic continuous functions such that
  \begin{equation}\label{eq:Fourier:function-comparison}
    \sum_{0\le k<m}(\log_qN)^k\Psi_k(\log_q N) = \sum_{0\le k<m}(\log_q N)^k
    \Xi_k(\log_q N) + c N^{-\gamma} + \oh{1}
  \end{equation}
  for integers $N\to\infty$. Then $\Psi_k=\Xi_k$ for $0\le k<m$.
\end{lemma}
\begin{proof}If $\Re \gamma <0$ and $c\neq 0$, then
  \eqref{eq:Fourier:function-comparison} is impossible as the growth of the
  right-hand side of the equation is larger than that on the left-hand side.
  So we can exclude this
  case from further consideration.
  We proceed indirectly and choose $k$ maximally such that $\Xi_k\neq\Psi_k$.
  Dividing \eqref{eq:Fourier:function-comparison} by $(\log_q N)^k$ yields
  \begin{equation}\label{eq:comparison}
    (\Xi_k-\Psi_k)(\log_q N) = cN^{-\gamma}\iverson{k=0} + \oh{1}
  \end{equation}
  for $N\to\infty$. Let
  $0< u<1$ and set $N_j=\floor{q^{j+u}}$. We
  clearly have $\lim_{j\to\infty} N_j=\infty$. Then
  \begin{equation*}
    j+u + \log_q(1-q^{-j-u}) = \log_q(q^{j+u}-1)\le \log_q N_j \le j+u.
  \end{equation*}
  We define $\nu_j\coloneqq \log_q N_j-j-u$ and see that $\nu_j=\Oh{q^{-j}}$ for
  $j\to\infty$, i.e., $\lim_{j\to\infty} \nu_j = 0$.
  This implies that $\lim_{j\to\infty}\fractional{\log_q N_j}=u$ and therefore
  \begin{equation*}
    \lim_{j\to\infty}(\Xi_k-\Psi_k)(\log_q N_j)=\lim_{j\to\infty}(\Xi_k-\Psi_k)(\fractional{\log_q N_j})=\Xi_k(u)-\Psi_k(u).
  \end{equation*}
  Setting $N=N_j$ in \eqref{eq:comparison} and letting $j\to \infty$ shows that
  \begin{equation}\label{eq:comparison-limit}
    \Xi_k(u)-\Psi_k(u) = \lim_{j\to\infty}cN_j^{-\gamma}\iverson{k=0}.
  \end{equation}
  If $k\neq 0$ or $\Re \gamma>0$, we immediately conclude that
  $\Xi_k(u)-\Psi_k(u)=0$. If $\Re \gamma<0$ we have
  $c=0$, which again implies that $\Xi_k(u)-\Psi_k(u)=0$.

  Now we assume that $\Re \gamma=0$ and $k=0$. We set
  $\beta\coloneqq -\frac{\log q}{2\pi i}\gamma$, which implies that
  $N^{-\gamma}=\exp(2\pi i \beta\log_q N)$. We choose sequences
  $(r_\ell)_{\ell\ge 1}$ and $(s_\ell)_{\ell\ge 1}$ such that
  $\lim_{\ell\to\infty }s_\ell=\infty$ and $\lim_{\ell\to\infty}\abs{s_\ell
    \beta - r_\ell}=0$: For rational $\beta=r/s$, we simply take $r_\ell=\ell
  r$ and $s_\ell=\ell s$, and for irrational $\beta$, we consider the sequence of
  convergents $(r_\ell/s_\ell)_{\ell\ge 1}$ of the continued fraction of
  $\beta$ and the required properties follow from the theory of continued
  fractions; see for example \cite[Theorems~155
  and~164]{Hardy-Wright:1975}. By using $\log_q N_j = j+u+\nu_j$, we get
  \begin{align*}
    \lim_{\ell\to\infty}N_{s_\ell}^{-\gamma} &= \lim_{\ell\to\infty}\exp(2\pi i
    (r_\ell + \beta u + (s_\ell \beta - r_\ell) + \beta \nu_{s_\ell})=\exp(2\pi i \beta u),\\
    \lim_{\ell\to\infty}N_{s_\ell+1}^{-\gamma} &= \lim_{\ell\to\infty}\exp(2\pi i
    (r_\ell + \beta + \beta u + (s_\ell \beta - r_\ell)+\beta \nu_{s_\ell+1})=\fexp[\big]{2\pi i \beta (1+u)}.
  \end{align*}
  These two limits are distinct as $\beta\notin\Z$ by assumption.
  Thus $\lim_{j\to\infty}N_j^{-\gamma}$ does not exist. Therefore,
  \eqref{eq:comparison-limit} implies that $c=0$ and therefore $\Xi_k(u)-\Psi_k(u)=0$.

  We proved that $\Xi_k(u)=\Psi_k(u)$ for $u\notin\Z$. By continuity, this
  also follows for all $u \in \R$; contradiction.
\end{proof}

\subsection{Proof of Theorem~\ref{theorem:use-Mellin--Perron}}

We again start with an outline of the proof.
\begin{proof}[Overview of the Proof of Theorem~\ref{theorem:use-Mellin--Perron}]
The idea is to compute the repeated summatory function of $F$
twice: On the one hand, we use the pseudo-Tauberian Proposition~\ref{proposition:pseudo-Tauber} to rewrite the
right-hand side of \eqref{eq:F-N-periodic} in terms of periodic
functions~$\Psi_{aj}$. On the other hand, we compute it using a higher
order Mellin--Perron summation
formula, relating it to the singularities of $\calF$.
More specifically, the expansions at the singularities of $\calF$ give
the Fourier coefficients of $\Psi_{aj}$. The Fourier coefficients of the functions
$\Psi_{aj}$ are related to those of the functions $\Phi_j$ via~\eqref{eq:pseudo-Tauber-Fourier}.
\end{proof}

And up next comes the actual proof.

\begin{proof}[Proof of Theorem~\ref{theorem:use-Mellin--Perron}]
  \proofparagraph{Initial observations and notations}
  As $\Phi_j$ is Hölder continuous, its Fourier series converges by Dini's
  criterion; see, for example, \cite[p.~52]{Zygmund:2002:trigon}.

  For any sequence $g$ on $\Z_{>0}$, we set $(\calS g)(N)\coloneqq \sum_{1\le n< N}g(n)$.
  We set $A=1 + \max\set{\floor{\eta}, 0}$. In particular, $A$ is a positive
  integer with $A>\eta$.

  \proofparagraph{Asymptotic Summation}
  We first compute the $A$th repeated summatory function~$\calS^A F$
  of~$F$ (i.e., the $(A+1)$th repeated summatory function $\calS^{A+1} f$ of
  the function~$f$) by applying Proposition~\ref{proposition:pseudo-Tauber} $A$ times.
  This results in an asymptotic expansion involving new periodic fluctuations
  while keeping track of the relation between the Fourier coefficients of
  the original fluctuations and those of the new fluctuations.

  A simple induction based on~\eqref{eq:F-N-periodic} and using
  Proposition~\ref{proposition:pseudo-Tauber}
  shows that
  there exist
  $1$-periodic continuous functions $\Psi_{aj}$ for $a\ge 0$ and $-1\le j<m$
  and some constants $c_{ab}$ for $0\le b<a$ such that
  \begin{equation}\label{eq:S-a+1-f-asymptotic}
    (\calS^{a+1} f)(N) = \sum_{0\le b<a}c_{ab}N^b +
    N^{\gamma+a}\sum_{\substack{j+k=m-1\\-1\le j<m}} \frac{(\log N)^k}{k!}
    \Psi_{aj}(\fractional{\log_q N}) + \Oh{N^{\gamma_0+a}}
  \end{equation}
  for integers $N\to\infty$. In fact, $\Psi_{0j}=\Phi_j$ for
  $0\le j<m$. For $a\ge 1$ and $-1\le j<m$, $\Psi_{aj}$ is continuously differentiable.
  Note that the case that $q^{\gamma+a+1}=1$ occurs for at most one $0\le a<A$,
  which implies that the number of non-vanishing fluctuations increases at most
  once in the application of Proposition~\ref{proposition:pseudo-Tauber}.
  Also note that the assumption $\alpha>\Re \gamma-\gamma_0$ implies that the
  error terms arising in the application of
  Proposition~\ref{proposition:pseudo-Tauber} are absorbed by the error term
  stemming from~\eqref{eq:F-N-periodic}.

  We denote the corresponding Fourier coefficients by
  \begin{equation*}
    \psi_{aj\ell}\coloneqq \int_{0}^1 \Psi_{aj}(u)\exp(-2\ell\pi i u)\,\dd u
  \end{equation*}
  for $0\le a\le A$, $-1\le j<m$, $\ell\in\Z$. By~\eqref{eq:pseudo-Tauber-Fourier}
  the generating functions of the Fourier coefficients fulfil
  \begin{equation*}
        \sum_{-1\le j<m}\psi_{aj\ell}Z^j = (\gamma+a+1+\chi_\ell + Z)\sum_{-1\le j<m}\psi_{(a+1)j\ell}Z^j
        +\Oh{Z^m}
  \end{equation*}
  for $0\le a<A$,
  $\ell\in\Z$ and $Z\to 0$. Iterating this recurrence
  yields
  \begin{equation}\label{eq:Fourier:Fourier-coefficient-recursion-full}
        \sum_{0\le j<m}\psi_{0j\ell}Z^j = \biggl(\prod_{1 \le a \le A} (\gamma+a+\chi_\ell + Z)\biggr)\sum_{-1\le j<m}\psi_{Aj\ell}Z^j
        +\Oh{Z^m}
  \end{equation}
  for $\ell\in\Z$ and $Z\to 0$.

  \proofparagraph{Explicit Summation}
  We now compute $\calS^{A+1} f$ explicitly with the aim of decomposing it into
  one part which can be computed by the $A$th order Mellin--Perron summation
  formula and another part which is smaller and can be absorbed by an error term.

  Explicitly, we have
  \begin{equation*}
    (\calS^{a+1}f)(N) = \sum_{1\le n_1<n_2<\cdots<n_{a+1}<N}f(n_1)
    = \sum_{1\le n<N}f(n)\sum_{n<n_2<\cdots<n_{a+1}<N}1
  \end{equation*}
  for $0\le a \le A$. Note that we formally write the outer sum over the range
  $1\le n<N$ although the inner sum is empty (i.e., equals~$0$) for $n\ge N-a$; this will be useful
  later on. The inner sum counts the number of selections of $a$ elements out
  of $\set{n+1,\ldots, N-1}$, thus we have
  \begin{equation}\label{eq:Fourier:explicit-summation}
    (\calS^{a+1}f)(N) = \sum_{1\le n< N}\binom{N-n-1}{a}f(n)=\sum_{1\le n< N}\frac1{a!}(N-n-1)^{\underline{a}}f(n)
  \end{equation}
  for $0\le a\le A$ and falling factorials
  $z^{\underline{a}}\coloneqq z(z-1)\dotsm (z-a+1)$.

  The polynomials $\frac1{a!}(U-1)^{\underline a}$, $0\le a\le A$, are clearly a basis of
  the space of polynomials in $U$ of degree at most $A$. Thus, there exist
  rational numbers $b_0$, \ldots, $b_A$ such that
  \begin{equation*}
    \frac{U^A}{A!}=\sum_{0 \le a \le A} \frac{b_a}{a!} (U-1)^{\underline{a}}.
  \end{equation*}
  Comparing the coefficients of  $U^A$ shows that $b_A=1$. Substitution of
  $U$ by $N-n$, multiplication by $f(n)$ and summation over $1\le n<N$
  yield
  \begin{equation*}
    \frac1{A!}\sum_{1\le n<N}(N-n)^A f(n) = \sum_{0 \le a \le A} b_a (\calS^{a+1}f)(N)
  \end{equation*}
  by~\eqref{eq:Fourier:explicit-summation}. When inserting the asymptotic
  expressions from \eqref{eq:S-a+1-f-asymptotic}, the summands involving
  fluctuations for $0\le a<
  A$ are absorbed by the error term~$\Oh{N^{\gamma_0+A}}$
  of the summand for $a=A$ because $\Re\gamma - \gamma_0 < 1$. Thus there are
  some constants $c_b$ for $0\le b<A$ such that
  \begin{multline}\label{eq:Mellin-Perron-sum}
    \frac1{A!}\sum_{1\le n<N}(N-n)^A f(n) = \sum_{0\le b<A}c_{b}N^b \\+
    N^{\gamma+A}\sum_{\substack{j+k=m-1\\-1\le j<m}} \frac{(\log N)^k}{k!}
    \Psi_{Aj}(\fractional{\log_q N}) + \Oh{N^{\gamma_0+A}}
  \end{multline}
  for integers $N\to\infty$.

  If $\gamma+A=b+\chi_{\ell'}$ for some $0\le b<A$ and $\ell'\in\Z$, then we
  assume without loss of generality that $c_{b}=0$: Otherwise, we replace
  $\Psi_{A(m-1)}(u)$ by $\Psi_{A(m-1)}(u) + c_{b}\exp(-2\ell'\pi i u)$ and
  $c_{b}$ by $0$. Both~\eqref{eq:Mellin-Perron-sum} and
  \eqref{eq:Fourier:Fourier-coefficient-recursion-full} remain intact: The
  former trivially, the latter because the factor for $a=A-b$
  in~\eqref{eq:Fourier:Fourier-coefficient-recursion-full} equals
  $\gamma+A-b-\chi_{\ell'} + Z=Z$ which compensates
  the fact that the Fourier coefficient $\psi_{A(m-1)(-\ell')}$ is modified.

  \proofparagraph{Mellin--Perron summation}
  We use the $A$th order Mellin--Perron summation formula to write the main contribution
  of $\calS^{A+1} f$ as determined above in terms of new periodic fluctuations $\Xi_j$ whose
  Fourier coefficients are expressed in terms of residues of a suitably
  modified version of the Dirichlet generating function $\calF$.

  Without loss of generality, we assume that $\sigmaabs>0$: The growth
  condition~\eqref{eq:Dirichlet-order} trivially holds with $\eta=0$ on the
  right of the abscissa of absolute convergence of the Dirichlet series.
  By the $A$th order Mellin--Perron summation
  formula (see \cite[Theorem~2.1]{Flajolet-Grabner-Kirschenhofer-Prodinger:1994:mellin}), we have
  \begin{equation*}
    \frac1{A!}\sum_{1\le n<N}(N-n)^A f(n)  = \frac1{2\pi
      i}\int_{\sigmaabs+1-i\infty}^{\sigmaabs+1+i\infty} \frac{\calF(s) N^{s+A}}{s(s+1)\dotsm(s+A)}\,\dd s
  \end{equation*}
  with the arbitrary choice $\sigmaabs+1>\sigmaabs$ for the real part of
  the line of integration.

  The growth condition~\eqref{eq:Dirichlet-order} allows us to shift the
  line of integration to the left such that
  \begin{align*}
    \frac1{A!}\sum_{1\le n<N}&(N-n)^A f(n) \\ &=
    \sum_{\ell\in\Z}
    \Res[\Big]{\frac{\calF(s)N^{s+A}}{s(s+1)\dotsm (s+A)}}%
    {s=\gamma+\chi_\ell}\\
    &\phantom{=}\hspace*{0.65em}+ \sum_{0\le a\le \min\{-\gamma_0, A\}}(-1)^a\frac{\calF(-a)}{a!(A-a)!}N^{A-a}\iverson[\Big]{\gamma\notin -a+\frac{2\pi i}{\log q}\Z}\\
    &\phantom{=}\hspace*{0.65em}+ \frac1{2\pi
      i}\int_{\gamma_0-i\infty}^{\gamma_0+i\infty} \frac{\calF(s) N^{s+A}}{s(s+1)\dotsm (s+A)}\,\dd s.
  \end{align*}
  The summand for~$a$ in the second term corresponds to a possible pole at $s=-a$ which is not taken care of in the first sum; note that $\calF(s)$ is analytic at $s=-a$
  in this case
  by assumption because of~$\gamma_0<-a$.

  We now compute the residue at $s=\gamma+\chi_\ell$. We use
  \begin{equation*}
    N^{s+A} = N^{\gamma+A+\chi_\ell}\sum_{k\ge 0}\frac{(\log N)^k}{k!}  (s-\gamma-\chi_\ell)^k
  \end{equation*}
  to split up the residue as
  \begin{equation*}
    \Res[\Big]{\frac{\calF(s)N^{s+A}}{s(s+1)\dotsm(s+A)}}{s=\gamma+\chi_\ell} =
    N^{\gamma+A+\chi_\ell}\sum_{\substack{k+j=m-1\\-1\le j<m}}\frac{(\log N)^k}{k!}  \xi_{j\ell}
  \end{equation*}
  with
  \begin{equation}\label{eq:Fourier:xi-as-residue}
    \xi_{j\ell} =
    \Res[\Big]{\frac{\calF(s)(s-\gamma-\chi_\ell)^{m-1-j}}{s(s+1)\dotsm(s+A)}}{s=\gamma+\chi_\ell}
  \end{equation}
  for $j\ge -1$.
  Note that we allow $j=-1$ for the case of $\gamma\in -a+\frac{2\pi i}{\log q}\Z$
  for some $1\le a\le A$ when
  $\calF(s)/\bigl(s\dotsm (s+A)\bigr)$ might have a pole of order $m+1$ at
  $s=-a$.
  Using the growth condition~\eqref{eq:Dirichlet-order} and the
  choice of~$A$ yields
  \begin{equation}\label{eq:Fourier:growth-frac}
    \frac{\calF(s)}{s(s+1)\dotsm(s+A)}
    = \Oh[\big]{\abs{\Im s}^{-1-A+\eta}} = \oh[\big]{\abs{\Im s}^{-1}}
  \end{equation}
  for $\abs{\Im s}\to\infty$ and $s$ which are at least a distance~$\delta$
  away from the poles~$\gamma+\chi_\ell$.
  By writing the residue in~\eqref{eq:Fourier:xi-as-residue}
  in terms of an integral over a rectangle around
  $s=\gamma+\chi_\ell$ (distance again at least~$\delta$ away from $\gamma+\chi_\ell$),
  we see that \eqref{eq:Fourier:growth-frac} implies
  \begin{equation}\label{eq:Fourier:psi-growth}
    \xi_{j\ell} = \Oh[\big]{\abs{\ell}^{-1-A+\eta}} = \oh[\big]{\abs{\ell}^{-1}}
  \end{equation}
  for $\abs{\ell}\to\infty$. Moreover,
  by~\eqref{eq:Fourier:growth-frac}, we see that
  \begin{equation*}
    \frac1{2\pi i} \int_{\gamma_0-i\infty}^{\gamma_0+i\infty}
    \frac{\calF(s) N^{s+A}}{s(s+1)\dotsm(s+A)}\,\dd s
    = \Oh{N^{\gamma_0+A}}.
  \end{equation*}

  Thus we proved that
  \begin{multline}\label{eq:calculate-Fourier-first}
    \frac1{A!}\sum_{1\le n<N}(N-n)^A f(n) = N^{\gamma+A}\sum_{\substack{k+j=m-1\\-1\le
      j<m}} \frac{(\log N)^k}{k!}
  \Xi_j(\log_q N) \\
+ \sum_{0\le a\le\min\{-\gamma_0,A\}}(-1)^a\frac{\calF(-a)}{a!(A-a)!}N^{A-a}\iverson[\Big]{\gamma\notin -a+\frac{2\pi i}{\log q}\Z}+ \Oh{N^{\gamma_0+A}}
  \end{multline}
  for
  \begin{equation}\label{eq:Psi-tilde-k-definition}
    \Xi_j(u) =\sum_{\ell\in\Z}\xi_{j\ell} \exp(2\ell\pi i u)
  \end{equation}
  where the $\xi_{j\ell}$ are given in \eqref{eq:Fourier:xi-as-residue}.
  By \eqref{eq:Fourier:psi-growth}, the Fourier series
  \eqref{eq:Psi-tilde-k-definition} converges uniformly and absolutely. This
  implies that $\Xi_j$  is a $1$-periodic continuous function.

  \proofparagraph{Fourier Coefficients}
  We will now compare the two asymptotic expressions for $\calS^{A+1} f$ obtained so far
  to see that the fluctations coincide. We know explicit expressions for the
  Fourier coefficients of the $\Xi_j$ in terms of residues, and we know how
  the Fourier coefficients of the fluctuations of the repeated summatory
  function are related to the Fourier coefficients of the fluctuations of $F$.
  Therefore, we are able to compute the latter.

  By
  \eqref{eq:Mellin-Perron-sum}, \eqref{eq:calculate-Fourier-first},
  elementary asymptotic considerations for the terms $N^b$ with $b>\Re \gamma+A$,
  Lemma~\ref{lemma:uniqueness-fluctuations} and the fact that $c_{b}=0$ if
  $b\in \gamma+A+\frac{2\pi i}{\log q}\Z$ for some $0\le b<A$, we see
  that $\Xi_j=\Psi_{Aj}$ for $-1\le j<m$. This immediately implies that $\calF(0)=0$ if
  $\gamma_0<0$ and $\gamma\notin\frac{2\pi i}{\log q}\Z$.

  To compute the Fourier coefficients $\psi_{Aj\ell}=\xi_{j\ell}$, we set
  $Z\coloneqq s-\gamma-\chi_\ell$ to rewrite~\eqref{eq:Fourier:xi-as-residue}
  using \eqref{eq:Fourier:F-s-principal-part} as
  \begin{equation*}
    \psi_{Aj\ell}=[Z^{-1}]
    \frac{\sum_{b\ge 0}\varphi_{b\ell}Z^{b-j-1}}{\prod_{1 \le a \le A} (\gamma+a+\chi_\ell+Z)}
    =[Z^{j}]
    \frac{\sum_{b\ge 0}\varphi_{b\ell} Z^{b}}{\prod_{1 \le a \le A} (\gamma+a+\chi_\ell+Z)}
  \end{equation*}
  for $-1\le j<m$ and $\ell\in\Z$.
  This is equivalent to
  \begin{equation*}
    \sum_{-1\le j<m}\psi_{Aj\ell}Z^j=\frac{\sum_{j\ge 0}\varphi_{j\ell}
      Z^{j}}{\prod_{1 \le a \le A} (\gamma+a+\chi_\ell+Z)} + \Oh{Z^{m}}
  \end{equation*}
  for $\ell\in\Z$ and $Z\to 0$. Clearing the denominator and
  using~\eqref{eq:Fourier:Fourier-coefficient-recursion-full} as announced in Remark~\ref{remark:recurrence-fluctuation} lead to
  \begin{equation*}
    \sum_{0\le j< m}\psi_{0j\ell}
    Z^{j}=\sum_{j\ge 0}\varphi_{j\ell}
    Z^{j} + \Oh{Z^{m}}
  \end{equation*}
  for $\ell\in\Z$ and $Z\to 0$. Comparing coefficients shows that
  $\psi_{0j\ell}=\varphi_{j\ell}$ for $0\le j<m$ and $\ell\in\Z$.
  This proves~\eqref{eq:Fourier:fluctuation-as-Fourier-series}.
\end{proof}

\section{Proof of Theorem~\ref{theorem:simple}}\label{section:proof-theorem-simple}
\begin{proof}[Proof of Theorem~\ref{theorem:simple}]
  By Remark~\ref{remark:regular-sequence-as-a-matrix-product}, we have
  $x(n)=e_1 f(n)v(0)$. If $v(0)=0$, there is nothing to show.
  Otherwise, as observed in Section~\ref{section:q-regular-matrix-product},
  $v(0)$ is a right eigenvector of $A_0$ associated to the eigenvalue $1$.
  As a consequence, $Kv(0)$, $\vartheta_m v(0)$ and $\vartheta v(0)$ all vanish.
  Therefore, \eqref{eq:formula-X-n} follows from Theorem~\ref{theorem:main}
  by multiplication by $e_1$ and $v(0)$ from left and right, respectively. Note
  that the notation is somewhat different: Instead of powers $(\log_q N)^k$ in
  Theorem~\ref{theorem:main} we write $(\log N)^k/k!$ here.

  The functional equation \eqref{eq:functional-equation-V} follows from
  Theorem~\ref{theorem:Dirichlet-series} for $n_0=1$ by multiplication from right
  by $v(0)$.

  For computing the Fourier coefficients, we denote the rows of $T$ by $w_1$,
  \ldots, $w_d$. Thus $w_a$ is a generalised left eigenvector of $C$ of some
  order $m_a$ associated to some eigenvalue $\lambda_a$ of $C$. We can write
  $e_1=\sum_{1 \le a \le d} c_a w_a$ for some suitable constants $c_1$, \ldots, $c_d$.
  For $1\le a\le d$, we consider the sequence~$h_a$ on $\Z_{>0}$
  with
  \begin{equation*}
    h_a(n)=w_a\bigl(v(n)+v(0)\iverson{n=1}\bigr).
  \end{equation*}
  The reason for incorporating
  $v(0)$ into the value for $n=1$ is that the corresponding Dirichlet series $\calH^{(a)}(s)\coloneqq \sum_{n\ge
    1}n^{-s}h_a(n)$ only takes values at $n\ge 1$ into account. By definition, we
  have  $\calH^{(a)}(s)=w_av(0) + w_a\calV(s)$. Taking the linear combination yields
  $\sum_{1 \le a \le d} c_a\calH^{(a)}(s)=x(0) + \calX(s)$. We choose
  $\gamma_0> \log_q R$ such that there are no eigenvalues
  $\lambda\in\sigma(C)$ with $\log_q R<\log_q\lambda\le \gamma_0$ and such
  that $\gamma_0\notin \Z_{\le 0}$.

  By
  Theorem~\ref{theorem:contribution-of-eigenspace}, we have
  \begin{equation}\label{eq:simple:sum-lambda_a}
    \sum_{1\le n<N}h_a(n) = N^{\log_q \lambda_a}\sum_{0\le k<m_a}\frac{(\log N)^k}{k!}
    \Psi_{ak}(\fractional{\log_q N}) + \Oh{N^{\gamma_0}}
  \end{equation}
  for $N\to\infty$ for suitable 1-periodic Hölder continuous functions $\Psi_{ak}$
  (which vanish if $\abs{\lambda_a}\le R$). By
  Theorem~\ref{theorem:Dirichlet-series}, the Dirichlet
  series $\calH^{(a)}(s)$ is meromorphic for $\Re s>\gamma_0$ with possible
  poles at $s=\log_q \lambda_a + \chi_\ell$ for $\ell\in\Z$.

  The sequence $h_a$ satisfies
  the prerequisites of Theorem~\ref{theorem:use-Mellin--Perron}, either with
  $\gamma=\log_q \lambda_a$ if $\Re \log_q \lambda_a>\gamma_0$ or with
  arbitrary real $\gamma>\gamma_0$ and $\Phi_j=0$ for all $j$
  if $\Re \log_q \lambda_a \le \gamma_0$. The theorem then
  implies that
  \begin{equation}\label{claim-H-a=0}
    \calH^{(a)}(0) = 0
  \end{equation}
  if $\gamma_0<0$ and $\lambda_a\neq 1$.

  If $\abs{\lambda_a}>R$,
  Theorem~\ref{theorem:use-Mellin--Perron} also yields
  \begin{equation*}
    \Psi_{ak}(u)=\sum_{\ell\in\Z}\psi_{ak\ell}\exp(2\pi i\ell u)
  \end{equation*}
  where the $\psi_{ak\ell}$ are given by the singular expansion
  \begin{equation}\label{eq:proof-theorem-simple-local-expansion}
    \frac{\calH^{(a)}(s)}{s}\asymp\sum_{\ell\in\Z}\sum_{0\le k<m_a}\frac{\psi_{ak\ell}}{(s-\log_q\lambda_a-\chi_\ell)^{k+1}}
  \end{equation}
  for $\Re s>\gamma_0$. Note that~\eqref{claim-H-a=0} ensures that there is no
  additional pole at $s=0$ when $\gamma_0<0$ and $\lambda_a\neq 1$.  Also note
  that in comparison to Theorem~\ref{theorem:use-Mellin--Perron}, $\Phi_{m_a-1-k}$
  there corresponds to $\Psi_{ak}$ here.

  We now have to relate the results obtained for the sequences $h_a$ with the
  results claimed for the original sequence $f$.
  For $\lambda\in\sigma(C)$ with $\abs{\lambda}>R$, we have
  \begin{equation*}
    \Phi_{\lambda k}(u)=\sum_{\substack{1\le a\le d\\\lambda_a=\lambda}}c_a\Psi_{ak}(u).
  \end{equation*}
  We denote the Fourier coefficients of $\Phi_{\lambda k}$ by $\varphi_{\lambda
    k\ell}$ for $\ell\in\Z$ and will show that these Fourier coefficients
  actually fulfil~\eqref{eq:Fourier-coefficient:simple}. Taking linear
  combinations of~\eqref{eq:proof-theorem-simple-local-expansion} shows that
  \begin{equation*}
    \sum_{\substack{1\le a\le d\\\lambda_a=\lambda}}\frac{c_a\calH^{(a)}(s)}{s}
    \asymp \sum_{\ell\in\Z}\sum_{0\le k<m(\lambda)}\frac{\varphi_{\lambda k\ell}}{(s-\log_q\lambda-\chi_\ell)^{k+1}}
    \label{eq:residue-with-condition}
  \end{equation*}
  for $\Re s>\gamma_0$.

  Summing over all $\lambda\in\sigma(C)$ yields~\eqref{eq:Fourier-coefficient:simple}
  because summands $\lambda$ with $\abs{\lambda}\le R$ are analytic for $\Re
  s>\gamma_0$ and do therefore not contribute to the right-hand side.
\end{proof}

It might seem to be somewhat artificial that
Theorem~\ref{theorem:use-Mellin--Perron} is used to prove that
$\calH^{(j)}(0)=0$ in some of the cases above. In fact, this can also be shown
directly using the linear representation; we formulate and prove this
in the following remark.

\begin{remark}
  With the notations of the proof of Theorem~\ref{theorem:simple},
  $\calH^{(j)}(0)=0$ if $\lambda_j\neq 1$ and $R<1$ can also be shown using the
  functional equation~\eqref{eq:functional-equation-V}.
\end{remark}
\begin{proof}
  We prove this by induction on $m_j$. By definition of $T$, we have
  $w_j(C-\lambda_j I)=\iverson{m_j>1}w_{j+1}$. (We have $m_d=1$ thus $w_{d+1}$
  does not actually occur.)
  If $m_j>1$, then $\calH^{(j+1)}(0)=0$ by induction hypothesis.

  We add $(I-q^{-s})\,v(0)$ to \eqref{eq:functional-equation-V} and get
  \begin{align*}
    \bigl(I-q^{-s}C\bigr)\bigl(v(0)+\calV(s)\bigr) = \bigl(I-q^{-s}C\bigr)v(0)
    &+ \sum_{1 \le n < q} n^{-s}v(n) \\
    &+ q^{-s}\sum_{0 \le r < q} A_r
    \sum_{k\ge 1}\binom{-s}{k}\Bigl(\frac r q\Bigr)^k \calV(s+k).
  \end{align*}
  Multiplication by $w_j$ from the left yields
  \begin{align*}
    \bigl(1-q^{-s}\lambda\bigr)\calH^{(j)}(s) &= \iverson{m_j>1}\,q^{-s}\calH^{(j+1)}(s) \\
    &\phantom{{}={}}+ w_j \bigl(I - q^{-s}C\bigr)v(0)
    + w_j\sum_{1 \le n < q} n^{-s}v(n) \\
    &\phantom{{}={}}+ w_jq^{-s}\sum_{0 \le r < q} A_r
    \sum_{k\ge 1}\binom{-s}{k}\Bigl(\frac r q\Bigr)^k \calV(s+k).
  \end{align*}
  As $R<1$ and $\lambda_j\neq 1$, the Dirichlet series $\calH^{(j)}(s)$ is
  analytic in $s=0$ by Theorem~\ref{theorem:Dirichlet-series}. It is therefore
  legitimate to set $s=0$ in the above equation. We use the induction hypothesis that
  $\calH^{(j+1)}(0)=0$ as well as the fact that $v(n)=A_nv(0)$
  (note that $v(0)$ is a right eigenvector of $A_0$ to the eigenvalue~$1$;
  see Section~\ref{section:q-regular-matrix-product})
  for $0\le n<q$ to get
  \begin{equation*}
    (1-\lambda)\calH^{(j)}(0)=w_j\sum_{0 \le n < q} A_n v(0) -w_jCv(0) = 0
  \end{equation*}
  because all binomial coefficients $\binom{0}{k}$ vanish.
\end{proof}

\section{Proof of Proposition~\ref{proposition:symmetric-eigenvalues}}
\label{sec:proof-symmetric-eigenvalues}

\begin{proof}[Proof of Proposition~\ref{proposition:symmetric-eigenvalues}]
  We set
  \begin{equation*}
    j_0\coloneqq \floor[\bigg]{-\frac{p\bigl(\pi+\arg(\lambda)\bigr)}{2\pi}}+1
  \end{equation*}
  with the motive that
  \begin{equation*}
    -\pi<\arg(\lambda) + \frac{2j\pi}{p}\le \pi
  \end{equation*}
  holds for $j_0\le j<j_0+p$.
  This implies that for $j_0\le j<j_0+p$, the $p$th root of unity~$\zeta_j\coloneqq \exp(2j\pi i/p)$
  runs through the elements of $U_p$ such
  that $\log_q(\lambda \zeta_j)=\log_q(\lambda) +  2j\pi i/(p\log q)$.
  Then
  \begin{align*}
    N^{\log_q(\zeta_j\lambda)}
    &= N^{\log_q \lambda} \exp\Bigl(\frac{2j\pi i}{p}\log_q N\Bigr)\\
    &= N^{\log_q \lambda} \exp(2j\pi i\log_{q^p} N)
    = N^{\log_q \lambda} \exp(2j\pi i\fractional{\log_{q^p} N}).
  \end{align*}
  We set
  \begin{equation*}
    \Phi(u)\coloneqq \sum_{j_0\le j<j_0+p} \exp\Bigl(\frac{2j\pi i}{p}u\Bigr)\Phi_{(\zeta_j\lambda)}(u),
  \end{equation*}
  thus $\Phi$ is a $p$-periodic function.

  For the Fourier series expansion, we get
  \begin{multline*}
    \Phi(u)=\sum_{\ell\in\Z} \sum_{j_0\le j<j_0+p}
    \Res[\bigg]{\calD(s)
      \Bigl(s - \log_q \lambda - \frac{2(\ell+\frac{j}{p})\pi i}{\log q}\Bigr)^k}%
    {s=\log_q \lambda + \frac{2(\ell+\frac{j}{p})\pi i}{\log q}} \\
    \times \f[\Big]{\exp}{2\pi i \Bigl(\ell+\frac{j}{p}\Bigr)u}.
  \end{multline*}
  Replacing $\ell p+j$ by $\ell$ leads to the Fourier series claimed in the
  proposition.
\end{proof}

\input{estimates}

\part{References}
\makeatletter\renewcommand{\@bibtitlestyle}{}\makeatother
\bibliography{bib/cheub}
\bibliographystyle{bibstyle/amsplainurl}

\end{document}

%% file: transducer.tex
\subsection{Transducer and Automata}

Let us start with two paragraphs recalling some notions around transducer
automata. A \emph{transducer automaton} has a finite set of
\emph{states} together with \emph{transitions} (directed edges)
between these states. Each transition has an \emph{input label} and an
\emph{output label} out of the \emph{input alphabet} and the
\emph{output alphabet}, respectively.
A transducer is said to be \emph{deterministic} and \emph{complete}
if for every state and every letter of the input alphabet, there is exactly one
transition starting in this state with this input label.

A deterministic and complete transducer processes a word (over the
input alphabet) in the following way:
\begin{itemize}
\item It starts at its unique initial state.
\item Then the transducer reads the word letter by letter and for each
  letter
  \begin{itemize}
  \item takes the transition with matching input label,
  \item the output label is written, and
  \item we proceed to the next state (according to the end of the
    transition).
  \end{itemize}
\item Each state has a \emph{final output label} that is written when
  we \emph{halt} in this final state; we call a transducer with this
  property a \emph{subsequential transducer}.
\end{itemize}
We refer to \cite[Chapter~1]{Berthe-Rigo:2010:combin} for a more
detailed introduction to transducers and automata.

Now we are ready to start with the set-up for our example.

\subsection{Sums of Output Labels}

Let $q\ge 2$ be a positive integer. We consider a complete deterministic
subsequential transducer  $\calT$ with input alphabet $\set{0, \ldots, q-1}$
and output alphabet $\C$; see \cite{Heuberger-Kropf-Prodinger:2015:output}.
For a non-negative integer $n$, let $\calT(n)$ be the sum of the output labels
(including the final output label) encountered when the transducer reads the
$q$-ary expansion of $n$. Therefore, letters of the input alphabet will from now on be called digits.

This concept has been thoroughly studied in
\cite{Heuberger-Kropf-Prodinger:2015:output}: There, $\calT(n)$ is considered
as a random variable defined on the probability space $\set{0, \ldots, N-1}$
equipped with uniform distribution. The expectation in this model corresponds
(up to a factor of~$N$) to our summatory function $\sum_{0\le n<N}\calT(n)$.
We remark that in \cite{Heuberger-Kropf-Prodinger:2015:output}, the variance
and limiting distribution of the random variable $\calT(n)$ have also been
investigated. Most of the results there are also valid for higher dimensional input.

The purpose of this section is to show that $\calT(n)$ is a $q$-regular
sequence and to see that the corresponding results
in~\cite{Heuberger-Kropf-Prodinger:2015:output} also follow from our
more general framework here. We note that the
binary sum of digits considered in Example~\ref{example:binary-sum-of-digits}
is the special case of $q=2$ and the transducer consisting of a single state
which implements the identity map. For additional special cases of this
concept; see \cite{Heuberger-Kropf-Prodinger:2015:output}. Note that our result
here for the summatory function contains (fluctuating) terms for all
eigenvalues $\lambda$ of the adjacency matrix of the underlying digraph with
$\abs{\lambda}>1$ whereas in  \cite{Heuberger-Kropf-Prodinger:2015:output} only
contributions of those eigenvalues $\lambda$ with $\abs{\lambda}=q$ are
available, all other contributions are absorbed by the error term there.

\subsection{Some Perron--Frobenius Theory}

We will need the following consequence of Perron--Frobenius theory.
By a \emph{component} of a digraph we always mean a strongly
connected component.  We call a component  \emph{final} if there
are no arcs leaving the component. The \emph{period} of a component
is the greatest common divisor of its cycle lengths. The \emph{final period} of
a digraph is the least common multiple of the periods of its final components.

\begin{lemma}\label{lemma:Perron--Frobenius-again}
  Let $D$ be a directed graph  where each
  vertex has outdegree $q$. Let $M$ be its adjacency matrix and $p$ be its
  final period.
  Then $M$ has spectral radius $q$, $q$ is an
  eigenvalue of $M$ and for all eigenvalues $\lambda$ of $M$ of modulus $q$, the
  algebraic and geometric multiplicities coincide and $\lambda = q\zeta$ for
  some $p$th root of unity $\zeta$.
\end{lemma}
This lemma follows from setting $t=0$ in
\cite[Lemma~2.3]{Heuberger-Kropf-Prodinger:2015:output}. As
\cite[Lemma~2.3]{Heuberger-Kropf-Prodinger:2015:output} proves more than we
need here and depends on the notions of that article, we extract the relevant
parts of \cite{Heuberger-Kropf-Prodinger:2015:output} to provide a
self-contained (apart from Perron--Frobenius theorem) proof of
Lemma~\ref{lemma:Perron--Frobenius-again}.
\input{perron_frobenius}

\subsection{Analysis of Output Sums of Transducers}

  We consider the states of $\calT$ to be numbered by $\set{1, \ldots, d}$ for some
  positive integer $d\ge 1$ such that the initial state is state~$1$. We set
  $\calT_j(n)$ to be the sum of the output labels (including the final output
  label) encountered when the transducer reads the $q$-ary expansion of $n$ when
  starting in state~$j$. By construction, we have $\calT(n)=\calT_1(n)$ and
  $\calT_j(0)$ is the final output label of state~$j$. We set
  $y(n)=\bigl(\calT_1(n), \ldots, \calT_d(n)\bigr)$.
  For $0\le r<q$, we define the $d\times d$-dimensional $\set{0, 1}$-matrix $P_r$ in such a
  way that there is a one in row~$j$, column~$k$ if and only if there is a
  transition from state~$j$ to state~$k$ with input label $r$. The vector $o_r$
  is defined by setting its $j$th coordinate to be the output label of the transition
  from state~$j$ with input label $r$.

For $n_0\ge 1$, we set
\begin{equation*}
\calX(s)=\sum_{n\ge 1}n^{-s}\calT(n),\qquad
\calY_{n_0}(s)=\sum_{n\ge n_0}n^{-s}y(n),\qquad
\zeta_{n_0}(s, \alpha)=\sum_{n\ge n_0}(n+\alpha)^{-s}.
\end{equation*}
The last Dirichlet series is a truncated version of the Hurwitz zeta function.

\begin{corollary}
  \label{corollary:transducer-main}
  Let $\calT$ be a transducer as described at the beginning of
  this section. Let $M$ be the adjacency matrix and $p$ be the final period of
  the underlying digraph. For $\lambda\in\C$ let $m(\lambda)$
  be the size of the largest Jordan block associated with the eigenvalue
  $\lambda$ of $M$.

  Then the sequence
  $n\mapsto\calT(n)$ is a $q$-regular sequence and
  \begin{equation}\label{eq:transducer:summatory-as-fluctuation}
  \begin{aligned}
    \sum_{0\le n<N}\calT(n) = e_\calT N\log_q N  &+ N\Phi(\log_q N)\\
    &+ \sum_{\substack{\lambda\in\sigma(M)\\
        1<\abs{\lambda}<q
      }} N^{\log_q \lambda} \sum_{0\le k<m(\lambda)}(\log_q N)^k\Phi_{\lambda
      k}(\log_q N)\\
    &+ \Oh[\big]{(\log N)^{\max\setm{m(\lambda)}{\abs{\lambda}=1}}}
  \end{aligned}
  \end{equation}
  for some continuous $p$-periodic function $\Phi$, some continuous
  $1$-periodic functions~$\Phi_{\lambda k}$ for $\lambda\in\sigma(M)$ with $1<\abs{\lambda}<q$ and $0\le
  k<m(\lambda)$ and some constant
  $e_\calT$.

  Furthermore,
  \begin{equation*}
    \Phi(u)=\sum_{\ell\in\Z}\varphi_\ell\exp\Bigl(\frac{2\ell\pi i}{p}u\Bigr)
  \end{equation*}
  with
  \begin{equation*}
    \varphi_\ell = \Res[\Big]{\frac{\calX(s)}{s}}{s=1+\frac{2\ell\pi
        i}{p\log q}}
  \end{equation*}
  for $\ell\in\Z$.
  The Fourier series expansion of $\Phi_{\lambda k}$ for $\lambda\in\sigma(M)$
  with $1<\abs{\lambda}<q$ is given in Theorem~\ref{theorem:simple}.

  The Dirichlet series $\calY_{n_0}$ satisfies the functional equation
  \begin{equation}\label{eq:transducer-functional-equation}
    \begin{aligned}
    \bigl(I-q^{-s}M\bigr)\calY_{n_0}(s) &= \sum_{n_0\le n<qn_0} n^{-s}y(n)
    + q^{-s}\sum_{0\le r<q}\zeta_{n_0}\bigl(s, \tfrac{r}{q}\bigr)o_r\\
    &\phantom{={}}+ q^{-s}\sum_{0\le r<q}P_r\sum_{k\ge 1}\binom{-s}{k}\Bigl(\frac rq\Bigr)^k\calY_{n_0}(s+k).
    \end{aligned}
  \end{equation}
\end{corollary}

Note that the functional equation~\eqref{eq:transducer-functional-equation}
is preferrable over the functional equation given in
Theorem~\ref{theorem:Dirichlet-series} for the generic case
of a regular sequence: The generic functional equation
suggests a double pole at $s=1+\chi_\ell$ for all $\ell\in\Z$
whereas the occurrence of the Hurwitz zeta function
in~\eqref{eq:transducer-functional-equation} shows that
there is a double pole $s=1$ but single poles at $s=1+\chi_\ell$
for all $\ell\in\Z\setminus\{0\}$. Numerically, the same occurrence
of the Hurwitz zeta function is also advantageous because it
allows to decouple the problem.

\subsection{Proof of Corollary~\ref{corollary:transducer-main}}

\begin{proof}[Proof of Corollary~\ref{corollary:transducer-main}]
  The proof is split into several steps.

  \proofparagraph{Recursive Description}
  We set
  $v(n)=\bigl(\calT_1(n), \ldots, \calT_d(n), 1\bigr)^\top$\!.
  For $1\le j\le d$ and $0\le r<q$, we define $t(j, r)$ and $o(j, r)$ to be the
  target state and output label of the unique transition from
  state $j$ with input label $r$, respectively. Therefore,
  \begin{equation}\label{eq:transducer-to-matrix-product}
    \calT_j(qn+r) = \calT_{t(j, r)}(n) + o(j, r)
  \end{equation}
  for $1\le j\le d$, $n\ge 0$, $0\le r<q$ with $qn+r>0$.

  For $0\le r<q$, define $A_r=(a_{rjk})_{1\le j,\, k\le d+1}$ by
  \begin{equation*}
    a_{rjk} =
    \begin{cases}
      \iverson{t(j, r) = k}& \text{if $j$, $k\le d$,}\\
      o(j, r)& \text{if $j\le d$, $k=d+1$,}\\
      \iverson{k=d+1}& \text{if $j=d+1$.}
    \end{cases}
  \end{equation*}
  Then \eqref{eq:transducer-to-matrix-product} is equivalent to
  \begin{equation*}
    v(qn+r) = A_r v(n)
  \end{equation*}
  for $n\ge 0$, $0\le r<q$ with $qn+r>0$. Defining $f(n)$ as in
  \eqref{eq:regular-matrix-sequence} for these $A_r$, we see that
  $v(n)=f(n)v(0)$.

  \proofparagraph{$q$-Regular Sequence}
  If we insist on a proper formulation as a regular sequence, we rewrite
  \eqref{eq:transducer-to-matrix-product} to
  \begin{equation}\label{eq:transducer-to-regular-sequence}
    \calT_j(qn+r)= \calT_{t(j,r)}(n) + o(j, r) +
    \iverson{r=0}\iverson{n=0}\bigl(\calT_j(0)-\calT_{t(j,0)}(0)-o(j, 0)\bigr)
  \end{equation}
  for $1\le j\le d$, $n\ge 0$, $0\le r<q$. Setting $\tildev(n)=\bigl(\calT_1(n), \ldots,
  \calT_d(n), 1, \iverson{n=0}\bigr)$ and
  $\tildeA_r=(\tildea_{rjk})_{1\le j,\, k\le d+2}$ with
  \begin{equation*}
    \tildea_{rjk} =
    \begin{cases}
      \iverson{t(j, r) = k}& \text{if $j$, $k\le d$,}\\
      o(j, r)& \text{if $j\le d$, $k=d+1$,}\\
      \iverson{r=0}\bigl(\calT_j(0)-\calT_{t(j,0)}(0)-o(j, 0)\bigr)& \text{if $j\le d$, $k=d+2$,}\\
      \iverson{k=d+1}& \text{if $j=d+1$,}\\
      \iverson{k=d+2}\iverson{r=0}& \text{if $j=d+2$,}
    \end{cases}
  \end{equation*}
  the system~\eqref{eq:transducer-to-regular-sequence} is equivalent to
  \begin{equation*}
    \tildev(qn+r) = \tildeA_r \tildev(n)
  \end{equation*}
  for $n\ge 0$, $0\le r<q$.

  \proofparagraph{Eigenvalue~$1$}
  By construction, the matrices $A_r$ have the shape
  \begin{equation*}
    A_r = \left(
      \begin{array}{c|c}
        P_r&o_r\\\hline
        0&1
      \end{array}
      \right).
  \end{equation*}
  It is
  clear that $(0, \ldots, 0, 1)$ is a left eigenvector of $A_r$ associated with
  the eigenvalue~$1$.

  \proofparagraph{Joint Spectral Radius}
  We claim that $A_0, \ldots, A_{q-1}$ have joint spectral radius $1$. Let
  $\inftynorm{\,\cdot\,}$ denote the maximum norm of complex vectors as well as the induced
  matrix norm, i.e., the maximum row sum norm. Let $j_1$, \ldots,
  $j_\ell\in\set{0,\ldots, q-1}$. It is easily shown by induction on $\ell$
  that
  \begin{equation*}
    A_{j_1}\dotsm A_{j_\ell}=\left(
      \begin{array}{c|c}
        P&b_P\\\hline
        0&1
      \end{array}
\right)
  \end{equation*}
  for some $P\in\C^{d\times d}$ and $b_P\in\C^d$ with $\inftynorm{P}\le 1$ and $\inftynorm{b_P}\le \ell \max_{0\le
    r<q}\inftynorm{o_r}$.
  Thus, we obtain
  \begin{equation*}
    \inftynorm{A_{j_1}\dotsm A_{j_\ell}}\le 1+\ell\max_{0\le
      r<q}\inftynorm{o_r}.
  \end{equation*}
  As $1$ is an eigenvalue of each matrix~$A_r$ for $0\le r<q$,
  the joint spectral radius equals~$1$, which proves the claim.

  \proofparagraph{Eigenvectors and Asymptotics}
  We now consider $C=\sum_{0\le r<q}A_r$. It has the shape
  \begin{equation*}
    C = \left(
      \begin{array}{c|c}
        M&b_M\\\hline
        0&q
      \end{array}
      \right)
  \end{equation*}
  where $b_M$ is some complex vector.

  Let $w_1$, \ldots, $w_\ell$ be a linearly independent system of left
  eigenvectors of $M$ associated with the eigenvector $q$.
  If $w_j b_M=0$ for $1\le j\le \ell$, then $(w_1, 0)$,
  \ldots, $(w_\ell, 0), (0, 1)$ is a linearly independent system of left
  eigenvectors of $C$ associated with the eigenvalue $q$. In that case
  and because of Lemma~\ref{lemma:Perron--Frobenius-again},
  algebraic and geometric multiplicities of $q$ as an eigenvalue of $C$ are
  both equal to $\ell+1$.

  Otherwise, assume without loss of generality that $w_1 b_M=1$. Then
  \begin{equation*}
    \bigl(w_2 - (w_2 b_M)w_1, 0\bigr),\,
    \ldots,\,
    \bigl(w_\ell - (w_\ell b_M)w_1, 0\bigr),\,
    \bigl(0, 1\bigr)
  \end{equation*}
  is a linearly independent
  system of left eigenvectors of $C$ associated with the eigenvalue
  $q$. Additionally, $(w_1, 0)$ is a generalised left eigenvector of rank $2$
  of $C$ associated with the eigenvalue $q$ with $(w_1, 0)(C-qI)=(0, 1)$. As
  noted above, the vector
  $(0, 1)$ is a left eigenvector to each matrix $A_0$, \ldots, $A_{q-1}$.

  Similarly, it is easily seen that any left eigenvector of $M$ associated with
  some eigenvalue $\lambda\neq q$ can  be extended uniquely to a left
  eigenvector of $C$ associated with the same eigenvalue. The same is true for
  chains of generalised left eigenvectors associated with $\lambda\neq q$.

  Therefore, in both of the above cases, Theorem~\ref{theorem:contribution-of-eigenspace}
  yields
  \begin{equation*}
    \begin{aligned}
    \sum_{0\le n<N}\calT(N) = e_\calT N\log_q N &+ \sum_{\zeta \in U_p} N^{\log_q
      (q\zeta)}\Phi_{(q\zeta)}(\fractional{\log_q N}) \\
    &+ \sum_{\substack{\lambda\in\sigma(M)\\
        1<\abs{\lambda}<q
      }} N^{\log_q \lambda} \sum_{0\le k<m(\lambda)}(\log_q N)^k\Phi_{\lambda
      k}(\log_q N)\\
    &+ \Oh[\big]{(\log N)^{\max\setm{m(\lambda)}{\abs{\lambda}=1}}}
    \end{aligned}
  \end{equation*}
  for some constant $e_\calT$ (which vanishes in the first case) and some
  $1$-periodic continuous functions $\Phi_{(q\zeta)}$ and $\Phi_{\lambda k}$ where $\zeta$ runs through
  the $p$th roots of unity~$U_p$ and $\lambda$ through the eigenvalues of $M$ with
  $1<\abs{\lambda}<q$ and $0\le k<m(\lambda)$.

  Proposition~\ref{proposition:symmetric-eigenvalues}
  leads to \eqref{eq:transducer:summatory-as-fluctuation}.

  \proofparagraph{Fourier Coefficients}
  By Theorem~\ref{theorem:simple}, we have
  \begin{equation*}
    \Phi_{(q\zeta)}(u)=\sum_{\ell\in\Z}\varphi_{(q\zeta)\ell}\exp(2\ell\pi i u)
  \end{equation*}
  with
  \begin{equation*}
    \varphi_{(q\zeta)\ell}=\Res[\Big]{\frac{\calT(0)+\calX(s)}{s}}{s=1+\log_q \zeta + \frac{2\ell\pi
        i}{\log q}}
  \end{equation*}
  for a $p$th root of unity $\zeta \in U_p$ and $\ell\in\Z$.
  Therefore and by noting that $\calT(0)$ does not contribute
  to the residue, Proposition~\ref{proposition:symmetric-eigenvalues}
  leads to the Fourier series given in the
  corollary.

  \proofparagraph{Functional Equation}
  By \eqref{eq:transducer-to-matrix-product}, we have
  \begin{align*}
    \calY_{n_0}(s) &= \sum_{n_0\le n<qn_0} n^{-s}y(n) + \sum_{n\ge
      n_0}\sum_{0\le r<q}(qn+r)^{-s}y(qn+r)\\
    &= \sum_{n_0\le n<qn_0} n^{-s}y(n) + \sum_{n\ge
      n_0}\sum_{0\le r<q}(qn+r)^{-s}\bigl(P_r y(n) + o_r\bigr)\\
    &= \sum_{n_0\le n<qn_0} n^{-s}y(n) + q^{-s}\sum_{0\le r<q}P_r
\sum_{n\ge
      n_0}\Bigl(n+\frac{r}{q}\Bigr)^{-s}y(n) \\
    &\hspace*{8.95em}
    + q^{-s}\sum_{0\le r<q}\zeta_{n_0}\bigl(s, \tfrac{r}{q}\bigr)o_r.
  \end{align*}
  Using Lemma~\ref{lemma:shifted-Dirichlet}
  yields the result.
\end{proof}

%%% Local Variables:
%%% mode: latex
%%% TeX-master: "regular-sequences.tex"
%%% End:

%% file: perron_frobenius.tex
\begin{proof}
  As usual, the condensation of $D$ is
  the graph resulting from contracting each component of the original digraph
  to a single new vertex. By construction, the condensation is acyclic.

  We choose a refinement of the partial order of the components given by the successor relation in
  the condensation to a linear order in such a way that the final components
  come last. Note that this implies that if there is an arc from one component to another, the
  former component comes before the latter component in our linear order. We
  then denote the components by $\calC_1$, \ldots, $\calC_k$, $\calC_{k+1}$,
  \ldots, $\calC_{k+\ell}$ where the first $k$ components are non-final and
  the last $\ell$ are final.
  Without loss of generality, we assume that the vertices of the original digraph $D$ are labeled such that
  vertices within a component get successive labels and such that the linear
  order of the components established above is respected.

  Therefore, the adjacency matrix $M$ is an upper block triagonal matrix of the
  shape
  \begin{equation*}
    M=\left(
    \begin{array}{c|c|c|c|c|c}
      M_1&\star&\star&\star&\star&\star\\\hline
      0&\raisebox{-0.5ex}[-0.25ex]{$\ddots$}&\star&\star&\star&\star\\\hline
      0&0&M_k&\star&\star&\star\\\hline
      0&0&0&M_{k+1}&0&0\\\hline
      0&0&0&0&\raisebox{-0.5ex}[-0.25ex]{$\ddots$}&0\\\hline
      0&0&0&0&0&M_{k+\ell}
    \end{array}
    \right)
  \end{equation*}
  where $M_j$ is the adjacency matrix of the component $\calC_j$.

  Each row of the non-negative square matrix $M$ has sum $q$ by construction.
  Thus $\inftynorm{M}=q$ and therefore the spectral radius of $M$ is bounded
  from above by $q$. As the all ones vector is obviously a right eigenvector
  associated with the eigenvalue $q$ of $M$, the spectral radius of $M$ equals
  $q$. The same argument applies to $M_{k+1}$, \ldots,
  $M_{k+\ell}$.

  By construction, the
  matrices $M_{k+1}$, \ldots, $M_{k+\ell}$ are irreducible.
  For $1\le j\le \ell$ all eigenvalues $\lambda$ of $M_{k+j}$ of modulus $q$ have
  algebraic and geometric multiplicities $1$ by Perron--Frobenius theory
  and $\lambda = q \zeta$ for some $p_{k+j}$th root of unity $\zeta$ where $p_{k+j}$ is
  the period of $\calC_{k+j}$.

  By construction, the vertices of the components $\calC_j$ for $1\le j\le k$
  have out-degree at most $q$. We add loops to these vertices to increase
  their out-degree to $q$, resulting in $\tildecalC_j$. The corresponding
  adjacency matrices are denoted by $\tildeM_j$. By the above argument,
  $\tildeM_j$ has spectral radius $q$ for $1\le j\le k$. As $M_j\le \tildeM_j$ (component-wise)
  and $M_j\neq \tildeM_j$ by construction, the spectral radius of $M_j$ is strictly less than
  $q$ by \cite[Theorem~8.8.1]{Godsil-Royle:2001:alggraphtheory}.

  A left eigenvector $v_j$ of $M_{k+j}$ for $1\le j\le \ell$ can easily be
  extended to a left eigenvector $(0, \ldots, 0, v_j, 0, \ldots, 0)$ of
  $M$. This observation shows that the geometric multiplicity of any eigenvalue
  of $M$ of modulus $q$ is at least its algebraic multiplicity. This concludes
  the proof.
\end{proof}
%%% Local Variables:
%%% mode: latex
%%% TeX-master: "regular-sequences.tex"
%%% End:

%% file: esthetic-numbers.tex
Let again be $q\geq2$ a fixed integer. We call a non-negative integer~$n$ a
\emph{$q$-esthetic number} (or simply an \emph{esthetic number}) if its
$q$-ary digit expansion $r_{\ell-1} \dots r_0$ satisfies
$\abs{r_j - r_{j-1}} = 1$ for all $j\in\set{1,\dots,\ell-1}$;
see~De~Koninck and Doyon~\cite{Koninck-Doyon:2009:esthetic-numbers}.

In~\cite{Koninck-Doyon:2009:esthetic-numbers} the authors count
$q$-esthetic numbers with a given length of their $q$-ary digit
expansion. They provide an explicit (in form of a sum of $q$ summands)
as well as an asymptotic formula for
these counts. We aim for a more precise analysis and head for an
asymptotic description of the amount of $q$-esthetic numbers up the an
arbitrary value~$N$ (in contrast to only powers of~$q$
in~\cite{Koninck-Doyon:2009:esthetic-numbers}).

\subsection{A $q$-Linear Representation}

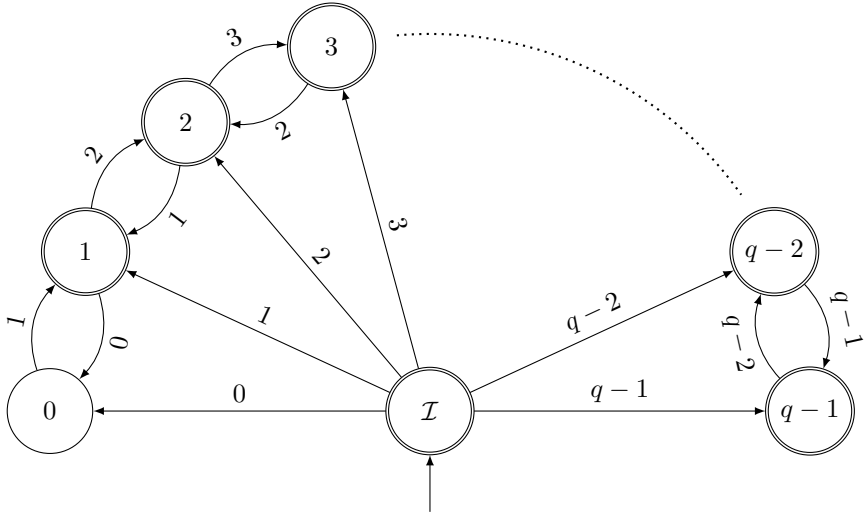
\begin{figure}
  \centering

  \begin{tikzpicture}[auto,
    initial text=, initial distance=5ex,
    >=latex,
    accepting text=,
    every state/.style={minimum size=3.2em}]

    \node[state, initial below, accepting] (I) at (0.000000, 0.000000) {$\mathcal{I}$};

    \node[state] (e0) at (180:5) {$0$};
    \node[state, accepting] (e1) at (155:5) {$1$};
    \node[state, accepting] (e2) at (130:5) {$2$};
    \node[state, accepting] (e3) at (105:5) {$3$};

    \draw[dotted, thick] (95:5) arc (95:35:5);

    \node[state, accepting] (eq2) at (25:5) {$q-2$};
    \node[state, accepting] (eq1) at (0:5) {$q-1$};

    \path[->] (I) edge node[rotate=0, anchor=south] {$0$} (e0);
    \path[->] (I) edge node[rotate=-25, anchor=south] {$1$} (e1);
    \path[->] (I) edge node[rotate=-50, anchor=south] {$2$} (e2);
    \path[->] (I) edge node[rotate=-75, anchor=south] {$3$} (e3);
    \path[->] (I) edge node[rotate=25, anchor=south] {$q-2$} (eq2);
    \path[->] (I) edge node[rotate=0, anchor=south] {$q-1$} (eq1);

    \path[->] (e0) edge[bend left] node[rotate=77.5, anchor=south] {$1$} (e1);
    \path[->] (e1) edge[bend left] node[rotate=77.5, anchor=north] {$0$} (e0);
    \path[->] (e1) edge[bend left] node[rotate=52.5, anchor=south] {$2$} (e2);
    \path[->] (e2) edge[bend left] node[rotate=52.5, anchor=north] {$1$} (e1);
    \path[->] (e2) edge[bend left] node[rotate=27.5, anchor=south] {$3$} (e3);
    \path[->] (e3) edge[bend left] node[rotate=27.5, anchor=north] {$2$} (e2);

    \path[->] (eq2) edge[bend left] node[rotate=-77.5, anchor=south] {$q-1$} (eq1);
    \path[->] (eq1) edge[bend left] node[rotate=-77.5, anchor=north] {$q-2$} (eq2);

  \end{tikzpicture}
  
  \caption{Automaton~$\mathcal{A}$ recognizing esthetic numbers.}
  \label{fig:esthetic-automaton}
\end{figure}

The language consisting of the $q$-ary digit expansions (seen as words
of digits) which are $q$-esthetic
is a regular language, because it is recognized by the
automaton~$\mathcal{A}$ in
Figure~\ref{fig:esthetic-automaton}. Therefore, the indicator sequence
of this language, i.e., the $n$th entry is $1$ if $n$ is $q$-esthetic
and $0$ otherwise, is a $q$-automatic sequence and therefore also
$q$-regular. Let us name this sequence~$x(n)$.

Let $A_0$, \dots, $A_{q-1}$ be the transition matrices of the
automaton~$\mathcal{A}$, i.e., $A_r$ is the adjacency matrix of the
directed graph induced by a transition with digit~$r$.
To make this more explicit, we have
the following $(q+1)$-dimensional square
matrices: Each row and column corresponds to the states~$0$, $1$,
\dots, $q-1$, $\mathcal{I}$. In matrix~$A_r$, the only non-zero entries
are in column~$r\in\set{0,1,\dots,q-1}$, namely $1$ in the rows~$r-1$ and $r+1$ (if
available) and in row~$\mathcal{I}$ as there are transitions from
these states to state~$r$ in the automaton~$\mathcal{A}$.

Let us make this more concrete by considering $q=4$. We obtain the matrices
\begin{align*}
  A_0 &=
  \begin{pmatrix}
    0 & 0 & 0 & 0 & 0 \\
    1 & 0 & 0 & 0 & 0 \\
    0 & 0 & 0 & 0 & 0 \\
    0 & 0 & 0 & 0 & 0 \\
    1 & 0 & 0 & 0 & 0
  \end{pmatrix},
  &
  A_1 &=
  \begin{pmatrix}
    0 & 1 & 0 & 0 & 0 \\
    0 & 0 & 0 & 0 & 0 \\
    0 & 1 & 0 & 0 & 0 \\
    0 & 0 & 0 & 0 & 0 \\
    0 & 1 & 0 & 0 & 0
  \end{pmatrix},
  \\
  A_2 &=
  \begin{pmatrix}
    0 & 0 & 0 & 0 & 0 \\
    0 & 0 & 1 & 0 & 0 \\
    0 & 0 & 0 & 0 & 0 \\
    0 & 0 & 1 & 0 & 0 \\
    0 & 0 & 1 & 0 & 0
  \end{pmatrix},
  &
  A_3 &=
  \begin{pmatrix}
    0 & 0 & 0 & 0 & 0 \\
    0 & 0 & 0 & 0 & 0 \\
    0 & 0 & 0 & 1 & 0 \\
    0 & 0 & 0 & 0 & 0 \\
    0 & 0 & 0 & 1 & 0
  \end{pmatrix}.
\end{align*}

We are almost at a $q$-linear representation of our sequence; we still
need vectors on both sides of the matrix products. We have
\begin{equation*}
  x(n) = e_{q+1}\, A_{r_0} \cdots A_{r_{\ell-1}} v(0)
\end{equation*}
for $r_{\ell-1} \dots r_0$ being the $q$-ary expansion of~$n$ and vectors
$e_{q+1}=\begin{pmatrix}0& \dotsc& 0&1\end{pmatrix}$ and
$v(0)=\begin{pmatrix}0&1& \dotsc& 1\end{pmatrix}^\top$.
As $A_0 v(0)=0\neq v(0)$, this is not a linear representation of a regular
sequence. Thus we cannot use Theorem~\ref{theorem:simple}, but need to use
Theorem~\ref{theorem:contribution-of-eigenspace}. However, the difference is
slight: we simply cannot omit the contributions of the constant vector $Kv(0)$.
However, it will turn out that the joint spectral radius is $1$, so the
contribution will be absorbed by the error term anyway.

To see that the above holds, we have two different interpretations:
The first is that the row vector
\begin{equation*}
  w(n) = e_{q+1}\, A_{r_0} \cdots A_{r_{\ell-1}}
\end{equation*}
is the unit vector corresponding to the most significant digit
of the $q$-ary expansion of~$n$ or, in view of the
automaton~$\mathcal{A}$, corresponding to the final state.
Note that we read the digit expansion from the least significant digit
to the most significant one
(although it would be possible the other way round as well).
We have $w(0)=e_{q+1}$
which corresponds to the empty word and
being in the initial state~$\mathcal{I}$ in the automaton.
The vector~$v(0)$ corresponds to the fact that
all states of~$\mathcal{A}$ except~$0$ are accepting.

The other interpretation is: The $r$th component of the column vector
\begin{equation*}
  v(n) = A_{r_0} \cdots A_{r_{\ell-1}} v(0)
\end{equation*}
has the following two meanings:
\begin{itemize}
\item In the automaton~$\mathcal{A}$, we start in state $r$ and then
  read the digit expansion of $n$. The $r$th component is then the indicator
  function whether we remain esthetic, i.e., end in an accepting
  state.
\item To a word ending with $r$ we append the digit expansion of
  $n$. The $r$th component is then the indicator function whether the result
  is an esthetic word.
\end{itemize}

At first glance, our problem here seems to be a special case of the
transducers studied in Section~\ref{sec:transducer}. However, the
automaton~$\mathcal{A}$ is not complete. Adding a sink to have a
formally complete automaton, however, adds an eigenvalue $q$ and thus
a much larger dominant asymptotic term, which would then be multiplied
by~$0$. Therefore, the results
of~\cite{Heuberger-Kropf-Prodinger:2015:output} do not apply to this
case here.

\subsection{Full Asymptotics}

We now formulate our main result for the amount of esthetic numbers
smaller than a given integer~$N$. We abbreviate this amount by
\begin{equation*}
  X(N) = \sum_{0 \le n < N} x(n)
\end{equation*}
and have the following corollary.

\begin{corollary}
  \label{corollary:esthetic:asy}
  Fix an integer~$q\geq2$.
  Then the number~$X(N)$ of $q$-esthetic numbers smaller than $N$ is
  \begin{multline}\label{eq:esthetic:asy-main}
    X(N) = \sum_{j\in\set{1,2,\dots,\ceil{\frac{q-2}{3}}}}
    N^{\log_q (2\cos(j\pi/(q+1)))} \Phi_{j}(2\fractional{\log_{q^2} N}) \\
    + \Oh[\big]{(\log N)^{\iverson{q \equiv -1 \tpmod 3}}}
  \end{multline}
  with $2$-periodic continuous functions~$\Phi_{j}$.
  Moreover, we can effectively compute the Fourier coefficients of
  each~$\Phi_{j}$ (as explained in Part~\ref{part:numerical}).
  If $q$ is even, then the functions $\Phi_{j}$ are actually $1$-periodic.
  If $q$ is odd, then the functions $\Phi_j$ for even $j$ vanish.
\end{corollary}

If $q=2$, then the corollary results in $X(N)=\Oh{\log N}$.
However, for each length, the only word of digits satisfying the
esthetic number condition has alternating digits $0$ and $1$,
starting with~$1$ at its most significant digit. The
corresponding numbers~$n$ form the so-called
Lichtenberg sequence~\oeis{A000975}.

Back to a general~$q$: For the asymptotics,
the main quantities influencing the growth are the
eigenvalues of the matrix~$C = A_0+\dots+A_{q-1}$. Continuing our
example $q=4$ above, this matrix is
\begin{equation*}
  C = A_0 + A_1 + A_2 + A_3 =
  \begin{pmatrix}
    0 & 1 & 0 & 0 & 0 \\
    1 & 0 & 1 & 0 & 0 \\
    0 & 1 & 0 & 1 & 0 \\
    0 & 0 & 1 & 0 & 0 \\
    1 & 1 & 1 & 1 & 0
  \end{pmatrix},
\end{equation*}
and its eigenvalues are
$\pm 2\cos(\frac{\pi}{5})=\pm \frac12\bigl(\sqrt{5} + 1\bigr) = \pm1.618\dots$,
$\pm 2\cos(\frac{2\pi}{5})=\pm \frac12\bigl(\sqrt{5} - 1\bigr) = \pm0.618\dots$
and $0$, all with algebraic and geometric multiplicity $1$. Therefore it turns out that
the growth of the main term is
$N^{\log_4(\sqrt{5} + 1) - \frac12}=N^{0.347\dots}$, see
Figure~\ref{fig:fluct-esthetic}.
The first few Fourier coefficients are shown in Table~\ref{table:esthetic:fourier}.

\begin{figure}
  \centering
  \includegraphics{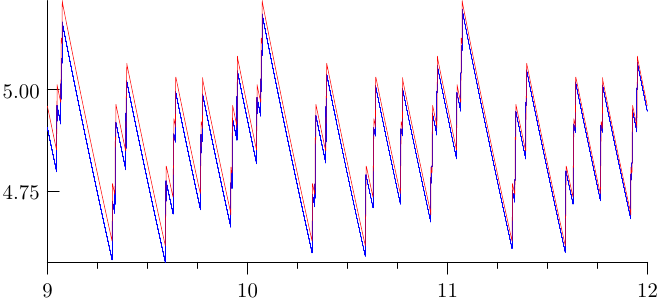}
  \caption{Fluctuation in the main term of the asymptotic expansion of $X(N)$
    for $q=4$.
    The figure shows $\f{\Phi_1}{u}$ (red) approximated by
    its trigonometric polynomial of degree~$1999$ as well as
    $X(4^u) / N^{u(\log_4(\sqrt{5} + 1) - \frac12)}$ (blue).}
  \label{fig:fluct-esthetic}
\end{figure}

\begin{table}
  \centering
  \begin{equation*}\footnotesize
    \begin{array}{r|l}
      \multicolumn{1}{c|}{\ell} &
      \multicolumn{1}{c}{\varphi_{1\ell}} \\
      \hline
      0&\phantom{-}4.886821584515\\
      1&\phantom{-}0.036565359077 - 0.012421753685i\\
      2&\phantom{-}0.0131103199420 - 0.017152133508i\\
      3&-0.0023895069366 - 0.0506880727105i\\
      4&-0.017328669452 + 0.025036392542i\\
      5&\phantom{-}0.011186380630 - 0.0066357472861i\\
      6&\phantom{-}0.0086354015002 + 0.018593736873i\\
      7&-0.014899262928 + 0.0297436287202i\\
      8&-0.003867454968 + 0.0064534688733i\\
      9&\phantom{-}0.0033747695643 + 0.006159612843i\\
      10&-0.002149675882 + 0.006474570022i
    \end{array}
  \end{equation*}
  \caption{Fourier coefficients of~$\Phi_1$ for $q=4$
    (Corollary~\ref{corollary:esthetic:asy}). All stated digits are
    correct; see also Part~\ref{part:numerical}.}
  \label{table:esthetic:fourier}
\end{table}

\subsection{Eigenvectors}

\input{esthetic-eigenvectors}

\subsection{Proof of the Asymptotic Result}

\begin{proof}[Proof of Corollary~\ref{corollary:esthetic:asy}]
  We work out the conditions and parameters for using
  Theorem~\ref{theorem:simple}.

  \proofparagraph{Joint Spectral Radius}
  As all the square matrices $A_0$, \dots, $A_{q-1}$ have a maximum
  absolute row sum norm equal to $1$, the joint spectral radius of
  these matrices is bounded by~$1$.

  Let $r\in\set{1,\dots,q-1}$. Then any product with alternating
  factors $A_{r-1}$ and $A_r$, i.e., a finite product
  $A_{r-1}A_rA_{r-1}\cdots$, has absolute row sum norm at least~$1$ as
  the word $(r-1)r(r-1)\dots$ is $q$-esthetic. Therefore the joint
  spectral radius of $A_{r-1}$ and $A_r$ is at
  least~$1$. Consequently, the joint spectral radius of $A_0$, \dots,
  $A_{q-1}$ equals~$1$.

  \proofparagraph{Asymptotics}
  We apply our Theorem~\ref{theorem:simple}.
  We have $\lambda_j=-\lambda_{q+1-j}$, so we combine our approach
  with Proposition~\ref{proposition:symmetric-eigenvalues}. Moreover,
  we have $\lambda_j>1$ iff $\frac{j}{q+1}<\frac{1}{3}$ iff
  $j\leq\ceil{\frac{q-2}{3}}$.
  This results
  in~\eqref{eq:esthetic:asy-main}.

  We now assume that $q$ is even. In this case, we still have to show
  that the functions $\Phi_j$ are actually $1$-periodic. We now need to
  use Theorem~\ref{theorem:contribution-of-eigenspace}. Let $w_1$, $w_2$,
  \ldots, $w_{q-1}$, $w_q$ be the rows of $T$ where the order is chosen in such
  a way that
  \begin{equation*}
    J=\diag\Bigl(2\cos\Bigl(\frac{\pi}{q+1}\Bigr), \ldots,
                 2\cos\Bigl(\frac{q\pi}{q+1}\Bigr), 0\Bigr).
  \end{equation*}
  We write $e_{q+1}=\sum_{k=1}^q c_k w_k$ for suitable $c_k\in\R$. Setting
  $c\coloneqq \begin{pmatrix}c_1&c_2&\cdots&c_q\end{pmatrix}$, this means that
  $e_{q+1}=cT$, or equivalently, $c=e_{q+1} T^{-1}$. The columns of $T^{-1}$ are
  the right eigenvectors of $C$ described in
  Proposition~\ref{proposition:esthetic-eigenvectors}. Then
  Proposition~\ref{proposition:esthetic-eigenvectors}~(\ref{enu:esthetic-eigenvectors:neq0}) implies that $c_k=0$ for
  even $k$ with $1\le k\le q$. This means that all fluctuations corresponding
  to eigenvalues $2\cos(k\pi/(q+1))$ for even $k$ with $1\le k\le q$ are
  multiplied by $0$ and do not contribute to the result.
  As $\abs{\cos(\frac{q+1-k}{q+1}\pi)}=\abs{\cos(\frac{k}{q+1}\pi)}$, but
  $q+1-k$ and $k$ have different parities, there is no need to use
  Proposition~\ref{proposition:symmetric-eigenvalues} and all fluctuations are
  $1$-periodic.

  The same argument can be used for the case of odd $q$, but in this case,
  $q+1-k$ and $k$ have the same parity. So
  Proposition~\ref{proposition:symmetric-eigenvalues}
  is used for odd $k$, and fluctuations to both eigenvalues $2\cos(k\pi/(q+1))$
  and $2\cos((q+1-k)\pi/(q+1))$ vanish for even~$k$.

  \proofparagraph{Fourier Coefficients}
  We can compute the Fourier coefficients according to
  Theorem~\ref{theorem:simple} and
  Proposition~\ref{proposition:symmetric-eigenvalues};
  see also Part~\ref{part:numerical}.
\end{proof}

%%% Local Variables:
%%% mode: latex
%%% TeX-master: "regular-sequences"
%%% End:

%% file: esthetic-eigenvectors.tex
Before proving Corollary~\ref{corollary:esthetic:asy}, we collect
information on the eigenvalues of $C$.

  The matrix $C = A_0+\dots+A_{q-1}$ has a block decomposition into
  \begin{equation*}
    C = 
    \left(\begin{array}{c|c}
      M & \mathbf{0} \\
      \hline
      \mathbf{1} & 0
    \end{array}\right)
  \end{equation*}
  for vectors~$\mathbf{0}$ (vector of zeros) and $\mathbf{1}$
  (vector of ones) of suitable dimension.
  Therefore, one eigenvalue of~$C$ is~$0$ and the others are the eigenvalues
  of~$M$.

In contrast to \cite[Sections~4 and~5]{Koninck-Doyon:2009:esthetic-numbers},
we use the Chebyshev polynomials\footnote{%
    Chebyshev polynomials are frequently occurring phenomena
    in lattice path analysis, see for instance~\cite{Bruijn-Knuth-Rice:1972,
      Flajolet:1980:combinat-continued-fractions}.
    We have such a lattice path here,
    so their appearance is not surprising.}%
  \footnote{%
    Up to replacing $2X$ by $X$, the polynomials~$U_n$ used here
    correspond to the
    polynomials~$p_n$ used in~\cite{Koninck-Doyon:2009:esthetic-numbers}.}
$U_n$ of the second kind defined by
\begin{align*}
  U_0(X) &= 1, & U_1(X)&=2X, & U_{n+1}(X)=2X\,U_n(X)-U_{n-1}(X)
\end{align*}
for $n\ge 1$.
It is well-known that
\begin{equation}\label{eq:U-n-sinus-relation}
  U_n(\cos\varphi)=\frac{\sin((n+1)\varphi)}{\sin(\varphi)}
\end{equation}
and, as a consequence, the roots of $U_n$ are given by
\begin{equation*}
  \cos\Bigl(\frac{k\pi}{n+1}\Bigr), \qquad 1\le k\le n,
\end{equation*}
for $n\ge 1$.

The following lemma is similar to~\cite[Proposition~3]{Koninck-Doyon:2009:esthetic-numbers}.
\begin{lemma}\label{lemma:eigenvector-M}
  Let $v\neq 0$ be a vector and $\lambda\in\C$.

  Then $v$ is an eigenvector to the eigenvalue $\lambda$ of $M$ if and only if
  $\lambda = 2\cos(\frac{k\pi}{q+1})$ for some $1\le k\le q$ and
  \begin{equation*}
    v=\Bigl(U_j\Bigl(\frac{\lambda}{2}\Bigl)\Bigr)_{0\le j<q}
  \end{equation*}
  (up to a scalar factor).

  In particular, $0$ is an eigenvalue of $M$ if and only if $q$ is odd.
\end{lemma}
\ifdetails
\begin{proof}Write $v=(v_j)_{0\le j<q}$. By assumption, we have
  \begin{align*}
    v_1 &= \lambda v_0,\\
    v_{j-1}+v_{j+1}&=\lambda v_j,&\qquad 1&\le j\le q-2,\\
    v_{q-2}&=\lambda v_{q-1}
  \end{align*}
  or, equivalently,
  \begin{align*}
    v_1 &= 2\frac{\lambda}{2} v_0,\\
    v_{j+1}&=2\frac{\lambda}{2} v_j-v_{j-1},&\qquad 1&\le j\le q-2,\\
    0&=2\frac{\lambda}{2} v_{q-1} - v_{q-2}.
  \end{align*}

  By induction, we easily see that the first $q-1$ equations are equivalent to
  $v_j=v_0 U_j(\frac{\lambda}{2})$ for $0\le j<q$. The last equation is
  equivalent to $v_0U_q(\lambda/2)=0$. We see that $v_0$ leads to $v=0$, so we
  can exclude this case. Thus $U_q(\lambda/2)=0$ and
  $\lambda/2=\cos(k\pi/(q+1))$ for some $1\le k\le q$.
\end{proof}
\else
\begin{proof}
  See the statement and the proof of~\cite[Proposition~3]{Koninck-Doyon:2009:esthetic-numbers}.
\end{proof}
\fi

\begin{lemma}\label{lemma:sum}
  Let $1\le k\le q$, $\lambda=2\cos(k\pi/(q+1))$ and $v$ be an eigenvector of
  $M$ to $\lambda$. Then
    $\langle \mathbf{1}, v\rangle = 0$ holds if and only if $k$ is even.
\end{lemma}
\begin{proof}\allowdisplaybreaks We write $\varphi\coloneqq k\pi/(q+1)$. By
  Lemma~\ref{lemma:eigenvector-M}
  and~\eqref{eq:U-n-sinus-relation} and a summation similar to the Dirichlet kernel, we have
  \begin{align*}
    \langle \mathbf{1}, v\rangle &=\sum_{0\le j<q}U_j(\cos\varphi)\\
    &=\frac{1}{\sin\varphi}\sum_{0\le j<q}\sin((j+1)\varphi)\\
    &=\frac1{\sin\varphi}\Im \sum_{0\le j<q}\exp(i\varphi)^{j+1}\\
    &=\frac1{\sin\varphi} \Im \Bigl(\exp(i\varphi)\frac{1-\exp(iq\varphi)}{1-\exp(i\varphi)}\Bigr)\\
    &=\frac1{\sin\varphi} \Im \Bigl(\exp\Bigl(\frac{i(q+1)\varphi}{2}\Bigr)\frac{\exp\bigl(-\frac{iq\varphi}{2}\bigr)-\exp\bigl(\frac{iq\varphi}{2}\bigr)}{\exp\bigl(-\frac{i\varphi}{2}\bigr)-\exp\bigl(\frac{i\varphi}{2}\bigr)}\Bigr)\\
    &=\frac{\sin\bigl(\frac{q\varphi}{2}\bigr)}{\sin\varphi\sin\bigl(\frac{\varphi}{2}\bigr)} \Im \exp\Bigl(\frac{i(q+1)\varphi}{2}\Bigr)\\
    &=\frac{\sin\bigl(\frac{q\varphi}{2}\bigr)\sin\bigl(\frac{(q+1)\varphi}{2}\bigr)}{\sin\varphi\sin\bigl(\frac{\varphi}{2}\bigr)}.
  \intertext{Inserting the value of $\varphi$ leads to}
    \langle \mathbf{1}, v\rangle &=\frac{\sin\bigl(\frac{qk\pi}{2(q+1)}\bigr)\sin\bigl(\frac{k\pi}{2}\bigr)}{\sin\bigl(\frac{k\pi}{q+1}\bigr)\sin\bigl(\frac{k\pi}{2(q+1)}\bigr)}.
  \end{align*}
  For $1\le k\le q$, it is clear that $0<k\pi/(q+1)<\pi$ and
  $0<k\pi/(2(q+1))<\pi$, so the denominator of this fraction is non-zero.
  We also claim that $\sin\bigl(\frac{qk\pi}{2(q+1)}\bigr)\neq 0$: Otherwise,
  we have $2(q+1)\mid qk$, hence $q+1 \mid qk$, which implies that $q+1\mid
  k$ because $\gcd(q, q+1)=1$. However, it cannot be that $q+1\mid k$ because $1\le
  k\le q$.

  As a consequence, $\langle \mathbf{1}, v\rangle=0$ if and only
  if $k/2$ is an integer.
\end{proof}

\begin{lemma}
  The characteristic polynomial of $C$ is
  \begin{equation*}
    X\prod_{1\le k\le q}\Bigl(X-2\cos\Bigl(\frac{k\pi}{q+1}\Bigr)\Bigr).
  \end{equation*}
  In particular, all eigenvalues of~$M$ apart from $0$ are eigenvalues of~$C$
  with algebraic multiplicity $1$. If $q$ is even, then $0$ has algebraic multiplicity
  $1$ as an eigenvalue of~$C$; if $q$ is odd, then $0$ has algebraic
  multiplicity $2$ as an eigenvalue of~$C$.
\end{lemma}
\begin{proof}
  The matrix $C$ is a block lower triangular matrix, so the characteristic
  polynomial is the product of the characteristic polynomials of
  the matrices~$M$ and~$0$.

  The statement on the algebraic multiplicities follows from Lemma~\ref{lemma:eigenvector-M}.
\end{proof}

We can summarise our findings on the eigenvectors and eigenvalues of $C$ as follows.

\begin{proposition}\label{proposition:esthetic-eigenvectors}
  Let $v\in\C^{q}$, $w\in\C$, not both $0$, and let $\lambda\in\C$.

  Then $\bigl(\begin{smallmatrix}v\\w\end{smallmatrix}\bigr)\neq 0$ is an eigenvector of
  $C$ to the eigenvalue $\lambda$ if and only if one of the following
  conditions hold:
  \begin{enumerate}
  \item\label{enu:esthetic-eigenvectors:neq0}
    $0\neq \lambda = 2\cos\bigl(\frac{k\pi}{q+1}\bigr)$ for some $1\le k\le
    q$ and $k\neq\frac{q+1}{2}$, $v$ is an eigenvector of $M$ to $\lambda$, and
    $w=0$ if $k$ is even and $\lambda w=\langle \mathbf{1}, v\rangle\neq 0$ if $k$ is
    odd;
  \item $\lambda=0$, $v=0$, $w\neq 0$;
  \item $\lambda=0$, $q\equiv 3\pmod 4$, $v$ is an eigenvector of $M$ and $w=0$.
  \end{enumerate}

  In particular, the eigenvalue~$\lambda=0$ of~$C$ has
  \begin{itemize}
  \item algebraic and geometric multiplicity $2$ if $q\equiv 3\pmod 4$,
  \item algebraic multiplicity $2$ and geometric multiplicity
    $1$ if $q\equiv 1\pmod 4$, and
  \item algebraic and geometric multiplicity $1$
    for even $q$.
  \end{itemize}
\end{proposition}
\begin{proof}
  The vector $\bigl(\begin{smallmatrix}v\\w\end{smallmatrix}\bigr)$
  is an eigenvector if and only if
  \begin{align*}
    Mv&=\lambda v,\\
    \langle \mathbf{1}, v\rangle&=\lambda w.
  \end{align*}
  
  First assume that $\lambda\neq 0$. Then $v=0$ leads to $w=0$,
  contradiction. Therefore, $v$ is an eigenvector of $M$ to the eigenvalue
  $\lambda$ and $\lambda=2\cos\bigl(\frac{k\pi}{q+1}\bigr)$ for some
  $1\le k\le q$ by Lemma~\ref{lemma:eigenvector-M}.
  Then $w=0$ if and only if $k$ is even
  by Lemma~\ref{lemma:sum}.

  Now assume that $\lambda=0$ and $q$ is even. Then $0$ is not an eigenvalue of
  $M$ by Lemma~\ref{lemma:eigenvector-M}. Thus $v=0$ and $w\neq 0$.

  Now, assume that $\lambda=0$ and $q\equiv 3\pmod 4$.
  Then
  $\lambda=2\cos\bigl(\frac{\pi}{2}\bigr)=2\cos\bigl(\frac{\frac{q+1}{2}\pi}{q+1}\bigr)$.
  By Lemma~\ref{lemma:sum}, the eigenvector $v$ of $M$ leads to an eigenvector
  $\bigl(\begin{smallmatrix}v\\0\end{smallmatrix}\bigr)$ of $C$; and there is
  an additional eigenvector
  $\bigl(\begin{smallmatrix}0\\w\end{smallmatrix}\bigr) \neq 0$.

  Finally, assume that $\lambda=0$ and $q\equiv 1\pmod 4$. In this case,
  by Lemma~\ref{lemma:sum}, it cannot be that $v\neq 0$ is an eigenvector of $M$
  because this would lead to $0\neq \langle \mathbf{1}, v\rangle=\lambda w=0$,
  a contradiction. Thus the only eigenvector is
  $\bigl(\begin{smallmatrix}0\\w\end{smallmatrix}\bigr) \neq 0$.
\end{proof}

%%% Local Variables:
%%% mode: latex
%%% TeX-master: "regular-sequences"
%%% End:

%% file: pascal-rhombus.tex
We consider Pascal's rhombus~$\mathfrak{R}$ which is,
for integers~$i\geq0$ and $j$, the array with entries $r_{i,j}$, where
\begin{itemize}
\item $r_{0,j} = 0$ all $j$,
\item $r_{1,0}=1$ and $r_{1,j}=0$ for all $j\neq0$,
\item and
\begin{equation*}
  r_{i,j} = r_{i-1,j-1} + r_{i-1,j} + r_{i-1,j+1} + r_{i-2,j}
\end{equation*}
for $i \geq 1$.
\end{itemize}

\begin{figure}
  \centering
  \includegraphics[width=\linewidth]{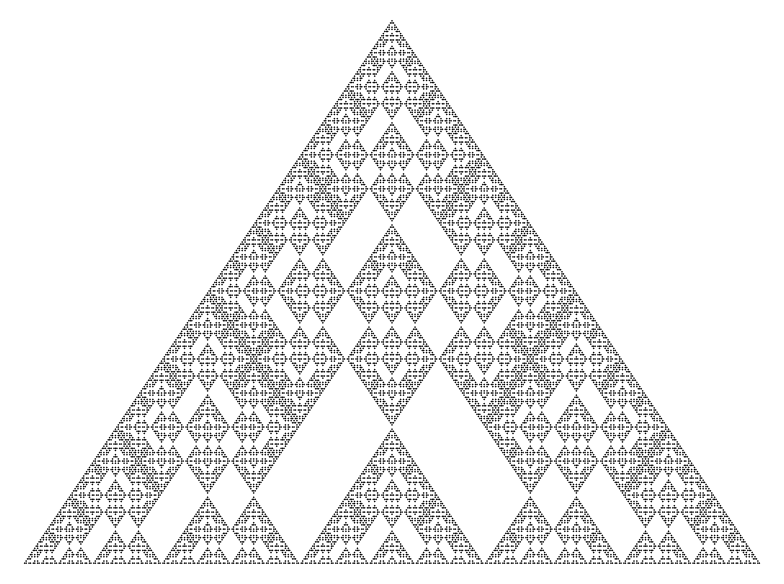}
  \caption{Pascal's rhombus modulo~$2$.}
  \label{fig:pascal-one}
\end{figure}

We are interested in the number of odd entries in the first~$N$ rows
of this
rhombus. In~\cite{Goldwasser-Klostermeyer-Mays-Trapp:1999:Pascal-rhombus}
the authors investigate this quantity for $N$ being a power of~$2$.
We again aim for a more precise analysis and asymptotic description.

So, let $\mathfrak{X}$ be equal to $\mathfrak{R}$ but with
entries taken modulo~$2$; see also Figure~\ref{fig:pascal-one}.
We partition $\mathfrak{X}$ into the four sub-arrays
\begin{itemize}
\item $\mathfrak{E}$ consisting only of the rows and columns of
  $\mathfrak{X}$ with even indices, i.e., the entries~$r_{2i, 2j}$,
\item $\mathfrak{Y}$ consisting only of the rows with odd indices and
  columns with even indices, i.e., the entries~$r_{2i-1, 2j}$,
\item $\mathfrak{Z}$ consisting only of the rows with even indices and
  columns with odd indices, i.e., the entries~$r_{2i, 2j-1}$, and
\item $\mathfrak{N}$ consisting only of the rows and columns with odd
  indices, i.e., the entries~$r_{2i-1, 2j-1}$.
\end{itemize}
Note that $\mathfrak{E} = \mathfrak{X}$ and $\mathfrak{N}=0$;
see~\cite{Goldwasser-Klostermeyer-Mays-Trapp:1999:Pascal-rhombus}.

\subsection{Recurrence Relations and $2$-Regular Sequences}
\label{sec:recurrences}

Let $X(N)$, $Y(N)$ and $Z(N)$ be the number of ones in the first $N$ rows
(starting with row index~$1$)
of $\mathfrak{X}$, $\mathfrak{Y}$ and $\mathfrak{Z}$, respectively.

Goldwasser, Klostermeyer, Mays and
Trapp~\cite[(12)--(14)]{Goldwasser-Klostermeyer-Mays-Trapp:1999:Pascal-rhombus}
get the recurrence relations
\begin{align*}
  X(N) &= X(\floor{\tfrac N2}) + Y(\ceil{\tfrac N2}) + Z(\floor{\tfrac N2}), \\
  Y(N) &= X(\ceil{\tfrac N2}) + X(\floor{\tfrac N2}-1) + Z(\floor{\tfrac N2}) + Z(\ceil{\tfrac N2}-1), \\
  Z(N) &= 2 X(\floor{\tfrac N2}) + 2 Y(\ceil{\tfrac N2})
\end{align*}
for $N\ge2$, and $X(0)=Y(0)=Z(0)=0$, $X(1)=1$, $Y(1)=1$ and $Z(1)=2$
(cf.~\cite[Figures~2 and~3]{Goldwasser-Klostermeyer-Mays-Trapp:1999:Pascal-rhombus}).
Distinguishing between even and odd indices gives
\begin{align*}
  X(2N) &= X(N) + Y(N) + Z(N), \\
  X(2N+1) &= X(N) + Y(N+1) + Z(N), \\
  Y(2N) &= X(N) + X(N-1) + Z(N) + Z(N-1), \\
  Y(2N+1) &= X(N+1) + X(N-1) + 2Z(N), \\
  Z(2N) &= 2X(N) + 2Y(N), \\
  Z(2N+1) &= 2X(N) + 2Y(N+1)
\end{align*}
for all $N\ge1$.
Now we build the backward differences
$x(n) = X(n) - X(n-1)$, $y(n) = Y(n) - Y(n-1)$ and $z(n) = Z(n) - Z(n-1)$.
These $x(n)$, $y(n)$ and $z(n)$ are the number
of ones in the $n$th row of $\mathfrak{X}$, $\mathfrak{Y}$ and
$\mathfrak{Z}$, respectively, and clearly
\begin{equation*}
  X(N) = \sum_{1\leq n \leq N} x(n),
  \qquad
  Y(N) = \sum_{1\leq n \leq N} y(n),
  \qquad
  Z(N) = \sum_{1\leq n \leq N} z(n)
\end{equation*}
holds. We obtain
\begin{subequations}
  \label{eq:rec-pascal-rhombus:main}
  \begin{align}
    x(2n)&=x(n)+z(n), &
    x(2n+1)&=y(n+1), \label{eq:rec-x}\\
    y(2n)&= x(n-1)+z(n), &
    y(2n+1)&=x(n+1) +z(n), \label{eq:rec-y}\\
    z(2n)&= 2x(n), &
    z(2n+1)&=2y(n+1) \label{eq:rec-z}
  \end{align}
\end{subequations}
for $n\ge1$, and $x(0)=y(0)=z(0)=0$, $x(1)=1$, $y(1)=1$ and $z(1)=2$.

Let us write our coefficients as the
vector
\begin{equation}\label{eq:pascal:vec-v}
  v(n) = \bigl(x(n), x(n+1), y(n+1), z(n), z(n+1)\bigr)^\top.
\end{equation}
It turns out that the components included into $v(n)$ are
sufficient for a self-contained linear representation of~$v(n)$.
In particular, it is not necessary to include~$y(n)$.
By using the recurrences~\eqref{eq:rec-pascal-rhombus:main}, we find that
\begin{equation*}
  v(2n) = A_0 v(n)
  \qquad\text{and}\qquad
  v(2n+1) = A_1 v(n)
\end{equation*}
for all\footnote{ Note that $v(0) = A_0 v(0)$ and $v(1) = A_1 v(0)$ are indeed
  true.}
 $n\ge0$ with the matrices
\begin{equation*}
  A_0 =
  \begin{pmatrix}
    1 & 0 & 0 & 1 & 0 \\
    0 & 0 & 1 & 0 & 0 \\
    0 & 1 & 0 & 1 & 0 \\
    2 & 0 & 0 & 0 & 0 \\
    0 & 0 & 2 & 0 & 0
  \end{pmatrix}
  \qquad\text{and}\qquad
  A_1 =
  \begin{pmatrix}
    0 & 0 & 1 & 0 & 0 \\
    0 & 1 & 0 & 0 & 1 \\
    1 & 0 & 0 & 0 & 1 \\
    0 & 0 & 2 & 0 & 0 \\
    0 & 2 & 0 & 0 & 0
  \end{pmatrix},
\end{equation*}
and with $v(0) = (0,1,1,0,2)^\top$.
Therefore, the sequences $x(n)$, $y(n)$ and $z(n)$ are $2$-regular.

\subsection{Full Asymptotics}
\label{sec:asymptotics}

\begin{corollary}\label{corollary:pascal-rhombus:main}
  We have
  \begin{equation}\label{eq:pascal-rhombus:main-asy}
    X(N) = \sum_{1\leq n \leq N} x(n)
    = N^\gamma \f{\Phi}{\fractional{\log_2 N}} + \Oh{N \log_2 N}
  \end{equation}
  with $\gamma = \log_2 \bigl(3+\sqrt{17}\,\bigr)-1 = 1.83250638358045\ldots$ and
  a $1$-periodic function $\Phi$ which is Hölder continuous with
  any exponent smaller than $\gamma-1$.

  Moreover, we can effectively compute the Fourier coefficients
  of~$\Phi$ (as explained in Part~\ref{part:numerical}).
\end{corollary}

We get analogous results for the sequences~$Y(N)$ and $Z(N)$ (each with
its own periodic function~$\Phi$, but the same exponent $\gamma$).
The fluctuation~$\Phi$ of $X(N)$ is visualized in Figure~\ref{fig:fluct-a} and its
first few Fourier coefficients are shown in Table~\ref{table:pascal-rhombus:fourier}.

\begin{figure}
  \centering
  \includegraphics{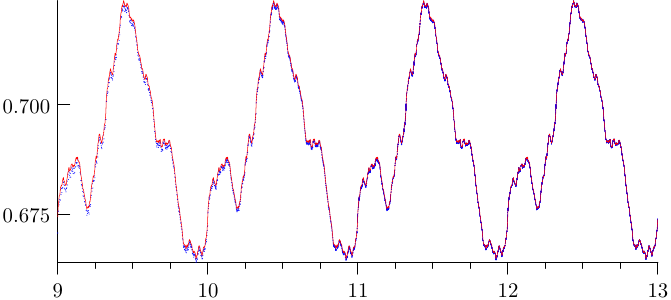}
  \caption{Fluctuation in the main term of the asymptotic expansion of $X(N)$.
    The figure shows $\f{\Phi}{u}$ (red) approximated by
    its trigonometric polynomial of degree~$1999$ as well as
    $X(2^u) / 2^{u\gamma}$ (blue).}
  \label{fig:fluct-a}
\end{figure}
\begin{table}
  \centering
  \begin{equation*}\footnotesize
    \begin{array}{r|l}
\multicolumn{1}{c|}{\ell} &
\multicolumn{1}{c}{\varphi_\ell} \\
\hline
0 & \phantom{-}0.6911615112341912755021246 \\
1 & -0.01079216311240407872950510 - 0.0023421761940286789685827i \\
2 & \phantom{-}0.00279378637350495172116712 - 0.00066736128659728911347756i \\
3 & -0.00020078258323645842522640 - 0.0031973663977645462669373i \\
4 & \phantom{-}0.00024944678921746747281338 - 0.0005912995467076061497650i \\
5 & -0.0003886698612765803447578 + 0.00006723866319930148568431i \\
6 & -0.0006223575988893574655258 + 0.00043217220614939859781542i \\
7 & \phantom{-}0.00023034317364181383130476 - 0.00058663168772856091427688i \\
8 & \phantom{-}0.0005339060804798716172593 - 0.0002119380802590974909465i \\
9 & \phantom{-}0.0000678898389770175928529 - 0.00038307823285486235280185i \\
10 & -0.00019981745997355255061991 - 0.00031394569060142799808175i \\
    \end{array}
  \end{equation*}
  \caption{Fourier coefficients of~$\Phi$
    (Corollary~\ref{corollary:pascal-rhombus:main}). All stated digits are
  correct; see also Part~\ref{part:numerical}.}
\label{table:pascal-rhombus:fourier}
\end{table}

\subsection{Proof of the Asymptotic Result}

At this point, we only prove~\eqref{eq:pascal-rhombus:main-asy} of
Corollary~\ref{corollary:pascal-rhombus:main}. We deal with the
Fourier coefficients in Section~\ref{sec:fourier}. As in the
introductory example of the binary sum-of-digits functions
(Example~\ref{example:binary-sum-of-digits}), we could get Fourier
coefficients by Theorem~\ref{theorem:simple} and the $2$-linear
representation of Section~\ref{sec:recurrences} directly. However,
the information in the vector~$v(n)$ (see \eqref{eq:pascal:vec-v})
is redundant with respect to the asymptotic main term as it contains
$x(n)$ and $z(n)$ as well as $x(n+1)$ and $z(n+1)$; both pairs are
asymptotically equal in the sense of~\eqref{eq:pascal-rhombus:main-asy}.
Therefore, we head for an only $3$-dimensional functional system of equations
for our Dirichlet series of $x(n)$, $y(n)$ and $z(n)$
(instead of a $5$-dimensional system).

\begin{proof}[Proof of~\eqref{eq:pascal-rhombus:main-asy}]
    We use Theorem~\ref{theorem:simple}.

    \proofparagraph{Joint Spectral Radius}
    First we compute the joint spectral radius $\rho$ of
    $A_0$ and $A_1$. Both matrices have a maximum absolute
    row sum equal to $2$, thus $\rho\leq 2$, and both
    matrices have~$2$ as an eigenvalue. Therefore we obtain
    $\rho=2$. Moreover, the finiteness property of the linear
    representation is satisfied by considering only products with exactly one
    matrix factor $A_0$ or $A_1$.

    Thus, we have $R=\rho=2$.

    \proofparagraph{Eigenvalues}
    Next, we compute the spectrum~$\sigma(C)$ of $C=A_0+A_1$. The
    matrix $C$ has the eigenvalues~$\lambda_1=\bigl(3+\sqrt{17}\,\bigr)/2=3.5615528128088\ldots$,
    $\lambda_2=2$, $\lambda_3=-2$, $\lambda_4=-1$ and
    $\lambda_5=\bigl(3-\sqrt{17}\,\bigr)/2=-0.5615528128088\ldots$ (each with multiplicity one).
    Note that $\lambda_1$ and $\lambda_5$ are the zeros of the
    polynomial~$U^2-3U-2$.

    \proofparagraph{Asymptotics}
    By using Theorem~\ref{theorem:simple}, we obtain an
    asymptotic formula for $X(N-1)$. Shifting from $N-1$ to $N$ does not
    change this asymptotic formula, as this shift is absorbed by the
    error term $\Oh{N \log_2 N}$.
\end{proof}

\subsection{Dirichlet Series and Meromorphic Continuation}
\label{sec:meromorphic}

In the lemma below,
we provide the functional equation~\eqref{eq:pascal:functional-equation}
as a system of three equations. This is in contrast
to the generic functional equation
provided by Theorem~\ref{theorem:Dirichlet-series} which is a system of five
equations.

Let $n_0\ge2$ be an integer and define
\begin{align*}
  \f{\calX_{n_0}}{s} &= \sum_{n\geq n_0} \frac{x(n)}{n^s}, &
  \f{\calY_{n_0}}{s} &= \sum_{n\geq n_0} \frac{y(n)}{n^s}, &
  \f{\calZ_{n_0}}{s} &= \sum_{n\geq n_0} \frac{z(n)}{n^s}.
\end{align*}

\begin{lemma}\label{lemma:meromorphic}
  Set
  \begin{equation*}
    M = I -
    \begin{pmatrix}
      2^{-s} & 2^{-s} & 2^{-s} \\
      2^{1-s} & 0 & 2^{1-s} \\
      2^{1-s} & 2^{1-s} & 0 \\
    \end{pmatrix}.
  \end{equation*}
  Then
  \begin{equation}\label{eq:pascal:functional-equation}
    M
    \begin{pmatrix}
      \f{\calX_{n_0}}{s} \\ \f{\calY_{n_0}}{s} \\ \f{\calZ_{n_0}}{s}
    \end{pmatrix}
    =
    \begin{pmatrix}
      \f{\calJ_{n_0}}{s} \\ \f{\calK_{n_0}}{s} \\ \f{\calL_{n_0}}{s}
    \end{pmatrix}\!,
  \end{equation}
  where
  \begin{align*}
    \f{\calJ_{n_0}}{s} &= 2^{-s} \f{\Sigma}{s, -\tfrac12, \calY_{n_0}}
    + \calI_{\calJ_{n_0}}(s), \\
    &\;\calI_{\calJ_{n_0}}(s) = - \frac{y(n_0)}{(2n_0-1)^s}
    + \sum_{n_0\leq n<2n_0} \frac{x(n)}{n^s}, \\
    \f{\calK_{n_0}}{s} &=
    2^{-s} \f{\Sigma}{s, 1, \calX_{n_0}} + 2^{-s} \f{\Sigma}{s, -\tfrac12, \calX_{n_0}}
    + 2^{-s} \f{\Sigma}{s, \tfrac12, \calZ_{n_0}}
    + \calI_{\calK_{n_0}}(s), \\
    &\;\calI_{\calK_{n_0}}(s) = \frac{x(n_0-1)}{(2n_0)^s} - \frac{x(n_0)}{(2n_0-1)^s}
    + \sum_{n_0\leq n<2n_0} \frac{y(n)}{n^s}, \\
    \f{\calL_{n_0}}{s} &= 2^{1-s} \f{\Sigma}{s, -\tfrac12, \calY_{n_0}}
    + \calI_{\calL_{n_0}}(s), \\
    &\;\calI_{\calL_{n_0}}(s) = - \frac{2 y(n_0)}{(2n_0-1)^s}
    + \sum_{n_0\leq n<2n_0} \frac{z(n)}{n^s},
  \end{align*}
  with the notion of $\Sigma$ as in Lemma~\ref{lemma:shifted-Dirichlet},
  provides meromorphic continuations
  of the Dirichlet series~$\f{\calX_{n_0}}{s}$, $\f{\calY_{n_0}}{s}$,
  and $\f{\calZ_{n_0}}{s}$ for $\Re s > \gamma_0=1$ with the only possible
  poles at $\gamma + \chi_\ell$ for $\ell\in\Z$,
  all of which are simple poles.
\end{lemma}

\begin{proof}
  We split the proof into several steps.

  \proofparagraph{Functional Equation}
  \ifdetails
  From \eqref{eq:rec-x} we obtain
  \begin{equation*}
    \f{\calX_{n_0}}{s} = \sum_{n_0\leq n<2n_0} \frac{x(n)}{n^s}
    + \sum_{n\geq n_0} \frac{x(n)}{(2n)^s}
    + \sum_{n\geq n_0} \frac{z(n)}{(2n)^s}
    + \sum_{n\geq n_0} \frac{y(n+1)}{(2n+1)^s}
  \end{equation*}
  The second and third summands become $2^{-s} \f{\calX_{n_0}}{s}$ and $2^{-s} \f{\calZ_{n_0}}{s}$,
  respectively, and we are left to rewrite the fourth summand. By
  using Lemma~\ref{lemma:shifted-Dirichlet} with $\beta=-1/2$ we
  get
  \begin{align*}
    \sum_{n\geq n_0} \frac{y(n+1)}{(2n+1)^s}
    &= 2^{-s} \sum_{n\geq n_0} \frac{y(n)}{(n-\frac12)^s}
    - \frac{y(n_0)}{(2n_0-1)^s} \\
    &= 2^{-s} \f{\calY_{n_0}}{s}
    + 2^{-s} \f{\Sigma}{s, -\tfrac12, \calY_{n_0}} - \frac{y(n_0)}{(2n_0-1)^s}.
  \end{align*}
  The first row of \eqref{eq:pascal:functional-equation} now follows.

  Similarly, from~\eqref{eq:rec-y}\else From~\eqref{eq:rec-y}\fi{} we obtain
  \begin{equation}\label{eq:func:Ys}
    \begin{split}
    \f{\calY_{n_0}}{s} &= \sum_{n_0\leq n<2n_0} \frac{y(n)}{n^s}
    + \sum_{n\geq n_0} \frac{x(n-1)}{(2n)^s}
    + \sum_{n\geq n_0} \frac{z(n)}{(2n)^s} \\
    &\phantom{=}\hphantom{0}
    + \sum_{n\geq n_0} \frac{x(n+1)}{(2n+1)^s}
    + \sum_{n\geq n_0} \frac{z(n)}{(2n+1)^s} \\
    &= \sum_{n_0\leq n<2n_0} \frac{y(n)}{n^s}
    + 2^{-s} \sum_{n\geq n_0} \frac{x(n)}{(n+1)^s} + \frac{x(n_0-1)}{(2n_0)^s}
    + 2^{-s} \sum_{n\geq n_0} \frac{z(n)}{n^s} \\
    &\phantom{=}\hphantom{0}
    + 2^{-s} \sum_{n\geq n_0} \frac{x(n)}{(n-\frac12)^s} - \frac{x(n_0)}{(2n_0-1)^s}
    + 2^{-s} \sum_{n\geq n_0} \frac{z(n)}{(n+\frac12)^s}\\
    &= \sum_{n_0\leq n<2n_0} \frac{y(n)}{n^s}
    + 2^{-s} (\f{\calX_{n_0}}{s} + \Sigma(s, 1, \calX_{n_0})) + \frac{x(n_0-1)}{(2n_0)^s}
    + 2^{-s} \calZ_{n_0}(s) \\
    &\phantom{=}\hphantom{0}
    + 2^{-s} \bigl(\f{\calX_{n_0}}{s} + \Sigma(s, -\tfrac{1}{2}, \calX_{n_0})\bigr) - \frac{x(n_0)}{(2n_0-1)^s} \\
    &\phantom{=}\hphantom{0}
    + 2^{-s} \bigl(\f{\calZ_{n_0}}{s} + \Sigma(s, \tfrac{1}{2}, \calZ_{n_0})\bigr).
    \end{split}
  \end{equation}
  The second row of \eqref{eq:pascal:functional-equation}
  follows.
  \ifdetails

  Similarly, \eqref{eq:rec-z} yields
  \begin{align*}
    \f{\calZ_{n_0}}{s} &= \sum_{n_0\leq n<2n_0} \frac{z(n)}{n^s}
    + 2 \sum_{n\geq n_0} \frac{x(n)}{(2n)^s}
    + 2 \sum_{n\geq n_0} \frac{y(n+1)}{(2n+1)^s} \\
    &= \sum_{n_0\leq n<2n_0} \frac{z(n)}{n^s}
    + 2^{1-s} \sum_{n\geq n_0} \frac{x(n)}{n^s}
    + 2^{1-s} \sum_{n\geq n_0} \frac{y(n)}{(n-\tfrac12)^s} - \frac{2 y(n_0)}{(2n_0-1)^s},
  \end{align*}
  and the third row of \eqref{eq:pascal:functional-equation} follows.
  \else
  Similarly,~\eqref{eq:rec-x} and~\eqref{eq:rec-z} yield the first and third
  rows of~\eqref{eq:pascal:functional-equation}, respectively.
  \fi

  \proofparagraph{Determinant and Zeros}
  The determinant of $M$ is
  \begin{equation*}
    \f{\Delta}{s} = \det M
    = 2^{-3s} \bigl(2^{2s} - 3\cdot 2^s - 2\bigr) \bigl(2^s + 2\bigr).
  \end{equation*}
  It is an entire function.

  All zeros of $\Delta$ are simple zeros.
  In particular, solving $\f{\Delta}{s} = 0$ gives $2^s = 3/2 \pm \sqrt{17}/2$ (the two zeros of $U^2-3U-U$) and $2^s = -2$.
  A solution $\f{\Delta}{s_0} = 0$
  implies that $s_0 + 2\pi i \ell/\log 2$ with $\ell\in\Z$ satisfies
  the same equation as well.

  Moreover, set $\gamma=\log_2 \bigl(3+\sqrt{17}\,\bigr) - 1 = 1.8325063835804\dots$.
  Then the only zeros with $\Re s > \gamma_0=1$ are at
  $\gamma + \chi_\ell$ with $\chi_\ell = 2\pi i \ell / \log 2$ for $\ell\in\Z$.

  It is no surprise that the $\gamma$ of this lemma and the $\gamma$
  in the proof of Corollary~\ref{corollary:pascal-rhombus:main} which comes from the
  $2$-linear representation of Section~\ref{sec:recurrences} coincide.

  \proofparagraph{Meromorphic Continuation}
  Let
  $\calD_{n_0}\in\set{\calX_{n_0},\calY_{n_0},\calZ_{n_0}}$.
  The Dirichlet series~$\f{\calD_{n_0}}{s}$ is
  analytic for $\Re s > 2 = \log_2 \rho + 1$ with $\rho=2$ being the
  joint spectral radius by Theorem~\ref{theorem:Dirichlet-series}.
  We use the
  functional equation~\eqref{eq:pascal:functional-equation} which
  provides the continuation, as we write $\f{\calD_{n_0}}{s}$ in terms of
  $\f{\calJ_{n_0}}{s}$, $\f{\calK_{n_0}}{s}$ and $\f{\calL_{n_0}}{s}$.
  By Lemma~\ref{lemma:shifted-Dirichlet},
  these three functions are analytic for $\Re s > 1$.

  The zeros (all are simple zeros)
  of the denominator~$\f{\Delta}{s}$ are the only possibilities
  for the poles of $\f{\calD_{n_0}}{s}$ for $\Re s > 1$.
\end{proof}

\subsection{Fourier Coefficients}
\label{sec:fourier}

We are now ready to prove the rest of Corollary~\ref{corollary:pascal-rhombus:main}.

\begin{proof}[Proof of Corollary~\ref{corollary:pascal-rhombus:main}]
  We verify that we can apply Theorem~\ref{theorem:use-Mellin--Perron}.

  The steps of this proof in Section~\ref{sec:asymptotics} provided us
  already with an asymptotic
  expansion~\eqref{eq:pascal-rhombus:main-asy}. Lemma~\ref{lemma:meromorphic}
  gives us the meromorphic function for $\Re s>\gamma_0=1$ which comes from
  the Dirichlet series
  $\bigl(\f{\calX_{n_0}}{s}, \f{\calY_{n_0}}{s}, \f{\calZ_{n_0}}{s}\bigr)^\top$\!.
  It can only have poles (all simple) at $s=\gamma + \chi_\ell$ for $\ell\in\Z$ and
  satisfies the assumptions in
  Theorem~\ref{theorem:use-Mellin--Perron} by
  Theorem~\ref{theorem:Dirichlet-series} and
  Remark~\ref{remark:Dirichlet-series:bound}.

  Therefore a computation of the Fourier coefficients via computing
  residues (see \eqref{eq:Fourier-coefficient:simple-as-residue}) is possible by
  Theorem~\ref{theorem:use-Mellin--Perron}, and this residue may be
  computed from~\eqref{eq:pascal:functional-equation} via Cramer's
  rule.
\end{proof}

We refer to Part~\ref{part:numerical} for details on the actual
computation of the Fourier coefficients.

%%% Local Variables:
%%% mode: latex
%%% TeX-master: "regular-sequences"
%%% End:

%% file: estimates.tex
\part{Computational Aspects}\label{part:numerical}

The basic idea
for computing the Fourier coefficients is to use the functional equation
in Theorem~\ref{theorem:Dirichlet-series}.
This part describes in detail how this is done.
We basically follow an approach found in Grabner and Hwang~\cite{Grabner-Hwang:2005:digit}
and Grabner and Heuberger~\cite{Grabner-Heuberger:2006:Number-Optimal}, but
provide error bounds.

An actual implementation is also available;
SageMath~\cite{SageMath:2018:8.3} code can be found at
\url{https://gitlab.com/dakrenn/regular-sequence-fluctuations}\,.
We use the Arb library~\cite{Johansson:2017:arb} (more precisely, its
SageMath bindings) for ball arithmetic
which keeps track of rounding errors such that we can be sure about the precision and accuracy of our results.

We use the results of this part to compute Fourier coefficients for
our examples, in particular for esthetic numbers
(Section~\ref{sec:esthetic-numbers}) and Pascal's rhombus
(Section~\ref{sec:pascal}).

\section{Strategy for Computing the Fourier Coefficients}
\label{section:strategy-for-computing}

The computation of the Fourier coefficients relies on the evaluation
of Dirichlet series at certain points~$s=s_0$. It turns out to be
numerically preferable to split up the sum as
\begin{equation*}
  \calF_{1}(s_0) = \sum_{1 \le n < n_0} n^{-s_0} f(n) + \calF_{n_0}(s_0)
\end{equation*}
for some suitable~$n_0$ (see Section~\ref{section:choice-parameters}),
compute the sum of the first $n_0-1$ summands directly and evaluate
$\calF_{n_0}(s_0)$ as it is described in the following.

For actually computing the Fourier coefficients, we use a formulation in
terms of a residue; for instance,
see~\eqref{eq:Fourier-coefficient:simple-as-residue} where this is
formulated explicitly in the set-up of Theorem~\ref{theorem:simple}.
As said, we will make use of the functional
equation~\eqref{eq:analytic-continuation} for the matrix-valued
Dirichlet series~$\calF_{n_0}(s)$ with its right-hand side, the
matrix-valued Dirichlet series~$\calG_{n_0}(s)$.

Let us make this explicit for a simple eigenvalue $\lambda\neq 1$ of~$C$ and
a corresponding eigenvector~$w$. Then
$w (I - q^{-s} C) = w (1 - q^{-s}\lambda)$
and~\eqref{eq:analytic-continuation} can be rewritten as
\begin{equation*}
  w\, \calF_{1}(s) = \frac{1}{1 - q^{-s}\lambda} w\, \calG_{1}(s).
\end{equation*}
Thus, $w\, \calF_{1}(s)$ has simple poles at~$s=\log_q\lambda+\chi_\ell$
for all $\ell\in\Z$, where $\chi_\ell=\frac{2\ell\pi i}{\log q}$.
By~\eqref{eq:Fourier:F-s-principal-part} and~\eqref{eq:Fourier:fluctuation-as-Fourier-series} of Theorem~\ref{theorem:use-Mellin--Perron}
(with $\gamma=\log_q\lambda$ and $m=1$), the $\ell$th
Fourier coefficient is given by the residue
\begin{equation*}
  \Res[\Big]{\frac{w\, \calF_{1}(s)}{s}}{s=\log_q \lambda+\chi_\ell}
  = w\, \calG_{1}(\log_q\lambda + \chi_\ell) \frac{1}{(\log q)(\log_q\lambda + \chi_\ell)}.
\end{equation*}
Note that $\log q$ is the
derivative of $1 - q^{-s}\lambda$ with respect to~$s$
evaluated at the pole~$s=\log_q\lambda$.

By~\eqref{eq:Dirichlet-recursion}, $\calG_{n_0}(\log_q\lambda+\chi_\ell)$ is expressed in
terms of an infinite sum containing $\calF_{n_0}(\log_q\lambda+\chi_\ell+k)$ for $k\ge1$.
We truncate this sum and bound the error; this is the aim of
Section~\ref{section:bounding-error} and in particular
Lemma~\ref{lemma:approximation-error}.
We can iterate the above idea for the
shifted Dirichlet series~$\calF_{n_0}(\log_q\lambda+\chi_\ell+k)$
which leads to a recursive evaluation scheme.
Note that once we have computed $\calG_{n_0}(\log_q\lambda +\chi_\ell+k)$,
we get $\calF_{n_0}(\log_q \lambda +\chi_\ell+k)$ by solving a
system of linear equations.

\section{Details on the Numerical Computation}
\label{section:computation-details}

\subsection{Bounding the Error}
\label{section:bounding-error}

We need to estimate the approximation error which arises if the infinite sum
over $k\ge 1$ in~\eqref{eq:Dirichlet-recursion} is replaced by a finite sum.
It is clear that for large $\Re s$ and $n_0$, the value $\calF_{n_0}(s)$ will
approximately be of the size of its first summand~$n_0^{-s} f(n_0)$. In view
of $\norm{f(n_0)}=\Oh{\rho^{\log_q n_0}}$, this will be rather small. We give a
precise estimate in a first lemma.

\begin{lemma}\label{lemma:Dirichlet-upper-bound}
  Let $n_0> 1$ and let $M\coloneqq \max_{0\le r<q} \norm{A_r}$.
  For $\Re s>\log_q M + 1$, we have
  \begin{equation*}
    \sum_{n\ge n_0}\frac{\norm{f(n)}}{n^{\Re s}}\le \frac{M}{(\Re
      s-\log_q M -1)(n_0-1)^{\Re s-\log_q M -1}}.
  \end{equation*}
\end{lemma}
\begin{proof}
  By definition of $M$, we have $\norm{f(n)}\le M^{1+\log_q n}=M n^{\log_q
    M}$. Therefore, we have
  \begin{align*}
    \sum_{n\ge n_0}\frac{\norm{f(n)}}{n^{\Re s}}&\le M\sum_{n\ge n_0}\frac1{n^{\Re s - \log_q M}}\le
    M\int_{n_0-1}^\infty \frac{\dd n}{n^{\Re s - \log_q M}}\\&=\frac{M}{(\Re
      s-\log_q M -1)(n_0-1)^{\Re s-\log_q M -1}}
  \end{align*}
  where we interpret the sum as a lower Riemann sum of the integral.
\end{proof}

We now give a bound for the approximation error in~\eqref{eq:Dirichlet-recursion}.

\begin{lemma}\label{lemma:approximation-error}
  Let $n_0>1$ and $M$ as in Lemma~\ref{lemma:Dirichlet-upper-bound}. Let $K\ge
  1$ and $s\in\C$ be such that $\Re s+K>\max(\log_q M + 1, 0)$.

  Then
  \begin{multline*}
    \norm[\bigg]{\calG_{n_0}(s) - \sum_{n_0\le n<qn_0} n^{-s}f(n) - q^{-s}\sum_{0\le
        r<q}A_r\sum_{1\le k<K}\binom{-s}{k}\Bigl(\frac{r}{q}\Bigr)^k
      \calF_{n_0}(s+k)} \\
    \le q^{-\Re s}\abs[\Big]{\binom{-s}{K}} \frac{M}{(\Re
      s+K-\log_q M -1)(n_0-1)^{\Re s+K-\log_q M -1}}\sum_{0\le
      r<q}\norm{A_r}\Bigl(\frac{r}{q}\Bigr)^K.
  \end{multline*}
\end{lemma}
\begin{proof}
  We set
  \begin{equation*}
    D\coloneqq \calG_{n_0}(s) - \sum_{n_0\le n<qn_0} n^{-s}f(n) - q^{-s}\sum_{0\le
        r<q}A_r\sum_{1\le k<K}\binom{-s}{k}\Bigl(\frac{r}{q}\Bigr)^k
      \calF_{n_0}(s+k)
    \end{equation*}
  and need to estimate $\norm{D}$.

  By definition of $\calG_{n_0}(s)$, we have
  \begin{align*}
    \calG_{n_0}(s) &= (1-q^{-s}C)\calF_{n_0}(s)\\
    &=\sum_{n_0\le n<qn_0} n^{-s}f(n) + \calF_{qn_0}(s) -
      q^{-s}C\calF_{n_0}(s)\\
    &=\sum_{n_0\le n<qn_0} n^{-s}f(n) +\sum_{0\le r<q}\sum_{n\ge n_0}\frac{A_r
      f(n)}{(qn+r)^s}- q^{-s}C\calF_{n_0}(s)\\
      &=\sum_{n_0\le n<qn_0} n^{-s}f(n) +q^{-s}\sum_{0\le r<q}A_r \sum_{n\ge n_0}\frac{f(n)}{n^s}\Bigl(\Bigl(1+\frac{r}{qn}\Bigr)^{-s}- 1\Bigr).
  \end{align*}

  Thus we have
  \begin{equation*}
    D = q^{-s}\sum_{0\le r<q}A_r \sum_{n\ge
      n_0}\frac{f(n)}{n^s}\biggl(\Bigl(1+\frac{r}{qn}\Bigr)^{-s}- \sum_{0\le k<K}\binom{-s}{k}\Bigl(\frac{r}{qn}\Bigr)^k\biggr).
  \end{equation*}

  For $0\le x<1$, Taylor's theorem (or induction on $K\ge 1$ using integration
  by parts) implies that
  \begin{equation*}
    (1+x)^{-s}-\sum_{0\le k<K}\binom{-s}{k}x^k = K\int_{0}^x
    \binom{-s}{K}(1+t)^{-s-K}(x-t)^{K-1}\,\dd t.
  \end{equation*}
  For $0\le t\le x<1$, we can bound $\abs{(1+t)^{-s-K}}$ from above by $1$ since we have assumed that $\Re s + K>0$. Thus
  \begin{equation*}
    \abs[\bigg]{(1+x)^{-s}-\sum_{0\le k<K}\binom{-s}{k}x^k} \le K\abs[\Big]{\binom{-s}{K}} \int_{0}^x
    (x-t)^{K-1}\,\dd t = \abs[\Big]{\binom{-s}{K}}x^K.
  \end{equation*}
  Thus we obtain the bound
  \begin{equation*}
    \norm{D} \le q^{-\Re s}\abs[\Big]{\binom{-s}{K}}\sum_{0\le
      r<q}\norm{A_r}\Bigl(\frac{r}{q}\Bigr)^K\sum_{n\ge
      n_0}\frac{\norm{f(n)}}{n^{\Re \sigma+K}}.
  \end{equation*}
  Bounding the remaining Dirichlet series by Lemma~\ref{lemma:Dirichlet-upper-bound} yields the result.
\end{proof}

\subsection{Choices of Parameters}
\label{section:choice-parameters}

As mentioned at the beginning of this part, we choose the Arb
library~\cite{Johansson:2017:arb} for reliable numerical ball
arithmetic.
In our examples (esthetic numbers in
Section~\ref{sec:esthetic-numbers} and Pascal's rhombus in
Section~\ref{sec:pascal}),
we choose $n_0=1024$ and recursively compute
$\calF_{n_0}(\log_q\lambda + \chi_\ell+k)$ for $k\ge 1$
by~\eqref{eq:Dirichlet-recursion}. In each step, we keep adding summands for
$k\ge 1$ until the bound of the approximation error in
Lemma~\ref{lemma:approximation-error} is smaller than
the smallest increment which can still be represented with the chosen number of
bits. For plotting the graphs, we simply took machine precision; for the larger number
of significant digits in Table~\ref{table:pascal-rhombus:fourier}, we used 128
bits precision.

\section{Non-vanishing Coefficients}
\label{section:non-vanishing}

Using reliable numerical arithmetic for the computations (see above) yields small
balls in which the true value of the Fourier coefficients
is.
If such a ball does not contain zero, we know that the Fourier
coefficient does not vanish. If the ball contains zero, however, we
cannot decide whether the Fourier coefficient vanishes. We can only
repeat the computation with higher precision and hope that this will
lead to a decision that the coefficient does not vanish,
or we can try to find a direct argument why the
Fourier coefficient does indeed vanish, for instance using the final
statement of
Theorem~\ref{theorem:contribution-of-eigenspace}~(\ref{item:large-eigenvalue}).

Vanishing Fourier coefficients appear in our introductory
Example~\ref{example:binary-sum-of-digits}: In its continuation
(Example~\ref{example:binary-sum-of-digits:cont}) an alternative
approach is used to compute these coefficients explicitly
symbolically. In this way a decision for them being zero is possible.
The same is true for the example of transducers in Section~\ref{sec:transducer}.

It should also be noted that in the analysis of esthetic numbers
(example in Section~\ref{sec:esthetic-numbers}) we could have modelled
the problem by a complete transducer (by just introducing a sink) and then
applied the results of Section~\ref{sec:transducer}. This would have led to an
asymptotic expansion where the fluctuations of the main term (corresponding to the eigenvalue $q$) would in fact have vanished, but an argument would have been needed.
So we chose a different approach in Section~\ref{sec:esthetic-numbers} to avoid
this problem. There the eigenvalue~$q$ does no longer occur. This
implies that the fluctuations for $q$ of the transducer approach
vanish. Note also that half of the remaining fluctuations still
turn out to vanish:
this is shown in the proof of
Corollary~\ref{corollary:esthetic:asy}.

%%% Local Variables:
%%% mode: latex
%%% TeX-master: "regular-sequences"
%%% End: